\documentclass[11pt]{amsart}
\usepackage{graphicx}
\usepackage{amsmath,amssymb,xy,array}
\usepackage[T1]{fontenc}
\usepackage{amsfonts}
\usepackage{mathabx}
\usepackage{amsmath}
\usepackage{amssymb}
\usepackage{amsthm}
\usepackage{graphicx}
\usepackage{enumitem}
\usepackage{xcolor}
\usepackage[utf8]{inputenc}
\usepackage{hyperref}
\usepackage{mathrsfs}
\usepackage[toc,page]{appendix}
\usepackage{fancyhdr}
\usepackage{tikz}
\usepackage{tikz-cd}
\usepackage{longtable}
\usepackage{geometry}
\usepackage{textcomp}
\usetikzlibrary{trees}
\usetikzlibrary{arrows}
\usepackage{multirow}
\usepackage{youngtab}
\usepackage{mathdots}

\definecolor{yqyqyq}{rgb}{0.5019607843137255,0.5019607843137255,0.5019607843137255}\definecolor{uuuuuu}{rgb}{0.26666666666666666,0.26666666666666666,0.26666666666666666}
\definecolor{uququq}{rgb}{0.25098039215686274,0.25098039215686274,0.25098039215686274}
\definecolor{wwwwww}{rgb}{0.4,0.4,0.4}
\definecolor{uuuuuu}{rgb}{0.26666666666666666,0.26666666666666666,0.26666666666666666}

\setlist[itemize]{leftmargin=6mm}

\renewcommand{\P}{\mathbb P}

\DeclareMathOperator{\codim}{codim}

\newcommand{\Aut}{\operatorname{Aut}}
\newcommand{\PsAut}{\operatorname{PsAut}}

\DeclareMathOperator{\Cl}{Cl}

\DeclareMathOperator{\Mor}{Mor}
\DeclareMathOperator{\NE}{NE}

\DeclareMathOperator{\lin}{lin}

\DeclareMathOperator{\mult}{mult}

\DeclareMathOperator{\Hom}{Hom}

\DeclareMathOperator{\Exc}{Exc}

\DeclareMathOperator{\Proj}{Proj}

\DeclareMathOperator{\Eff}{Eff}
\DeclareMathOperator{\Nef}{Nef}
\DeclareMathOperator{\Mov}{Mov}
\DeclareMathOperator{\Pic}{Pic}
\DeclareMathOperator{\rank}{rank}

\renewcommand{\sec}{\mathbb{S}ec}

\DeclareMathOperator{\Sym}{Sym}
\DeclareMathOperator{\Cox}{Cox}

\DeclareMathOperator{\mov}{mov}

\newcommand{\f}{\varphi}

\renewcommand{\P}{\mathbb{P}}

\newcommand\scalemath[2]{\scalebox{#1}{\mbox{\ensuremath{\displaystyle #2}}}}

\newtheorem{thm}{Theorem}[section]

\newtheorem{Question}[thm]{Question}
\newtheorem{Lemma}[thm]{Lemma}
\newtheorem{Proposition}[thm]{Proposition}

\newtheorem{Corollary}[thm]{Corollary}

\theoremstyle{definition}

\newtheorem{Definition}[thm]{Definition}
\newtheorem{Remark}[thm]{Remark}

\newtheorem{Example}[thm]{Example}
\newtheorem{Notation}[thm]{Notation}

\newtheorem{Construction}[thm]{Construction}
\newtheorem{Script}[thm]{Script}

\hypersetup{pdfpagemode=UseNone}
\hypersetup{pdfstartview=FitH}

\begin{document}
\title[\resizebox{6.1in}{!}{On the birational geometry of spaces of complete forms I: collineations and quadrics}]{On the birational geometry of spaces of complete forms I: collineations and quadrics}

\author[Alex Massarenti]{Alex Massarenti}
\address{\sc Alex Massarenti\\
Universidade Federal Fluminense\\
Campus Gragoat\'a, Rua Alexandre Moura 8 - S\~ao Domingos\\
24210-200, Niter\'oi, Rio de Janeiro\\ Brazil}
\email{alexmassarenti@id.uff.br}

\date{\today}

\begin{abstract}
Moduli spaces of complete collineations are wonderful compactifications of spaces of linear maps of maximal rank between two fixed vector spaces. We investigate the birational geometry of moduli spaces of complete collineations and quadrics from the point of view of Mori theory. We compute their effective, nef and movable cones, the generators of their Cox rings, and their groups of pseudo-automorphisms. Furthermore, we give a complete description of both the Mori chamber and stable base locus decompositions of the effective cone of the space of complete collineations of the 3-dimensional projective space.
\end{abstract}

\maketitle 

\setcounter{tocdepth}{1}

\tableofcontents

\section{Introduction}
\textit{Moduli spaces of complete collineations} are compactifications of spaces of linear maps of maximal rank between two vector spaces where the added boundary divisor is simple normal crossing. These spaces were introduced and studied, along with their symmetric and skew-symmetric counterparts called spaces of \textit{complete quadrics} and \textit{complete skew-fomrs}, by many leading algebraic geometers of the 19th-century such as M. Chasles \cite{Ch64}, G. Z. Giambelli \cite{Gi03}, T. A. Hirst \cite{Hi75}, \cite{Hi77}, H. Schubert \cite{Sc86}, and C. Segre \cite{Se84}. These algebraic varieties have been fundamental in enumerative geometry \cite{LH82}, and moreover they carry a beautiful geometry, enriched by symmetries inherited from their modular nature.  

During the 20th century, spaces of complete forms continued to be extensively studied both from the geometrical and enumerative point of view by J. G. Semple \cite{Se48}, \cite{Se51}, \cite{Se52}, J. A. Tyrrell \cite{Ty56}, I. Vainsencher \cite{Va82}, \cite{Va84}, S. Kleiman, D. Laksov, A. Lascoux, A. Thorup \cite{TK88}, \cite{LLT89}, and M. Thaddeus \cite{Tha99}. Furthermore, spaces of complete collineations are central not only in the study of other moduli spaces such as Hilbert schemes and Kontsevich spaces of twisted cubics \cite{Al56}, \cite{Pi81}, but also in their construction. Indeed, recently F. Cavazzani constructed moduli spaces of complete homogeneous varieties as GIT quotients of spaces of complete collineations \cite{Ca16}. Finally, the birational geometry of the spaces of complete quadrics and their relation with other moduli spaces such as Hilbert schemes and Kontsevich spaces of conics has been carefully investigated by C. Lozano Huerta in \cite{Ce15}.

The aim of this paper is to investigate the birational geometry of moduli spaces of complete forms from the point of view of Mori theory by exploiting their spherical nature. Roughly speaking, a \textit{Mori dream space} is a projective variety $X$ whose cone of effective divisors $\Eff(X)$ admits a well-behaved decomposition into convex sets called Mori chambers, and these chambers are the nef cones of birational models of $X$. These varieties, introduced by Y. Hu and S. Keel in \cite{HK00}, are named so because they behave in the best possible way from the point of view of Mori's minimal model program.

Given a reductive algebraic group $\mathscr{G}$ and a Borel subgroup $\mathscr{B}$, a \textit{spherical variety} is a variety admitting an action of $\mathscr{G}$ with an open dense $\mathscr{B}$-orbit. A special class of spherical varieties are the so called \textit{wonderful varieties} for which we require the existence of an open orbit whose complementary set is a simple normal crossing divisor. On a spherical variety we distinguish two types of $\mathscr{B}$-invariant prime divisors: a \textit{boundary divisor} is a $\mathscr{G}$-invariant prime divisor on $X$, a \textit{color} is a $\mathscr{B}$-invariant prime divisor that is not $\mathscr{G}$-invariant. Our interest in spherical and wonderful varieties comes from the fact that they are Mori dream spaces and that spaces of complete forms are wonderful; we refer to \cite{Pe14} for a comprehensive treatment of these topics. 

In this paper, having fixed two $K$-vector spaces $V,W$ respectively of dimension $n+1$, $m+1$ with $n\leq m$ over an algebraically closed field of characteristic zero, we will denote by $\mathcal{X}(n,m)$ the spaces of complete collineations $W\rightarrow V$ and we will set $\mathcal{X}(n):=\mathcal{X}(n,n)$, and by $\mathcal{Q}(n)$ the space of complete $(n-1)$-dimensional quadrics of $V$.

Recall that given an irreducible and reduced non-degenerate variety $X\subset\P^N$, and a positive integer $h\leq N$ the \textit{$h$-secant variety} $\sec_h(X)$ of $X$ is the subvariety of $\P^N$ obtained as the closure of the union of all $(h-1)$-planes spanned by $h$ general points of $X$. Spaces of complete forms reflect the triad of varieties parametrizing rank one matrices, composed by Grassmannians together with Veronese and Segre varieties. Spaces of matrices admit a natural stratification dictated by the rank. Indeed, a general point of the $h$-secant variety of a Grassmannian, a Veronese or a Segre corresponds to a matrix of rank $h$. The starting point of our investigation are constructions of the spaces of complete collineations and quadrics due to I. Vainsencher \cite{Va82}, \cite{Va84}, as sequences of blow-ups along the relevant variety parametrizing rank one matrices and the strict transforms of its secant varieties in order of increasing dimension. For instance, to obtain the space of complete collineations $\mathcal{X}(3)$ of $\mathbb{P}^3$ we must blow-up $\mathbb{P}^{15} = \mathbb{P}(\Hom(V,V))$ along the Segre variety $\mathcal{S}\cong\mathbb{P}^3\times\mathbb{P}^3$, and then along the strict transform of its variety of secant lines $\sec_2(\mathcal{S})$. Note that we do not need to blow-up the strict transform of the variety of $3$-secant planes $\sec_3(\mathcal{S})$ since it is a hypersurface in $\mathbb{P}^{15}$ which becomes a smooth divisor after the first two blow-ups.

In Section \ref{sec2}, as a warm up, we analyze the natural actions of $SL(n)\times SL(m)$ on $\mathcal{X}(n,m)$, and of $SL(n)$ on $\mathcal{Q}(n)$, and we compute the respective boundary divisors and colors. Thanks to general results on the cones of divisors of spherical varieties due to M. Brion \cite{Br89}, combined with an analysis of the projective geometry of distinguished hypersurfaces defined by the vanishing of certain minors of a general matrix, we manage to compute the effective and nef cones of these spaces, and as a consequence we get also an explicit presentation of their Mori cones and their cones of moving curves.

In Section \ref{sec3}, which is the core of the paper, we take advantage of the computation of the boundary divisors and colors in Section \ref{sec2} to give minimal sets of generators for the \textit{Cox rings} of spaces of complete forms. Cox rings were first introduced by D. A. Cox for toric varieties \cite{Cox95}, and then his construction was generalized to projective varieties in \cite{HK00}. These algebraic objects are basically universal homogeneous coordinate rings of projective varieties, defined as direct sum of the spaces of sections of all isomorphism classes of line bundles on them. For instance, as a consequence of the main results in Theorems \ref{theff}, \ref{gen} and in Propositions \ref{Grass}, \ref{LG} we have the following statement.
\begin{thm}
For any $i = 1,\dots,n+1$ let us denote by $D_i$ the strict transform in $\mathcal{X}(n)$ of the divisor in $\mathbb{P}(\Hom(V,V))$, with homogeneous coordinates $[z_{0,0}:\dots:z_{n,n}]$, given by 
\begin{equation*}
\det\left(
\begin{array}{ccc}
z_{n-i+1,n-i+1} & \dots & z_{n-i+1,n}\\ 
\vdots & \ddots & \vdots\\ 
z_{n,n-i+1} & \dots & z_{n,n}
\end{array}\right)=0
\end{equation*}
and by $E_j$ the exceptional divisors of the blow-ups in Vainsencher's construction. Then we have that $\Eff(\mathcal{X}(n))=\left\langle E_1,\dots,E_n\right\rangle$ and $\Nef(\mathcal{Q}(n)) = \left\langle D_1,\dots,D_{n}\right\rangle$.

Furthermore, the canonical sections associated to the $D_i$ and the $E_j$ form a set of minimal generators of $\Cox(\mathcal{X}(n))$, and the analogous statements hold for $\mathcal{X}(n,m)$ and $\mathcal{Q}(n)$. 

For those of these spaces which have Picard rank two we have that $\Cox(\mathcal{X}(1,m))$, $\Cox(\mathcal{X}(2))$, and $\Cox(\mathcal{Q}(2))$ are isomorphic to the homogeneous coordinate rings respectively of the Grassmannian $\mathcal{G}(1,m+2)$ parametrizing lines in $\mathbb{P}^{m+2}$, of the Grassmannian $\mathcal{G}(2,5)$ parametrizing planes in $\mathbb{P}^{5}$, and of the Lagrangian Grassmannian $\mathcal{LG}(2,5)$ parametrizing $3$-dimensional Lagrangian subspaces of a fixed $6$-dimensional vector space.
\end{thm} 

In Section \ref{birmod} we study the birational models of $\mathcal{X}(n,m)$ induced by the extremal rays of the nef cone $\Nef(\mathcal{X}(n,m))$. In particular, we show that these models parametrize rational normal curves of osculating spaces of degree $n$ rational normal curves in $\mathbb{P}^n$. For instance, $\Nef(\mathcal{X}(3))$ has three extremal rays that we will denote by $D_1,D_2,D_2$. The divisors $D_i$ are big and so they induce birational models $\mathcal{X}(3)(D_i)$ of $\mathcal{X}(3)$. A general point of $\mathcal{X}(3)(D_1)$ corresponds to a twisted cubic, a general point of $\mathcal{X}(3)(D_2)$ corresponds to a rational normal quartic in the Grassmannian of lines of $\mathbb{P}^3$ parametrizing tangent lines to a twisted cubic, and a general point of $\mathcal{X}(3)(D_3)$ corresponds to a twisted cubic parametrizing osculating planes to another twisted cubic. 

A normal $\mathbb{Q}$-factorial projective variety $X$, over an algebraically closed field, with finitely generated Picard group is a Mori dream space if and only if its Cox ring is finitely generated \cite[Proposition 2.9]{HK00}. In this case the birational geometry of $X$ is completely encoded in the combinatorics of $\Cox(X)$. So, theoretically there is a way of computing the Mori chamber decomposition of a Mori dream space starting from an explicit presentation of its Cox ring, even though in practice, due to the intricacy of the involved combinatorics, this is impossible to do. Let us recall that the pseudo-effective cone $\overline{\Eff}(X)$ of a projective variety $X$ with irregularity zero, such as a Mori dream space, can be decomposed into chambers depending on the stable base locus of the corresponding linear series. Such decomposition called \textit{stable base locus decomposition} in general is coarser than the Mori chamber decomposition.   

In Section \ref{sec4} thanks to the computation of the generators of the Cox rings in Section \ref{sec3} we determine the Mori chamber and the stable base locus decomposition of $\Eff(\mathcal{X}(3))$. Indeed, since $\Pic(\mathcal{X}(3))\otimes\mathbb{R}$ is $3$-dimensional, the information on the generators of $\Cox(\mathcal{X}(3))$, even without knowing the relations among them, together with the computation of the stable base loci and basic considerations of convex geometry, is enough to completely describe both decompositions, which interestingly enough do not coincide. The same techniques apply to $\mathcal{X}(2,m)$ for $m> 2$. Indeed, as a consequence of Theorems \ref{MCD_main}, \ref{MCD_main2} we have the following.

\begin{thm}
The Mori chamber decomposition of $\Eff(\mathcal{X}(3))$ consists of $9$ chambers while its stable base locus decomposition consists of $8$ chambers. Furthermore, the same result holds for the variety $\mathcal{Q}(3)\subset\mathcal{X}(3)$ parametrizing complete quadric surfaces.

If $n > 2$ then the Mori chamber decompositon $\Eff(\mathcal{X}(2,m))$ has $5$ chambers while its stable base locus decomposition consists of $4$ chambers. Finally the Mori chamber and stable base locus decompositions of $\Eff(\mathcal{X}(2))$ coincide and have $3$ chambers, and the same result holds for the variety $\mathcal{Q}(2)\subset\mathcal{X}(2)$ parametrizing complete conics.
\end{thm}

These decompositions are described in detail in Theorems \ref{MCD_main}, \ref{MCD_main2}. We would like to stress that in \cite{Ce15} C. Lozano Huerta computed, with different methods, $\Eff(\mathcal{Q}(n))$, $\Nef(\mathcal{Q}(n))$, and the Mori chamber decomposition of $\Eff(\mathcal{Q}(3))$ also giving a modular interpretation of some of the corresponding birational models.

Furthermore, in Proposition \ref{firstbu} we describe the Mori chamber decomposition of the first blow-up in Vainsencher's construction, namely the blow-up of $\mathbb{P}(\Hom(V,V))$ along the Segre variety. Thanks to this in Proposition \ref{Sarkisov} we get a neat description as a Sarkisov link of the birational involution $\mathbb{P}^{N}\dasharrow\mathbb{P}^{N*}$ defined by mapping a full rank matrix to its inverse. Such involution induces an automorphism $Z^{inv}$ of $\mathcal{X}(n)$. In Section \ref{psaut} we will prove that the groups of automorphisms and pseudo-automorphisms of spaces of complete forms coincide, and we will explicitly compute these groups where $Z^{inv}$ will appear as the only non trivial generator of the discrete part.   

In Appendix \ref{appendix} we discuss the movable cones of spaces of complete forms. Indeed, by Proposition \ref{PropMov} the number of extremal rays of the movable cones are given by geometric progressions. For instance, $\Mov(\mathcal{X}(n))$ and $\Mov(\mathcal{Q}(n))$ have $2^{n-1}$ extremal rays, and the isomorphism $\Pic(\mathcal{X}(n))\rightarrow\Pic(\mathcal{Q}(n))$, induced by the natural inclusion $\mathcal{Q}(n)\hookrightarrow\mathcal{X}(n)$, maps isomorphically $\Mov(\mathcal{X}(n))$ onto $\Mov(\mathcal{Q}(n))$. In Appendix \ref{appendix} we present Maple scripts explicitly computing these extremal rays.

We conclude the introduction with a general consideration. Note that, the main results in this paper suggest that the birational geometry of $\mathcal{X}(n)$ is completely determined by those of $\mathcal{Q}(n)$, and lead us to the following general question.

\begin{Question}
Let $X\hookrightarrow Y$ be an inclusion of Mori dream spaces. Under which hypothesis is the birational geometry of $Y$, meaning its cones of divisors and curves, and its Mori chamber and stable base locus decompositions, determined by those of $X$ and vice versa?  
\end{Question}

An obvious but not very satisfactory answer would be: when $X\hookrightarrow Y$ yields an isomorphism of groups $\Pic(Y)\rightarrow\Pic(X)$, and an isomorphism $\Cox(Y)\rightarrow\Cox(X)$ of graded $K$-algebras with respect to a fixed grading on $\Pic(Y)\cong\Pic(X)$. 

However, some partial general results in this direction are given in Lemma \ref{MCrest}, while more specific facts for the varieties $\mathcal{Q}(n)$ and $\mathcal{X}(n)$ are presented in Propositions \ref{MCXnQn}, \ref{firstbu}, and Theorem \ref{MCD_main}.

\subsection*{Organization of the paper}
All through the paper we will work over an algebraically closed field of characteristic zero. Both statements and proofs regarding spaces of complete collineations will be given in full detail. However, we will sometimes go quickly over the analogous statements and proofs for their symmetric counterpart. Anyway, whenever some additional argument is needed we will explain it in detail. 

In Section \ref{sec1} we recall the basics on moduli spaces of complete forms and their constructions as blow-ups. In Section \ref{sec2} we compute their cones of divisors and curves, and in Section \ref{sec3} we compute the generators of their Cox rings. In Section \ref{birmod} we describe the birational models of the spaces of complete collineations induced by the extremal rays of the nef cone, and in Section \ref{sec4} we compute the Mori chamber and stable base locus decompositions of the space of complete collineations of $\mathbb{P}^3$. Finally, in Section \ref{psaut} we give an explicit presentation of the pseudo-automorphism groups of space of complete forms, and in Appendix \ref{appendix} we present Maple scripts computing the extremal rays of their effective cones.  

While this paper is devoted to spaces of complete collineations and quadrics, their skew-symmetric counterpart, that is the space of complete skew-symmetric forms, is considered in \cite{Ma18}.

\subsection*{Acknowledgments}
I thank Antonio Laface, C\'esar Lozano Huerta and Rick Rischter for their useful comments, and Jethro van Ekeren for fruitful discussions on the representation theoretical aspects of this paper. The author is a member of the Gruppo Nazionale per le Strutture Algebriche, Geometriche e le loro Applicazioni of the Istituto Nazionale di Alta Matematica "F. Severi" (GNSAGA-INDAM).

\section{Moduli spaces of complete forms}\label{sec1}
Let $V,W$ be $K$-vector spaces of dimension respectively $n+1$ and $m+1$ with $n\leq m$, and let $\mathbb{P}^N$ with $N = nm+n+m$ be the projective space parametrizing collineations from $V$ to $W$ that is non-zero linear maps $V\rightarrow W$ up to a scalar multiple. 

The line bundle $\mathcal{O}_{\mathbb{P}^n\times \mathbb{P}^m}(1,1)=\mathcal{O}_{\mathbb{P}(V)}(1)\boxtimes\mathcal{O}_{\mathbb{P}(W)}(1)$
induces an embedding
$$
\begin{array}{cccc}
\sigma:
&\mathbb{P}(V)\times\mathbb{P}(W)& \longrightarrow & \mathbb{P}(V\otimes W)
=\mathbb{P}^{N},\\
      & (\left[u\right],\left[v\right]) & \longmapsto & [u\otimes v]
\end{array}
$$ 
The image $\mathcal{S} = \sigma(\mathbb{P}^n\times \mathbb{P}^m) \subset \mathbb{P}^{N}$ is the \textit{Segre variety}. Let $[x_0,\dots, x_n],[y_0,\dots,y_m]$ be homogeneous coordinates respectively on $\mathbb{P}^n$ and $\mathbb{P}^m$. Then the morphism $\sigma$ can be written as
$$\sigma([x_0,\dots, x_n],[y_0,\dots,y_m]) = [x_0y_0:\dots:x_0y_m:x_1y_0:\dots :x_ny_m]$$
We will denote by $[z_{0,0}:\dots :z_{n,m}]$ the homogeneous coordinates on $\mathbb{P}^N$, where $z_{i,j}$ corresponds to the product $x_iy_j$.  

\subsubsection*{Secant varieties} 
Given an irreducible and reduced non-degenerate variety $X\subset\P^N$, and a positive integer $h\leq N$ we denote by $\sec_h(X)$ 
the \emph{$h$-secant variety} of $X$. This is the subvariety of $\P^n$ obtained as the closure of the union of all $(h-1)$-planes 
$\langle x_1,...,x_{h}\rangle$ spanned by $h$ general points of $X$. 

A point $p\in \mathbb{P}^N = \mathbb{P}(\Hom(W,V))$ can be represented by an $(n+1)\times (m+1)$ matrix $Z$. The Segre variety $\mathcal{S}$ is the locus of rank one matrices. More generally, $p\in \sec_h(S)$ if and only if $Z$ can be written as a linear combination of $h$ rank one matrices that is if and only if $\rank(Z)\leq h$. If $p = [z_{0,0}:\cdots:z_{n,m}]$ then we may write
\stepcounter{thm}
\begin{equation}\label{matrix}
Z = \left(
\begin{array}{ccc}
z_{0,0} & \dots & z_{0,m}\\ 
\vdots & \ddots & \vdots\\ 
z_{n,0} & \dots & z_{n,m}
\end{array}\right)
\end{equation}
Then, the ideal of $\sec_h(\mathcal{S})$ is generated by the $(h+1)\times (h+1)$ minors of $Z$. The \textit{space of complete collineations} from $V$ to $W$ is the closure in 
\stepcounter{thm}
\begin{equation}\label{ambient}
\mathbb{P}(\Hom(W,V))\times\mathbb{P}(\Hom(\bigwedge^2 W,\bigwedge^2 V))\times\dots\times \mathbb{P}(\Hom(\bigwedge^{n+1} W,\bigwedge^{n+1} V))
\end{equation}
of the graph of the rational map
\stepcounter{thm}
\begin{equation}\label{graph}
\begin{array}{ccc}
\mathbb{P}(\Hom(W,V))& \dasharrow & \mathbb{P}(\Hom(\bigwedge^2 W,\bigwedge^2 V))\times\dots\times \mathbb{P}(\Hom(\bigwedge^{n+1} W,\bigwedge^{n+1} V))\\
 Z & \longmapsto & (\wedge^2Z,\dots,\wedge^{n+1}Z)
\end{array}
\end{equation}

\begin{Construction}\label{ccc}
Let us consider the following sequence of blow-ups:
\begin{itemize}
\item[-] $\mathcal{X}(n,m)_1$ is the blow-up of $\mathcal{X}(n,m)_0:=\mathbb{P}^{N}$ along $\mathcal{S}$;
\item[-] $\mathcal{X}(n,m)_2$ is the blow-up of $\mathcal{X}(n,m)_1$ along the strict transform of $\sec_2(\mathcal{S})$;\\
$\vdots$
\item[-] $\mathcal{X}(n,m)_i$ is the blow-up of $\mathcal{X}(n,m)_{i-1}$ along the strict transform of $\sec_i(\mathcal{S})$;\\
$\vdots$
\item[-] $\mathcal{X}(n,m)_n$ is the blow-up of $\mathcal{X}(n,m)_{n-1}$ along the strict transform of $\sec_n(\mathcal{S})$.
\end{itemize}
Let $f_i:\mathcal{X}(n,m)_i\rightarrow \mathcal{X}(n,m)_{i-1}$ be the blow-up morphism. We will denote by $E_i$ both the exceptional divisor of $f_i$ and its strict transforms in the subsequent blow-ups, and by $H$ the pull-back to $\mathcal{X}(n,m):=\mathcal{X}(n,m)_n$ of the hyperplane section of $\mathbb{P}^{N}$. We will denote by $f:\mathcal{X}(n,m)\rightarrow\mathbb{P}^N$ the composition of the $f_i$'s.

By \cite[Theorem 1]{Va84} we have that for any $i = 1,\dots,n$ the variety $\mathcal{X}(n,m)_{i}$ is smooth, the strict transform of $\sec_i(\mathcal{S})$ in $\mathcal{X}(n,m)_{i-1}$ is smooth, and the divisor $E_1\cup E_2\cup\dots \cup E_{i-1}$ in $\mathcal{X}(n,m)_{i-1}$ is simple normal crossing. Furthermore, the variety $\mathcal{X}(n,m)$ is isomorphic to the space of complete collineations from $V$ to $W$. When $n=m$ we will write simply $\mathcal{X}(n)$ for $\mathcal{X}(n,n)$.
\end{Construction}
 
\begin{Remark}\label{sym_anti}
Consider the case $n = m$. Note that in $\mathbb{P}^N = \mathbb{P}(\Hom(V,V))$ we have the  distinguished linear subspace $\mathbb{P}^{N_{+}} = \{z_{i,j}-z_{j,i}=0 \: \forall \: i\neq j\}$, where $N_{+} = \binom{n+2}{2}-1$, of symmetric matrices. Therefore, $\mathbb{P}^{N_{+}}$ cuts out scheme-theoretically on $\mathcal{S}$ the Veronese variety $\mathcal{V}\subseteq \mathbb{P}^{N_{+}}$ parametrizing rank one $(n+1)\times (n+1)$ symmetric matrices, and more generally $\mathbb{P}^{N_{+}}$ cut out scheme-theoretically on $\sec_h(\mathcal{S})$ the $h$-secant variety $\sec_h(\mathcal{V})$.
\end{Remark}

The \textit{space of complete quadrics} is the closure of the graph of the rational map 
$$
\begin{array}{ccc}
\mathbb{P}(\Sym^2V)& \dasharrow & \mathbb{P}(\Sym^2\bigwedge^{2}V)\times\dots\times \mathbb{P}(\Sym^2\bigwedge^nV)\\
 Z & \longmapsto & (\wedge^2Z,\dots,\wedge^{n}Z)
\end{array}
$$

By Remark \ref{sym_anti} restricting Construction \ref{ccc} to $\mathbb{P}^{N_{+}}$ we get the following blow-up construction of the space of complete quadrics \cite[Theorem 6.3]{Va82}.

\begin{Construction}\label{ccq}
Let us consider the following sequence of blow-ups:
\begin{itemize}
\item[-] $\mathcal{Q}(n)_1$ is the blow-up of $\mathcal{Q}(n)_0:=\mathbb{P}^{N_{+}}$ along the Veronese variety $\mathcal{V}$;
\item[-] $\mathcal{Q}(n)_2$ is the blow-up of $\mathcal{Q}(n)_1$ along the strict transform of $\sec_2(\mathcal{V})$;\\
$\vdots$
\item[-] $\mathcal{Q}(n)_i$ is the blow-up of $\mathcal{Q}(n)_{i-1}$ along the strict transform of $\sec_i(\mathcal{V})$;\\
$\vdots$
\item[-] $\mathcal{Q}(n)_n$ is the blow-up of $\mathcal{Q}(n)_{n-1}$ along the strict transform of $\sec_n(\mathcal{V})$.
\end{itemize}
Let $f_i^{+}:\mathcal{Q}(n)_i\rightarrow \mathcal{Q}(n)_{i-1}$ be the blow-up morphism. We will denote by $E_i^+$ both the exceptional divisor of $f_i^{+}$ and its strict transforms in the subsequent blow-ups, and by $H^{+}$ the pull-back to $\mathcal{Q}(n):=\mathcal{Q}(n)_n$ of the hyperplane section of $\mathbb{P}^{N_{+}}$. We will denote by $f^{+}:\mathcal{Q}(n)\rightarrow\mathbb{P}^{N_{+}}$ the composition of the $f_i^{+}$'s. 

Then for any $i = 1,\dots,n$ the variety $\mathcal{Q}(n)_{i}$ is smooth, the strict transform of $\sec_i(\mathcal{V})$ in $\mathcal{Q}(n)_{i-1}$ is smooth, and the divisor $E_1^+\cup E_2^+\cup\dots \cup E_{i-1}^+$ in $\mathcal{Q}(n)_{i-1}$ is simple normal crossing. Furthermore, the variety $\mathcal{Q}(n)$ is isomorphic to the space of complete $(n-1)$-dimensional quadrics.
\end{Construction}
 
\section{Cones of divisors and curves}\label{sec2}
Let $X$ be a normal projective $\mathbb{Q}$-factorial variety over an algebraically closed field of characteristic zero. We denote by $N^1(X)$ the real vector space of $\mathbb{R}$-Cartier divisors modulo numerical equivalence. 
The \emph{nef cone} of $X$ is the closed convex cone $\Nef(X)\subset N^1(X)$ generated by classes of nef divisors. 

The stable base locus $\textbf{B}(D)$ of a $\mathbb{Q}$-divisor $D$ is the set-theoretic intersection of the base loci of the complete linear systems $|sD|$ for all positive integers $s$ such that $sD$ is integral
\stepcounter{thm}
\begin{equation}\label{sbl}
\textbf{B}(D) = \bigcap_{s > 0}B(sD)
\end{equation}
The \emph{movable cone} of $X$ is the convex cone $\Mov(X)\subset N^1(X)$ generated by classes of 
\emph{movable divisors}. These are Cartier divisors whose stable base locus has codimension at least two in $X$.
The \emph{effective cone} of $X$ is the convex cone $\Eff(X)\subset N^1(X)$ generated by classes of 
\emph{effective divisors}. We have inclusions $\Nef(X)\ \subset \ \overline{\Mov(X)}\ \subset \ \overline{\Eff(X)}$. We refer to \cite[Chapter 1]{De01} for a comprehensive treatment of these topics. 

\begin{Definition}
A \textit{spherical variety} is a normal variety $X$ together with an action of a connected reductive affine algebraic group $\mathscr{G}$, a Borel subgroup $\mathscr{B}\subset \mathscr{G}$, and a base point $x_0\in X$ such that the $\mathscr{B}$-orbit of $x_0$ in $X$ is a dense open subset of $X$. 

Let $(X,\mathscr{G},\mathscr{B},x_0)$ be a spherical variety. We distinguish two types of $\mathscr{B}$-invariant prime divisors: a \textit{boundary divisor} of $X$ is a $\mathscr{G}$-invariant prime divisor on $X$, a \textit{color} of $X$ is a $\mathscr{B}$-invariant prime divisor that is not $\mathscr{G}$-invariant. We will denote by $\mathcal{B}(X)$ and $\mathcal{C}(X)$ respectively the set of boundary divisors and colors of $X$.
\end{Definition}

For instance, any toric variety is a spherical variety with $\mathscr{B}=\mathscr{G}$ equal to the torus. For a toric variety there are no colors, and the boundary divisors are the usual toric invariant divisors.

In the following we will carefully analyze the natural actions of $SL(n+1)\times SL(m+1)$ on $\mathcal{X}(n,m)$, and of $SL(n+1)$ on $\mathcal{Q}(n)$ in order to get information of the cones of divisors of these spaces.

\begin{Lemma}\label{l1}
The $\mathscr{G} = SL(n+1)\times SL(m+1)$-action on $\mathbb{P}^n\times\mathbb{P}^m$ given by 
$$
\begin{array}{cccc}
\mathscr{G}\times (\mathbb{P}^n\times\mathbb{P}^m) & \longrightarrow & \mathbb{P}^n\times\mathbb{P}^m\\
((A,B),([v],[w])) & \longmapsto & ([Av],[Bw])
\end{array}
$$ 
induces the following $\mathscr{G}$-action on $\mathbb{P}^N = \mathbb{P}(\Hom(V,W))$
$$
\begin{array}{cccc}
\mathscr{G}\times \mathbb{P}(\Hom(V,W)) & \longrightarrow & \mathbb{P}(\Hom(V,W))\\
((A,B),Z) & \longmapsto & AZB^{t}
\end{array}
$$ 
Furthermore, if $n = m$ the $SL(n+1)$-action
$$
\begin{array}{cccc}
SL(n+1)\times \mathbb{P}^n & \longrightarrow & \mathbb{P}^n\\
(A,[v]) & \longmapsto & [Av]
\end{array}
$$
induce the $SL(n+1)$-action on $\mathbb{P}^{N_{+}}$ given by
\stepcounter{thm}
\begin{equation}\label{acss}
\begin{array}{cccc}
SL(n+1)\times \mathbb{P}^{N_{+}} & \longrightarrow & \mathbb{P}^{N_{+}}\\
(A,Z) & \longmapsto & AZA^{t}
\end{array}
\end{equation}
\end{Lemma} 
\begin{proof}
Let $[v] = [x_0:\dots: x_n]$ and $[w] = [y_0:\dots :y_m]$. We have that $\sigma(Av,Bw)$ is given by
\begin{normalsize}
$$
\left(\begin{array}{ccccc}
a_{0,0}b_{0,0}x_0y_0+\dots +a_{0,n}b_{0,m}x_ny_m & \hdots & a_{0,0}b_{1,0}x_0y_0+\dots + a_{0,n}b_{m,m}x_ny_m\\ 
\vdots & \ddots & \vdots\\ 
a_{i,0}b_{0,0}x_0y_0+\dots + a_{i,n}b_{0,m}x_ny_m & \hdots & a_{i,0}b_{m,0}x_0y_0+\dots + a_{i,n}b_{m,m}x_ny_m \\ 
\vdots & \ddots & \vdots\\ 
a_{n,0}b_{0,0}x_0y_0+\dots +a_{n,n}b_{0,m}x_ny_m & \hdots & a_{n,0}b_{m,0}x_0y_0+\dots + a_{n,n}b_{m,m}x_ny_m
\end{array}\right) 
$$
\end{normalsize}
where $A = (a_{i,j})$ and $B = (b_{i,j})$. Recall that $z_{i,j}$ is the homogeneous coordinate on $\mathbb{P}^N = \mathbb{P}(\Hom(V,W))$ corresponding to the monomial $x_ix_j$. Therefore, $z_{i,j}$ is mapped by $(A,B)\in \mathscr{G}$ to 
$$a_{i,0}b_{j,0}z_{0,0}+\dots + a_{i,i}b_{j,j}z_{i,j} + \dots + a_{i,j}b_{j,i}z_{j,i} +\dots + a_{in}b_{j,m}z_{n,m}$$
that is the matrix $Z = (z_{i,j})$ is mapped by $(A,B)\in \mathscr{G}$ to the matrix $AZB^{t}$.

Now, let $n = m$. If $Z$ is symmetric then $(AZA^{t})^t = (A^t)^tZ^tA^t = AZA^t$. Therefore, the $SL(n+1)$-action on $\mathbb{P}^{N_{+}}$ in (\ref{acss}) is well-defined. 

In the symmetric case note that the action $SL(n+1)\curvearrowright \mathbb{P}^n$ is just the diagonal action of $\mathscr{G}\curvearrowright \mathbb{P}^n\times \mathbb{P}^n$.
\end{proof} 
 
\begin{Definition}
A \textit{wonderful variety} is a smooth projective variety $X$ with the action of a semi-simple simply connected group $\mathscr{G}$ such that:
\begin{itemize}
\item[-] there is a point $x_0\in X$ with open $\mathscr{G}$ orbit and such that the complement $X\setminus \mathscr{G}\cdot x_0$ is a union of prime divisors $E_1,\cdots, E_r$ having simple normal crossing;
\item[-] the closures of the $\mathscr{G}$-orbits in $X$ are the intersections $\bigcap_{i\in I}E_i$ where $I$ is a subset of $\{1,\dots, r\}$.
\end{itemize} 
\end{Definition} 
 
\begin{Proposition}\label{p1}
The varieties $\mathcal{X}(n,m)$, $\mathcal{Q}(n)$ are wonderful. Hence, in particular they are spherical.  

The Picard group of $\mathcal{X}(n,m)$ is given by 
$$\Pic(\mathcal{X}(n,m)) = 
\left\lbrace\begin{array}{ll}
\mathbb{Z}[H,E_1,\dots,E_{n}] & \text{if n $<$ m}\\ 
\mathbb{Z}[H,E_1,\dots,E_{n-1}] & \text{if n $=$ m}
\end{array}\right.
$$
For any $i = 1,\dots,n+1$ let us denote by $D_i$ the strict transform of the divisor in $\mathbb{P}^N$ given by 
\stepcounter{thm}
\begin{equation}\label{eqcol}
\det\left(
\begin{array}{ccc}
z_{n-i+1,m-i+1} & \dots & z_{n-i+1,m}\\ 
\vdots & \ddots & \vdots\\ 
z_{n,m-i+1} & \dots & z_{n,m}
\end{array}\right)=0
\end{equation}
The set of boundary divisors and colors of $\mathcal{X}(n,m)$ are given respectively by
$$
\mathcal{B}(\mathcal{X}(n,m)) = \{E_1,\dots,E_{n}\}
$$
$$
\mathcal{C}(\mathcal{X}(n,m)) = 
\left\lbrace\begin{array}{ll}
\{D_1,\dots,D_{n+1}\} & \text{if n $<$ m}\\ 
\{D_1,\dots,D_{n}\} & \text{if n $=$ m}
\end{array}\right.
$$
Now, let $n=m$. Then 
$$\Pic(\mathcal{Q}(n)) = \mathbb{Z}[H^{+},E_1^{+},\dots,E_{n-1}^{+}]$$ 
and
$$\mathcal{B}(\mathcal{Q}(n)) = \{E_1^{+},\dots,E_n^{+}\},\quad \mathcal{C}(\mathcal{Q}(n))=\{D_1^{+},\dots,D_n^{+}\}$$
where the $D_i^{+}$'s are defined by setting in (\ref{eqcol}) $z_{i,j} = z_{j,i}$.
\end{Proposition}
\begin{proof}
Since by Construction \ref{ccc} $\mathcal{X}(n,m)$ can be obtained by a sequence of blow-ups of smooth varieties along smooth centers the statement on the Picard group is straightforward. It is enough to observe that if $n = m$ then the strict transform of $\sec_n(\mathcal{S})$ is a Cartier divisor in $X[n]_{n-1}$ and hence the last blow-up in Construction \ref{ccc} is an isomorphism. The situation in the symmetric case is completely analogous.

By \cite[Theorem 1]{Va84}, \cite[Theorem 6.3]{Va82} the varieties $\mathcal{X}(n,m)$, $\mathcal{Q}(n)$ are wonderful with the actions of $SL(n+1)\times SL(m+1)$, and $SL(n+1)$ described in Lemma \ref{l1}. Therefore, they are in particular spherical as proven by D. Luna in \cite{Lu96}. 

We will develop in full detail the computation of the boundary divisors and the colors of $\mathcal{X}(n,m)$. 

As noticed in \cite[Remark 4.5.5.3]{ADHL15}, if $(X,\mathscr{G},\mathscr{B},x_0)$ is a spherical wonderful variety with colors $D_1,\dots,D_s$ the big cell $X\setminus (D_1\cup\dots \cup D_s)$ is an affine space. Therefore, it admits only constant invertible global functions and $\Pic(X) = \mathbb{Z}[D_1,\dots,D_s]$.

We consider the Borel subgroup
\stepcounter{thm}
\begin{equation}\label{borel}
\mathscr{B} = \{(A,B)\in SL(n+1)\times SL(m+1) \: |\: A, B\: \text{are upper triangular}\}
\end{equation}
We know that
$$\rank(\Pic(\mathcal{X}(n,m))= 
\left\lbrace\begin{array}{ll}
n+1 & \text{if \textit{n < m}}\\ 
n & \text{if \textit{n = m}}
\end{array}\right.
$$
Therefore, we must exhibit exactly $s = n+1$ colors when $n<m$, and $s = n$ colors when $n=m$. Note that these are exactly the number of colors predicted by the statement. By Lemma \ref{l1} an element $(A,B)\in \mathscr{G}$ maps the matrix $Z$ to the matrix $\overline{Z}=AZB^{t}$. Now, set
$$
A_{2,2}^i=\left(
\begin{array}{ccc}
a_{n-i+1,n-i+1} & \dots & a_{n-i+1,n}\\ 
\vdots & \ddots & \vdots\\ 
a_{n,n-i+1} & \dots & a_{n,n}
\end{array}\right)\quad
B_{2,2}^i=\left(
\begin{array}{ccc}
b_{m-i+1,m-i+1} & \dots & b_{m-i+1,m}\\ 
\vdots & \ddots & \vdots\\ 
b_{m,m-i+1} & \dots & b_{m,m}
\end{array}\right)
$$
$$
Z_{2,2}^i=\left(
\begin{array}{ccc}
z_{n-i+1,m-i+1} & \dots & z_{n-i+1,m}\\ 
\vdots & \ddots & \vdots\\ 
z_{n,m-i+1} & \dots & z_{n,m}
\end{array}\right)\quad
\overline{Z}_{2,2}^i=\left(
\begin{array}{ccc}
\overline{z}_{n-i+1,m-i+1} & \dots & \overline{z}_{n-i+1,m}\\ 
\vdots & \ddots & \vdots\\ 
\overline{z}_{n,m-i+1} & \dots & \overline{z}_{n,m}
\end{array}\right)
$$
and subdivide the matrices $A,B,Z,\overline{Z}$ in blocks as follows
$$
A = \left(\begin{array}{cc}
A_{1,1}^i & A_{1,2}^i\\ 
A_{2,1}^i & A_{2,2}^i
\end{array}\right)\:
B = \left(\begin{array}{cc}
B_{1,1}^i & B_{1,2}^i\\ 
B_{2,1}^i & B_{2,2}^i
\end{array}\right)\:
Z = \left(\begin{array}{cc}
Z_{1,1}^i & Z_{1,2}^i\\ 
Z_{2,1}^i & Z_{2,2}^i
\end{array}\right)\:
\overline{Z} = \left(\begin{array}{cc}
\overline{Z}_{1,1}^i & \overline{Z}_{1,2}^i\\ 
\overline{Z}_{2,1}^i & \overline{Z}_{2,2}^i
\end{array}\right)
$$
Then for $i=1,\dots, n$ the matrix $AZB^{t}$ can be subdivided in four blocks as follows
$$
\left\lbrace\begin{array}{l}
\overline{Z}_{1,1}^i = A_{1,1}^iZ_{1,1}^iB_{1,1}^{i\:t}+A_{1,2}^iZ_{2,1}^iB_{1,1}^{i\:t}+A_{1,1}^{i\:t}Z_{1,2}^{i\:t}B_{1,2}^{i\:t}+A_{1,2}^{i\:t}Z_{2,2}^{i\:t}B_{1,2}^{i\:t}\\ 
\overline{Z}_{1,2}^i = A_{1,1}^iZ_{1,1}^iB_{2,1}^{i\:t}+A_{1,2}^iZ_{2,1}^iB_{2,1}^{i\:t}+A_{1,1}^iZ_{1,2}^iB_{2,2}^{i\:t}+A_{1,2}^iZ_{2,2}^iB_{2,2}^{i\:t}\\
\overline{Z}_{2,1}^i = A_{2,1}^iZ_{1,1}^iB_{1,1}^{i\:t}+A_{2,2}^iZ_{2,1}^iB_{1,1}^{i\:t}+A_{2,1}^iZ_{1,2}^iB_{1,2}^{i\:t}+A_{2,2}^iZ_{2,2}^iB_{1,2}^{i\:t}\\ 
\overline{Z}_{2,2}^i = A_{2,1}^iZ_{1,1}^iB_{2,1}^{i\:t}+A_{2,2}^iZ_{2,1}^iB_{2,1}^{i\:t}+A_{2,1}^iZ_{1,2}^iB_{2,2}^{i\:t}+A_{2,2}^iZ_{2,2}^iB_{2,2}^{i\:t}
\end{array}\right.
$$
Therefore, the degree $i$ hypersurface $\{\det(Z_{2,2}^i)=0\}\subset\mathbb{P}^N$ is stabilized by the action of $\mathscr{G}$ if and only if $n = m$ and $i = n+1$. Indeed, $\{\det(Z_{2,2}^{n+1})=0\}\subset\mathbb{P}^N$ is the $n$-secant variety $\sec_n(\mathcal{S})$ of the Segre variety $\mathcal{S}$. Note that since $\mathscr{G}$ stabilizes the Segre variety $\mathcal{S}$ it must map a general $(h-1)$-plane $h$-secant to $\mathcal{S}$ to an $(h-1)$-plane with the same property. Since $\sec_h(\mathcal{S})$ is irreducible this means that $\mathscr{G}$ stabilizes $\sec_h(\mathcal{S})$ for $h = 1,\dots,n$. 

Now, assume that $(A,B)\in\mathscr{B}$. Then $A_{2,1}^i$ and $B_{2,1}^i$ are the zero matrix, and 
\stepcounter{thm}
\begin{equation}\label{funeq}
\overline{Z}_{2,2}^i = A_{2,2}^iZ_{2,2}^iB_{2,2}^{i\:t}
\end{equation}
This yields 
$$\det(\overline{Z}_{2,2}^i) = \det(A_{2,2}^i)\det(Z_{2,2}^i)\det(B_{2,2}^{i\:t})$$
Note that $A_{2,2}^i$ is upper triangular while $B_{2,2}^{i\:t}$ is lower triangular, and that the diagonals of $A_{2,2}^i$ and $B_{2,2}^{i\:t}$ are made of elements of the diagonals respectively of $A$ and $B$. 

Therefore, $\det(A_{2,2}^i)\neq 0$, $\det(B_{2,2}^{i\:t})\neq 0$ and hence the hypersurface $\{\det(Z_{2,2}^i)=0\}\subset\mathbb{P}^N$ is stabilized by the action of the Borel subgroup $\mathscr{B}$ on $\mathbb{P}^N$. Therefore, the strict transform $D_i$ of $\{\det(Z_{2,2}^i)=0\}$ is stabilized by the action of $\mathscr{B}$, and it is stabilized by the action of $\mathscr{G}$ on $\mathcal{X}(n,m)$ if and only if $n = m$ and $i = n+1$. This concludes the proof of the statement on the set of colors of $\mathcal{X}(n,m)$.

Now, note that since $\mathscr{G}$ stabilizes $\sec_h(\mathcal{S})$ for $h = 1,\dots,n$, \cite[Chapter II, Section 7, Corollary 7.15]{Har77} yields that $\mathscr{G}$ stabilizes the exceptional divisors $E_i$ and the strict transform of $\sec_n(\mathcal{S})$ when $n = m$. Therefore, these are boundary divisors. 

Now, let $D\subset\mathcal{X}(n,m)$ be a $\mathscr{G}$-invariant divisor which is not exceptional for the blow-up morphism $f:\mathcal{X}(n,m)\rightarrow \mathbb{P}^N$ in Construction \ref{ccc}. Then $f_{*}D\subset\mathbb{P}^N$ is $\mathscr{G}$-invariant as well. Hence, in particular $f_{*}D\subset\mathbb{P}^N$ is $\mathscr{B}$-invariant. On the other hand, by the previous computation of the colors of $\mathcal{X}(n,m)$ we have that the only $\mathscr{G}$-invariant hypersurface in $\mathbb{P}^N$ is $\sec_n(\mathcal{S})$ when $n = m$. 

Therefore, the set of the boundary divisors is given by the exceptional divisors if $n<m$, and by the exceptional divisors plus the strict transform of $\sec_n(\mathcal{S})$ when $n = m$. In the symmetric case we may argue in an analogous way with $\mathscr{G} = SL(n+1)$.
\end{proof}

The following result will be fundamental in order to write down the classes of the colors in Proposition \ref{p1} in terms of the generators of the Picard groups.

\begin{Lemma}\label{mult}
Consider hypersurfaces $Y_k\subset\mathbb{P}^N$ defined as the zero locus of the determinant of a $(k+1)\times (k+1)$ minor of the matrix $Z$ in (\ref{matrix}), $Y_k^{+}\subset\mathbb{P}^{N_{+}}$ defined as the zero locus of the determinant of a $(k+1)\times (k+1)$ minor of the symmetrization of $Z$. Then 
$$
\mult_{\sec_h(\mathcal{S})}Y_k = \mult_{\sec_h(\mathcal{V})}Y_k^{+} =
\left\lbrace\begin{array}{ll}
k-h+1 & \text{if h $\leq$ k}\\ 
0 & \text{if h $>$ k}
\end{array}\right.
$$ 
\end{Lemma}
\begin{proof}
Note that $\deg(Y_k) = k+1$. Since the ideal of $\sec_h(\mathcal{S})$ is generated by the $(h+1)\times (h+1)$ minors of $Z$ in $I(\sec_i(\mathcal{S}))$ there are not polynomials of degree less than or equal to $h$. In particular, if $h> k$ a hypersurface of type $Y_k$ can not contain $\sec_h(\mathcal{S})$.

Now, consider the case $h \leq k$. Without loss of generality we may consider the hypersurface $Y_k$ given by $Y_k = \{\det(F)=0\}$ where 
\stepcounter{thm}
\begin{equation}\label{pol}
F = F(z_{0,0},\dots,z_{k,k}) = \det\left(
\begin{array}{ccc}
z_{0,0} & \dots & z_{0,k}\\ 
\vdots & \ddots & \vdots\\ 
z_{k,0} & \dots & z_{k,k}
\end{array}\right)
\end{equation} 
Note that 
$$\frac{\partial^j F}{\partial z_{0,0}^{j_{0,0}},\dots,\partial z_{k,k}^{j_{k,k}}}=0, \: j_{0,0}+\dots +j_{k,k}$$
whenever either $j_{r,s}\geq 2$ for some $r,s = 0,\dots,k$ or in the expression of the partial derivative there are at least two indexes either of type $j_{r,s},j_{r',s}$ or of type $j_{r,s},j_{r,s'}$. In all the other cases, a partial derivative of order $j$ is the determinant of the $(k-j+1)\times (k-j+1)$ minor of the matrix in \ref{pol} obtained by deleting the rows and the columns crossing in the elements which correspond to the variables with respect to which we are deriving.

Therefore, all the partial derivatives of order $j$ of $F$ vanish on $\sec_{k-j}(\mathcal{S})$. On the other hand, the non-zero partial derivatives of order $j+1$ of $F$ have degree $k-j$, and since $I(\sec_{k-j}(\mathcal{S}))$ is generated in degree $k-j+1$ they can not vanish on $\sec_{k-j}(\mathcal{S})$. We conclude that $\mult_{\sec_{k-j}(\mathcal{S})}Y_k = j+1$, that is $\mult_{\sec_h(\mathcal{S})}Y_k = k-h+1$ if $h \leq k$. 

In the symmetric case it is enough to recall that by Remark \ref{sym_anti} $\mathbb{P}^{N_{+}}$ cuts out scheme-theoretically on $\sec_h(\mathcal{S})$ the $h$-secant variety $\sec_h(\mathcal{V})$.   
\end{proof}

\begin{Remark}\label{stpb}
Let $Y$ be a smooth and irreducible subvariety of a smooth variety $X$, and let $f:Bl_YX\rightarrow X$ be the blow-up of $X$ along $Y$ with exceptional divisor $E$. Then for any divisor $D\in \Pic(X)$ in $\Pic(Bl_YX)$ we have
$$\widetilde{D} \sim f^{*}D-\mult_{Y}(D) E$$
where $\widetilde{D}\subset Bl_YX$ is the strict transform of $D$, and $\mult_Y(D)$ is the multiplicity of $D$ at a general point of $Y$. 
\end{Remark}

Now, we are ready to compute the effective and nef cones of the spaces of complete forms.

\begin{thm}\label{theff}
Let $n< m$. For the colors $D_1,\dots,D_{n+1}$ in Proposition \ref{p1} we have $D_1\sim H$ and
\stepcounter{thm}
\begin{equation}\label{Dk}
D_k \sim kH-\sum_{h=1}^{k-1}(k-h)E_h
\end{equation}
for $k = 2,\dots,n+1$. Furthermore, $\{E_1,\dots,E_n,D_{n+1}\}$ generate the extremal rays of $\Eff(\mathcal{X}(n,m))$, and $\{D_1,\dots,D_{n+1}\}$ generate the extremal rays of $\Nef(\mathcal{X}(n,m))$.

Now, let $n = m$. For the colors $D_1,\dots,D_{n}$ and $D_1^{+},\dots,D_{n}^{+}$ in Proposition \ref{p1} we have $D_1\sim H$, $D_1^{+}\sim H^{+}$ and
\stepcounter{thm}
\begin{equation}\label{Dkq}
D_k \sim kH-\sum_{h=1}^{k-1}(k-h)E_h, \quad D_k^{+} \sim kH^{+}-\sum_{h=1}^{k-1}(k-h)E_h^{+}
\end{equation}
for $k = 2,\dots,n$, and for the boundary divisors $E_n$ and $E_n^{+}$ we have
\stepcounter{thm}
\begin{equation}\label{lastsec}
E_{n}\sim (n+1)H-\sum_{h=1}^{n-1}(n-h+1)E_h, \quad E_{n}^{+}\sim (n+1)H^{+}-\sum_{h=1}^{n-1}(n-h+1)E_h^{+}
\end{equation}
Furthermore, $\{E_1,\dots,E_{n}\}$, $\{E_1^{+},\dots,E_{n}^{+}\}$ generate the extremal rays respectively of $\Eff(\mathcal{X}(n))$ and $\Eff(\mathcal{Q}(n))$, and $\{D_1,\dots,D_{n}\}$, $\{D_1^{+},\dots,D_{n}^{+}\}$ generate the extremal rays respectively of $\Nef(\mathcal{X}(n))$ and $\Nef(\mathcal{Q}(n))$. 
\end{thm}
\begin{proof}
First, consider the case $n < m$. Recall that $D_1$ is the strict transform of $\{z_{n,m}=0\}\subset\mathbb{P}^N$, and since the centers of the blow-ups in Construction \ref{ccc} are non-degenerated Remark \ref{stpb} yields $D_1\sim H$. Let $f:\mathcal{X}(n,m)\rightarrow\mathbb{P}^N$ be the blow-up morphism in Construction \ref{ccc}. By Remark \ref{stpb} in $\Pic(\mathcal{X}(n,m))$ we have 
$$D_k = \deg(f_{*}D_k)H-\sum_{h=1}^{n}\mult_{\sec_h(\mathcal{S})}(f_{*}D_k)E_h$$
for $k=2,\dots, n+1$. Now, to get (\ref{Dk}) it is enough to recall that since $f_{*}D_k$ is the zero locus of a $k\times k$ minor of the matrix $Z$ then $\deg(f_{*}D_k)=k$, and that by Lemma \ref{mult} we have $\mult_{\sec_h(\mathcal{S})}(f_{*}D_k) = k-h$ if $h = 1,\dots k-1$ and $\mult_{\sec_h(\mathcal{S})}(f_{*}D_k) = 0$ if $h\geq k$.  

By Proposition \ref{p1} we have $\mathcal{B}(\mathcal{X}(n,m))= \{E_1,\dots,E_{n}\}$ and $\mathcal{C}(\mathcal{X}(n,m)) = \{D_1,\dots,D_{n+1}\}$, and by \cite[Proposition 4.5.4.4]{ADHL15} $\Eff(\mathcal{X}(n,m))$ is generated by $\{E_1,\dots,E_{n},D_1,\dots,D_{n+1}\}$. Note that since 
$$D_{n+1} \sim (n+1)H-\sum_{h=1}^{n}(n-h+1)E_h$$
in the set $\{E_1,\dots,E_{n},D_{n+1}\}$ there is not a divisor that can be written as a linear combination with non-negative coefficients of the other ones. On the other hand, if $k\leq n$ we may write
$$D_k\sim kH-\sum_{h=1}^{k-1}(k-h)E_h \sim \frac{k}{n+1}D_{n+1}+\sum_{h=k}^{n}\frac{k(n-h+1)}{n+1}E_h+\sum_{h=1}^{k-1}\frac{h(n-k+1)}{n+1}E_h$$  
with $n-k+1> 0$.

Furthermore, by \cite[Section 2.6]{Br89} the colors $\{D_1,\dots,D_{n+1}\}$ generate $\Nef(\mathcal{X}(n,m))$. Now, to conclude that they generate the extremal rays of $\Nef(\mathcal{X}(n,m))$ it is enough to observe that by (\ref{Dk}) they are numerical independent.  

Now, consider the case $n = m$. In order to get (\ref{Dkq}) it is enough to argue as in the first part of the proof. Note that $E_n$ is the strict transform of $\sec_{n}(\mathcal{S})\subset\mathbb{P}^N$. This secant variety is a hypersurface of degree $n+1$. Now, to get (\ref{lastsec}) it is enough to apply Lemma \ref{mult}.

In order to prove that $\{E_1,\dots,E_{n}\}$ generate the extremal rays of $\Eff(\mathcal{X}(n))$ by \cite[Proposition 4.5.4.4]{ADHL15} it is enough to notice that these divisors are numerically independent, and that any of the colors in $\mathcal{C}(\mathcal{X}(n))$ can be written as a linear combination with non-negative coefficients of these divisors. Indeed, we may write
$$
D_k\sim kH-\sum_{h=1}^{k-1}(k-h)E_h \sim \frac{k}{n+1}E_{n}+\sum_{h=k}^{n-1}\frac{k(n-h+1)}{n+1}E_h+\sum_{h=1}^{k-1}\frac{h(n-k+1)}{n+1}E_h
$$
with $n-k+1> 0$. As for the case $n<m$ the statement on the nef cone follows from \cite[Section 2.6]{Br89} and Proposition \ref{p1}. In the symmetric case it is enough to apply Proposition \ref{p1} and Lemma \ref{mult} as we did for $\mathcal{X}(n)$.
\end{proof}

Now, consider the intermediate spaces $\mathcal{X}(n,m)_i$ appearing in Construction \ref{ccc}. By a slight abuse of notation we will keep denoting by $E_i,D_i$ the push-forwards to $\mathcal{X}(n,m)_i$ of the corresponding divisors on $\mathcal{X}(n,m)$ in Proposition \ref{p1}.

\begin{Proposition}\label{int}
Let $\mathcal{X}(n,m)_i$ be the intermediate space appearing at the step $1\leq i\leq n-1$ of Construction \ref{ccc}. Then $\mathcal{X}(n,m)_i$ is spherical. Furthermore, $\Pic(\mathcal{X}(n,m)_i) = \mathbb{Z}[H,E_1,\dots,E_i]$ and 
$$
\mathcal{B}(\mathcal{X}(n,m)_i) = 
\left\lbrace\begin{array}{ll}
\{E_1,\dots,E_{i}\} & \text{if n $<$ m}\\ 
\{E_1,\dots,E_{i},D_{n+1}\} & \text{if n $=$ m}
\end{array}\right.
$$
$$
\mathcal{C}(\mathcal{X}(n,m)_i) = 
\left\lbrace\begin{array}{ll}
\{D_1,\dots,D_{n+1}\} & \text{if n $<$ m}\\ 
\{D_1,\dots,D_{n}\} & \text{if n $=$ m}
\end{array}\right.
$$
Finally, $\Eff(\mathcal{X}(n,m)_i) = \left\langle E_1,\dots,E_i,D_{n+1}\right\rangle$ and $\Nef(\mathcal{X}(n,m)_i) = \left\langle D_1,\dots,D_{i+1}\right\rangle$. The analogous statements, with the obvious modifications, hold for the intermediate spaces $\mathcal{Q}(n)_i$ in Construction \ref{ccq}.  
\end{Proposition}
\begin{proof}
Note that $\mathcal{X}(n,m)_i$ is spherical, even though it is not wonderful, with respect to the action of $\mathscr{G} = SL(n+1)\times SL(m+1)$, with the Borel subgroup $\mathscr{B}$ in (\ref{borel}). Since the actions of $\mathscr{G}$ on $\mathcal{X}(n,m)$ and $\mathcal{X}(n,m)_i$ are compatible with the blow-up morphism $\mathcal{X}(n,m)\rightarrow \mathcal{X}(n,m)_i$, the intermediate space $\mathcal{X}(n,m)_i$ has at most the same number of boundary divisors and colors as $\mathcal{X}(n,m)$. Now, arguing as in the proof of Proposition \ref{p1} it is straightforward to compute $\mathcal{B}(\mathcal{X}(n,m)_i)$ and $\mathcal{C}(\mathcal{X}(n,m)_i)$.

Finally, arguing as in the proof of Theorem \ref{theff} we get also the statements on the effective and the nef cones.
\end{proof}

\subsection*{Cones of curves}
Let $X$ be a normal projective variety and let $N_1(X)$ be the real vector space of numerical equivalence classes of $1$-cycles on $X$. The closure of the cone in $N_1(X)$ generated by the classes of irreducible curves in $X$ is called the \textit{Mori cone} of $X$, we will denote it by $\NE(X)$. 

A class $[C]\in N_1(X)$ is called \textit{moving} if the curves in $X$ of class $[C]$ cover a dense open subset of $X$. The closure of the cone in $N_1(X)$ generated by classes of moving curves in $X$ is called the \textit{moving cone} of $X$ and we will denote it by $\mov(X)$.  

Our aim now is to describe these cones for spaces of complete forms. When $n < m$ we will denote by $\{l,e_1,\dots,e_n\}$ the basis of $N_1(\mathcal{X}(n,m))$ dual to $\{H,E_1,\dots,E_n\}$, and similarly when $n = m$ we will denote by $\{l,e_1,\dots,e_{n-1}\}$ the basis of $N_1(\mathcal{X}(n))$ dual to $\{H,E_1,\dots,E_{n-1}\}$. We will adopt the analogous notations for the basis of $N_1(\mathcal{Q}(n))$ dual to the basis of $\Pic(\mathcal{Q}(n))$ in Proposition \ref{p1}.

\begin{Proposition}
Let $n<m$ and write the class of a general curve $C\in N_1(\mathcal{X}(n,m))$ as $C = dl-\sum_{i=1}^nm_ie_i$. Then the cone of moving curves $\mov(\mathcal{X}(n,m))$ is defined by 
$$
\left\lbrace
\begin{array}{l}
m_i\geq 0 \text{ for $i=1,\dots n$}\\ 
d(n+1)-\sum_{i=1}^{n}(n-i+1)m_i\geq 0
\end{array}\right.
$$
In particular the extremal rays of $\mov(\mathcal{X}(n,m))$ are generated by the curves of class $l$ and $(n-i+1)l-(n+1)e_i$ for $i = 1,\dots,n$.

If $n = m$ then $\mov(\mathcal{X}(n))$ and $\mov(\mathcal{Q}(n))$ are defined by
$$
\left\lbrace
\begin{array}{l}
m_i\geq 0 \text{ for $i=1,\dots n-1$}\\ 
d(n+1)-\sum_{i=1}^{n-1}(n-i+1)m_i\geq 0
\end{array}\right.
$$
\end{Proposition}
\begin{proof}
By \cite[Theorem 2.2]{BDPP13} the cone of moving curves is dual to the effective cone. By Theorem \ref{theff} we have $\Eff(\mathcal{X}(n,m)) = \left\langle E_1,\dots,E_n,D_{n+1}\right\rangle$. Now, to conclude it is enough to note that $C\cdot E_i = m_i$ for $i = 1,\dots,n$ and $C\cdot D_{n+1} = d(n+1)-\sum_{i=i}^n(n-i+1)m_i$.

In the case $n=m$, and in the symmetric case we can argue exactly in the same way by using the relevant presentation of the effective cone in Theorem \ref{theff}. 
\end{proof}

\begin{Example}
Consider for instance the case $n = m = 3$. Then the extremal rays of $\mov(\mathcal{X}(3))$ are generated by $l$, $3l-4e_1$ and $l-2e_2$. In this case we have $\mathcal{S}\subset\mathbb{P}^{15}$. Take four general points $p_1,\dots,p_4\in \mathcal{S}$, let $\left\langle p_1,\dots,p_4\right\rangle\cong\mathbb{P}^{3}$ be their span, and consider a twisted cubic $C\subset\left\langle p_1,\dots,p_4\right\rangle$ passing through $p_1,\dots,p_4$. Then the strict transform of $C$ in $\mathcal{X}(3)$ has class $3l-4e_1$. Note that since $\sec_4(\mathcal{S}) = \mathbb{P}^{15}$ the curves of this class cover a dense open subset of $\mathcal{X}(3)$. Now, the strict transform of a line secant to $\sec_2(\mathcal{S})$ has class $l-2e_2$. A general point $p\in\mathbb{P}^{15}$ lies on a line secant to $\sec_2(\mathcal{S})$ if and only if $p$ lies on a $3$-plane four secant to $\mathcal{S}$. Again, since $\sec_4(\mathcal{S}) = \mathbb{P}^{15}$ the curves of class $l-2e_2$ cover a dense open subset of $\mathcal{X}(3)$. 
\end{Example}

\begin{Proposition}\label{moricone}
Let $n<m$. Then the extremal rays of the Mori cone $\NE(\mathcal{X}(n,m))$ of $\mathcal{X}(n,m)$ are generated by the curves of class $l-2e_1+e_2$, $e_i-2e_{i+1}+e_{i+2}$ for $i = 1,\dots,n-2$, $e_{n-1}-2e_n$ and $e_n$.  

If $n = m$ then the extremal rays of the Mori cone of $\mathcal{X}(n)$ and $\mathcal{Q}(n)$ are generated respectively by the curves of class $l-2e_1+e_2$, $e_i-2e_{i+1}+e_{i+2}$ for $i = 1,\dots,n-3$, $e_{n-2}-2e_{n-1}$ and $e_{n-1}$, and by the curves of class $l^{+}-2e_1^{+}+e_2^{+}$, $e_i^{+}-2e_{i+1}^{+}+e_{i+2}^{+}$ for $i = 1,\dots,n-3$, $e_{n-2}^{+}-2e_{n-1}^{+}$ and $e_{n-1}^{+}$.
\end{Proposition}
\begin{proof}
The Mori cone is dual to the nef cone. Since by Theorem \ref{theff} $\Nef(\mathcal{X}(n,m))$ has $n+1$ extremal rays and by Proposition \ref{p1} $\Pic(\Nef(\mathcal{X}(n,m)))$ has rank $n+1$, $\NE(\mathcal{X}(n,m))$ must have $n+1$ extremal rays as well. It is enough to find for any of the listed curves $n$ nef divisors among the generators of $\Nef(\mathcal{X}(n,m))$ in Theorem \ref{theff} having zero intersection product with the curve. Indeed, we have $D_k\cdot (l-2e_1-e_2) =0$ for any $k\neq 1$, and $D_k\cdot (e_i-2e_{i+1}+e_{i+2}) = 0$ for any $k \neq i+1$. Furthermore, $D_k\cdot (e_{n-1}-2e_n)=0$ for $k\neq n-1$, $D_k\cdot e_n=0$ for $k\neq n$ and $D_k\cdot e_n =0$ for $k\neq n+1$. The case $n=m$, and the symmetric case can be developed in a similar way.
\end{proof}

\begin{Corollary}\label{Fano}
The spaces of complete forms $\mathcal{X}(n,m)$, $\mathcal{Q}(n)$ are Fano varieties. 
\end{Corollary}
\begin{proof}
Consider the case $n<m$. Since
$$\codim_ {\P^N}(\sec_h(\mathcal{S})) = (n-h+1)(m-h+1)$$
for $1\leq h\leq n$ Construction \ref{ccc} yields
$$-K_{\mathcal{X}(n,m)}\sim (nm+n+m+1)H-\sum_{h=1}^{n}(nm+h^2-nh-mh+n+m-2h)E_h$$
Therefore, $-K_{\mathcal{X}(n,m)}\cdot (l-2e_1+e_2) = 3$, $-K_{\mathcal{X}(n,m)}\cdot (e_i-2e_{i+1}+e_{i+2}) = 2$, $-K_{\mathcal{X}(n,m)}\cdot (e_{n-1}-2e_n) = 2$, $-K_{\mathcal{X}(n,m)}\cdot e_n = m-n >0$ and by Proposition \ref{moricone} $-K_{\mathcal{X}(n,m)}$ is ample.

In the case $n=m$, and the symmetric case it is enough to argue similarly by recalling that $\codim_{\P^{N_{+}}}(\sec_h(\mathcal{V})) = \frac{n^2+3n-h(2n-h+3)+2}{2}$ for $h\leq n$.
\end{proof}

\section{On the Cox ring}\label{sec3}
Let $X$ be a normal $\mathbb{Q}$-factorial variety. We say that a birational map  $f: X \dasharrow X'$ to a normal projective variety $X'$  is a \emph{birational contraction} if its
inverse does not contract any divisor. 
We say that it is a \emph{small $\mathbb{Q}$-factorial modification} 
if $X'$ is $\mathbb{Q}$-factorial  and $f$ is an isomorphism in codimension one.
If  $f: X \dasharrow X'$ is a small $\mathbb{Q}$-factorial modification, then 
the natural pullback map $f^*:N^1(X')\to N^1(X)$ sends $\Mov(X')$ and $\Eff(X')$
isomorphically onto $\Mov(X)$ and $\Eff(X)$, respectively.
In particular, we have $f^*(\Nef(X'))\subset \overline{\Mov(X)}$.

Now, assume that the divisor class group $\Cl(X)$ is free and finitely generated, and fix a subgroup $G$ of the group of Weil divisors on $X$ such that the canonical map $G\rightarrow\Cl(X)$, mapping a divisor $D\in G$ to its class $[D]$, is an isomorphism. The \textit{Cox ring} of $X$ is defined as
$$\Cox(X) = \bigoplus_{[D]\in \Cl(X)}H^0(X,\mathcal{O}_X(D))$$
where $D\in G$ represents $[D]\in\Cl(X)$, and the multiplication in $\Cox(X)$ is defined by the standard multiplication of homogeneous sections in the field of rational functions on $X$. 

\begin{Definition}\label{def:MDS} 
A normal projective $\mathbb{Q}$-factorial variety $X$ is called a \emph{Mori dream space}
if the following conditions hold:
\begin{enumerate}
\item[-] $\Pic{(X)}$ is finitely generated, or equivalently $h^1(X,\mathcal{O}_X)=0$,
\item[-] $\Nef{(X)}$ is generated by the classes of finitely many semi-ample divisors,
\item[-] there is a finite collection of small $\mathbb{Q}$-factorial modifications
 $f_i: X \dasharrow X_i$, such that each $X_i$ satisfies the second condition above, and $
 \Mov{(X)} \ = \ \bigcup_i \  f_i^*(\Nef{(X_i)})$.
\end{enumerate}
\end{Definition}

The collection of all faces of all cones $f_i^*(\Nef{(X_i)})$ above forms a fan which is supported on $\Mov(X)$.
If two maximal cones of this fan, say $f_i^*(\Nef{(X_i)})$ and $f_j^*(\Nef{(X_j)})$, meet along a facet,
then there exist a normal projective variety $Y$, a small modification $\varphi:X_i\dasharrow X_j$, and $h_i:X_i\rightarrow Y$, $h_j:X_j\rightarrow Y$ small birational morphisms of relative Picard number one such that $h_j\circ\varphi = h_i$. The fan structure on $\Mov(X)$ can be extended to a fan supported on $\Eff(X)$ as follows. 

\begin{Definition}\label{MCD}
Let $X$ be a Mori dream space.
We describe a fan structure on the effective cone $\Eff(X)$, called the \emph{Mori chamber decomposition}.
We refer to \cite[Proposition 1.11]{HK00} and \cite[Section 2.2]{Ok16} for details.
There are finitely many birational contractions from $X$ to Mori dream spaces, denoted by $g_i:X\dasharrow Y_i$.
The set $\Exc(g_i)$ of exceptional prime divisors of $g_i$ has cardinality $\rho(X/Y_i)=\rho(X)-\rho(Y_i)$.
The maximal cones $\mathcal{C}$ of the Mori chamber decomposition of $\Eff(X)$ are of the form: $\mathcal{C}_i \ = \left\langle g_i^*\big(\Nef(Y_i)\big) , \Exc(g_i) \right\rangle$. We call $\mathcal{C}_i$ or its interior $\mathcal{C}_i^{^\circ}$ a \emph{maximal chamber} of $\Eff(X)$.
\end{Definition}

For instance, to get a concrete idea on how this decomposition works one could see \cite[Theorem 3.4]{AM16} in which the case of $\mathbb{P}^{n}$ blown-up at $n+3$ general points is considered. 

\begin{Remark}\label{sphMDS}
By \cite[Corollary 1.3.2]{BCHM10} smooth Fano varieties are Mori dream spaces. In fact, there is a larger class of varieties called log Fano varieties which are Mori dream spaces as well. By the work of M. Brion \cite{Br93} we have that $\mathbb{Q}$-factorial spherical varieties are Mori dream spaces. An alternative proof of this result can be found in \cite[Section 4]{Pe14}. In particular, by Proposition \ref{p1} $\mathcal{X}(n,m),\mathcal{Q}(n)$ are Mori dream spaces.
\end{Remark}

\begin{Remark}\label{dimCox}
By \cite[Proposition 2.9]{HK00} a normal and $\mathbb{Q}$-factorial projective
variety $X$ over an algebraically closed field $K$, with finitely generated Picard group is a Mori dream space if and only if $\Cox(X)$ is a finitely generated $K$-algebra. Furthermore, the following equality holds:
$$\dim\Cox(X) = \dim(X)+\rank\Pic(X)$$
see for instance \cite[Theorem 3.2.1.4]{ADHL15}.
\end{Remark}

\subsection*{Generators}
Let $I = \{i_0,\dots, i_k\}$, $J = \{j_0,\dots, j_k\}$ be two ordered sets of indexes with $0\leq i_0\leq\dots\leq i_k\leq n$ and $0\leq j_0\leq\dots\leq j_k\leq m$, and denote by $Z_{I,J}$ the $(k+1)\times (k+1)$ minor of $Z$ built with the rows indexed by $I$ and the columns indexed by $J$. Similarly, let $Z_{I,J}^{+}$ be the $(k+1)\times (k+1)$ minor of the symmetrization of $Z$ built with the rows indexed by $I$ and the columns indexed by $J$.  

\begin{Lemma}\label{rep}
Let $H_{k+1}$ be the $SL(n+1)$-module fitting in the following exact sequence
\stepcounter{thm}
\begin{equation}\label{exseq}
0\rightarrow I(\mathcal{G}(k,n))_2\rightarrow \Sym^2(\bigwedge^{k+1}V)\rightarrow H_{k+1}\rightarrow 0
\end{equation}
Then $H_{k+1}\cong H^0(\mathcal{G}(k,n),\mathcal{O}_{\mathcal{G}(k,n)}(2))$ and $\dim(H_{k+1}) = \frac{1}{k+2}\binom{n+1}{k+1}\binom{n+2}{k+1}$.
\end{Lemma}
\begin{proof}
The structure exact sequence for $\mathcal{G}(k,n)\subseteq\mathbb{P}(\Sym^2(\bigwedge^{k+1}V))= \mathbb{P}^{N}_{-}$ twisted by two reads
$$0\rightarrow \mathcal{I}_{\mathcal{G}(k,n)}(2)\rightarrow\mathcal{O}_{\mathbb{P}^{N}_{-}}(2)_{|\mathcal{G}(k,n)}\rightarrow\mathcal{O}_{\mathcal{G}(k,n)}(2)\rightarrow 0$$
Now, recall that as a consequence of the Borel–Weil–Bott theorem we have that $\mathcal{G}(k,n)$ is projectively normal in its Pl\"ucker embedding, in particular $H^1(\mathcal{G}(k,n),\mathcal{I}_{\mathcal{G}(k,n)}(2))=0$, and taking cohomology in the above exact sequence we get 
$$
0\rightarrow I(\mathcal{G}(k,n))_2\rightarrow \Sym^2(\bigwedge^{k+1}V)\rightarrow H^0(\mathcal{G}(k,n),\mathcal{O}_{\mathcal{G}(k,n)}(2))\rightarrow 0
$$
Therefore, $H_{k+1}\cong H^0(\mathcal{G}(k,n),\mathcal{O}_{\mathcal{G}(k,n)}(2))$. Finally, by \cite[Proposition 5.2]{EF09} we have
$$
\begin{array}{lll}
\dim H^0(\mathcal{G}(k,n),\mathcal{O}_{\mathcal{G}(k,n)}(2)) & = & \prod_{j=k+2}^{n+1}\frac{\binom{j+1}{2}}{\binom{j-k}{2}} = \prod_{j=k+2}^{n+1}\frac{(j+1)j}{(j-k)(j-k-1)}\\
 & = & \frac{(k+3)(k+2)}{2} \frac{(k+4)(k+3)}{3\cdot 2}\dots \frac{(n+1)n}{(n-k)(n-k-1)}\frac{(n+2)(n+1)}{(n-k+1)(n-k)}\\
 & = & \frac{(n+2)\dots (k+3)}{(n-k+1)!} \frac{(n+1)\dots (k+2)}{(n-k)!} = \frac{1}{k+2}\frac{(n+2)\dots (k+2)}{(n-k+1)!}\frac{(n+1)\dots (k+2)}{(n-k)!}\\
 & = & \frac{1}{k+2}\binom{n+2}{k+1}\binom{n+1}{k+1}
\end{array} 
$$
which is the formula for the dimension of $H_{k+1}$ in the statement. 
\end{proof}

\begin{Remark}
Lemma \ref{rep} yields that the dimension of the $SL(n+1)$-module $I(\mathcal{G}(k,n))_2$ of degree two Pl\"ucker relations generating $I(\mathcal{G}(k,n))$ as an ideal is 
$$\frac{1}{2}\binom{n+1}{k+1}^2+\frac{1}{2}\binom{n+1}{k+1}-\frac{1}{k+2}\binom{n+2}{k+1}$$
Note that if $k = 1$ this number is $\binom{n+1}{4}$ which is exactly the number of $4\times 4$ sub-Pfaffians of an $(n+1)\times (n+1)$ skew-symmetric matrix.
\end{Remark}

\begin{Lemma}\label{dimrep}
Let $\Gamma_{(0,\dots,0,2,0,\dots,0)}$ be the irreducible representation of $SL(n+1)$ with highest weight $(\lambda_1,\dots,\lambda_n)$ given by $\lambda_i =0$ if $i\neq k+1$ and $\lambda_{k+1} = 2$. Then
$$\dim \Gamma_{(0,\dots,0,2,0,\dots,0)} = \frac{1}{k+2}\binom{n+2}{k+1}\binom{n+1}{k+1}$$
\end{Lemma}
\begin{proof}
By the Weyl dimension formula we have
$$\dim \Gamma_{(\lambda_1,\dots,\lambda_n)} =\prod_{1\leq i < j\leq n+1} \frac{\lambda_i+\dots +\lambda_{j-1}+j-i}{j-i}$$
Consider the case $(\lambda_1,\dots,\lambda_{k+1},\dots,\lambda_n) = (0,\dots,2,\dots,0)$. For $i\geq k+2$ all terms in the product are equal to $1$. For any $i\leq k+1$ we must consider all terms with $i < j\leq n+1$. When $i = 1,2,\dots, k,k+1$ the product of the corresponding terms are given by
$$
\left\lbrace\begin{array}{ll}
i=1 & \frac{k+3}{k+1}\cdot\frac{k+4}{k+2}\dots \frac{n+2}{n}\\ 
i=2 & \frac{k+2}{k}\cdot\frac{k+3}{k+1}\dots \frac{n+1}{n-1}\\ 
\vdots & \vdots\\ 
i=k & \frac{4}{2}\cdot\frac{5}{3}\dots \frac{n-k+3}{n-k+1}\\ 
i = k+1 & 3\cdot\frac{4}{2}\dots \frac{n-k+2}{n-k}
\end{array}\right. 
$$
Therefore, we get
$$
\begin{array}{lll}
\dim \Gamma_{(0,\dots,0,2,0,\dots,0)} & = & \frac{(n+2)!k!}{n!(k+2)!} \frac{(n+1)!(k-1)!}{(n-1)!(k+1)!}\dots  \frac{(n-k+3)!}{3!(n-k+1)!} \frac{(n-k+2)!}{2(n-k)!}\\ 
• & = & \frac{(n+2)(n+1)}{(k+2)(k+1)}\frac{(n+1)n}{(k+1)k}\dots\frac{(n-k+3)(n-k+2)}{3\cdot 2}\frac{(n-k+2)(n-k+1)}{2}\\ 
• & = & \frac{1}{k+2} \frac{(n+2)(n+1)\dots (n-k+2)}{(k+1)!} \frac{(n+1)n\dots (n-k+1)}{(k+1)!}= \frac{1}{k+2}\binom{n+2}{k+1}\binom{n+1}{k+1}
\end{array} 
$$
which is exactly the formula in the statement. 
\end{proof}

By (\ref{exseq}) we have a splitting 
$$\Sym^2(\bigwedge^{k+1}V)\cong I(\mathcal{G}(k,n))_2\oplus H_{k+1}$$ 
as $K$-vector spaces. Furthermore, since $I(\sec_{k}(\mathcal{V}))_{k+1}$ is generated by the $Z_{I,J}^{+}$ we have an isomorphism $\Sym^2(\bigwedge^{k+1}V)\cong I(\sec_{k}(\mathcal{V}))_{k+1}$, and then we may identify $H_{k+1}$ with a vector subspace of $ K[z_{0,0},\dots,z_{n,n}]_{k+1}$. 

The blow-up morphism $f^{+}:\mathcal{Q}(n)\rightarrow\mathbb{P}^{N_{+}}$ in Construction \ref{ccq} induces an injective pull-back map
$$f^{+*}:\Cox(\mathbb{P}^{N_{+}})\cong K[z_{0,0},\dots,z_{n,n}]\rightarrow\Cox(\mathcal{Q}(n))$$

\begin{Definition}
For any $k=0,\dots,n-1$ we define the vector subspace $H_{k+1}^{+}\subseteq \Cox(\mathcal{Q}(n))$ as $H_{k+1}^{+}:= f^{+*}H_{k+1}$ 
\end{Definition}    

Now, taking advantage of these basic representation theoretical results in the symmetric case, we can give explicit sets of minimal generators for the Cox rings of the spaces of complete forms. 

\begin{thm}\label{gen}
Let $T_{I,J}$ be the canonical section associated to the strict transform of the hypersurface $\{\det(Z_{I,J})=0\}\subset\mathbb{P}^N$, and let $S_i$ be the canonical section associated to the exceptional divisor $E_i$ in Construction \ref{ccc}. 

If $n < m$ then $\Cox(\mathcal{X}(n,m))$ is generated by the $T_{I,J}$ for $1\leq |I|,|J|\leq n+1$ and the $S_i$ for $i=1,\dots,n-1$. If $n = m$ then $\Cox(\mathcal{X}(n))$ is generated by the $T_{I,J}$ for $1\leq |I|,|J|\leq n$ and the $S_i$ for $i=1,\dots,n$.

Now, consider $\mathcal{Q}(n)$ and let $S_i^{+}$ be the canonical section associated to the exceptional divisor $E_i^{+}$ in Construction \ref{ccq}. Then $\Cox(\mathcal{Q}(n))$ is generated by the subspaces $H_{k+1}^{+}$ for $k=0,\dots,n-1$ and the $S_i^{+}$ for $i=1,\dots,n$. 
\end{thm}
\begin{proof}
By \cite[Theorem 4.5.4.6]{ADHL15} if $\mathscr{G}$ is a semi-simple and simply connected algebraic group and $(X,\mathscr{G},\mathscr{B},x_0)$ is a spherical variety with boundary divisors $E_1,\dots,E_r$ and colors $D_1,\dots,D_s$ then $\Cox(X)$ is generated as a $K$-algebra by the canonical sections of the $E_i$'s and the finite dimensional vector subspaces $\lin_{K}(\mathscr{G}\cdot D_i)\subseteq \Cox(X)$ for $1\leq i\leq s$.

Consider the case $n < m$. By Proposition \ref{p1} we have that $\mathcal{B}(\mathcal{X}(n,m)) = \{E_1,\dots,E_n\}$ and $\mathcal{C}(\mathcal{X}(n,m)) = \{D_1,\dots,D_{n+1}\}$. Recall that for any $k=0,\dots, n$ the divisor $D_{k+1}$ is the strict transform of the hypersurface in $\mathbb{P}^N$ defined by the determinant of the $(k+1)\times (k+1)$ most right down minor of the matrix $Z$. Let us denote by $Z_{I,J}^{rd}$ such minor. 

Note that $\det(Z_{I,J}^{rd})\in I(\sec_{k}(\mathcal{S}))$, and by the description of the action of $\mathscr{G} =  SL(n+1)\times SL(m+1)$ in Lemma \ref{l1} we have that $g\cdot \det(Z_{I,J}^{rd})\in I(\sec_{k}(\mathcal{S}))$ for any $g\in G$, and hence $\lin_{K}(\mathscr{G}\cdot \det(Z_{I,J}^{rd})) \subseteq I(\sec_{k}(\mathcal{S}))$. Since $I(\sec_{k}(\mathcal{S}))$ is generated by the $(k+1)\times (k+1)$ minors of $Z$ it is enough to show that $\lin_{K}(\mathscr{G}\cdot \det(Z_{I,J}^{rd})) = I(\sec_{k}(\mathcal{S}))$.

Now, let $Z_{I,J}$ be the $(k+1)\times (k+1)$ minor of $Z$ corresponding to $I = \{i_0,\dots, i_k\}$, $J = \{j_0,\dots, j_k\}$. Recall that $z_{i,j} = x_ix_j$ in the notation of Section \ref{sec1}. Consider an element $a\in SL(n+1)$ such that $a\cdot x_{i_{n-k}} = x_{0}, a\cdot x_{n-k+1} = x_{i_1},\dots,a\cdot x_{n} = x_{i_k}$, and an element $b\in SL(m+1)$ such that $b\cdot y_{m-k} = y_{i_0}, a\cdot y_{m-k+1} = y_{i_1},\dots,a\cdot y_{m} = y_{i_k}$. Then $g = (a,b)\in \mathscr{G}$ is such that $g\cdot \det(Z_{I,J}^{rd})=\det(Z_{I,J})$.   

Therefore, $\lin_{K}(\mathscr{G}\cdot D_{k+1})$ is the subspace of $H^0(\mathcal{X}(n,m),D_{k+1})$ generated by the $T_{I,J}$'s with $|I| = |J| = k+1$.

For $\mathcal{X}(n)$ we can argue in an analogous way. The situation gets trickier when we consider $\mathcal{Q}(n)$. Note that in this case $\{\det(Z_{I,J}^{+}=0\}\subset\mathbb{P}^{N_{+}}$ is a cone and the dimension of the vertex of such cone depends on whether the minor $Z_{I,J}^{+}$ is principal or not. Indeed, starting from $Z_{I,J}^{rd+}$ we can reach any principal minor by a translation with an element of $\mathscr{G} = SL(n+1)$ as we did in the previous cases, but we can not reach via such a translation the non-principal minors of $Z$. 

Recall that to any $(n+1)\times (n+1)$ symmetric matrix $Z^{+}$ we may associate the $\binom{n+1}{k+1}\times \binom{n+1}{k+1}$ symmetric matrix $\wedge^{k+1}Z^{+}$. Now, $\wedge^{k+1}Z^{+}$ corresponds to a degree two homogeneous polynomial on $\bigwedge^{k+1}V$ that is to an element of $\Sym^2(\bigwedge^{k+1}V)$. This association gives a representation of $\mathscr{G} = SL(n+1)$ on $\Sym^2(\bigwedge^{k+1}V)\cong I(\sec_{k}(\mathcal{V}))$. 

Consider the orbit $\mathscr{G}\cdot\det(Z_{I,J}^{rd+})\subseteq I(\sec_{k}(\mathcal{V}))$. Our aim is to describe the irreducible sub representation $\lin_{K}(\mathscr{G}\cdot \det(Z_{I,J}^{rd+})) \subseteq  I(\sec_{k}(\mathcal{V}))$. 

Recall that if $v\in \bigwedge^{k+1}V$ is a highest weight vector of weight $(0,\dots,0,1,0,\dots,0)$ for $\mathscr{G}$, where the $1$ is in the $(k+1)$-th entry then $v\otimes v$ is a highest weight vector of weight $(0,\dots,0,2,0,\dots,0)$ for $\mathscr{G}$, where again the $2$ is in the $(k+1)$-th entry. In our case the role of such a highest weight vector is played by $Z_{I,J}^{rd+}$. Let $\Gamma_{(0,\dots,0,2,0,\dots,0)}$ be the corresponding irreducible representation.

Note that $I(\mathcal{G}(k,n))_2\subseteq \Sym^2(\bigwedge^{k+1}V)$ is stabilized by the action of $\mathscr{G}$. Therefore, we get a representation of $\mathscr{G}$ on $H_{k+1}$ and by the above discussion we have $\Gamma_{(0,\dots,0,2,0,\dots,0)}\subseteq H_{k+1}$. 
Finally, Lemmas \ref{rep} and \ref{dimrep} yield $\Gamma_{(0,\dots,0,2,0,\dots,0)} = H_{k+1}$, and hence $\lin_{K}(\mathscr{G}\cdot \det(Z_{I,J}^{rd+})) = \Gamma_{(0,\dots,0,2,0,\dots,0)} = H_{k+1}$.
\end{proof}

\begin{Example}
In order to help the reader in grasping the argument for the symmetric case in the proof of Theorem \ref{gen} we work out explicitly the first non-trivial case. Consider a $4\times 4$ symmetric matrix $Z^{+}$ with entries $z_{i,j}$. Then $\wedge^{2}Z^{+}$ is given by 
\[
\scalemath{0.8}{
\left(
\begin{array}{cccccc}
z_{0,0}z_{1,1}-z_{0,1}^2 & z_{0,0}z_{1,2}-z_{0,1}z_{0,2} & z_{0,0}z_{1,3}-z_{0,1}z_{0,3} & z_{0,1}z_{1,2}-z_{1,1}z_{0,2} & z_{0,1}z_{1,3}-z_{1,1}z_{0,2} & z_{0,2}z_{1,3}-z_{1,2}z_{0,3}\\ 
z_{0,0}z_{1,2}-z_{0,2}z_{0,1} & z_{0,0}z_{2,2}-z_{0,2}^2 & z_{0,0}z_{2,3}-z_{0,2}z_{0,3} & z_{0,1}z_{2,2}-z_{1,2}z_{0,2} & z_{0,1}z_{2,3}-z_{1,2}z_{0,3} & z_{0,2}z_{2,3}-z_{2,2}z_{0,3}\\ 
z_{0,0}z_{1,3}-z_{0,3}z_{0,1} & z_{0,0}z_{2,3}-z_{0,3}z_{0,2} & z_{0,0}z_{3,3}-z_{0,3}^2 & z_{0,1}z_{2,3}-z_{1,3}z_{0,2} & z_{0,1}z_{3,3}-z_{1,3}z_{0,3} & z_{0,2}z_{3,3}-z_{2,3}z_{0,3}\\
z_{0,1}z_{1,2}-z_{0,2}z_{1,1} & z_{0,1}z_{2,2}-z_{0,2}z_{1,2} & z_{0,1}z_{2,3}-z_{0,2}z_{1,3} & z_{1,1}z_{2,2}-z_{1,2}^2 & z_{1,1}z_{2,3}-z_{1,2}z_{1,3} & z_{1,2}z_{2,3}-z_{2,2}z_{1,3}\\ 
z_{0,1}z_{1,3}-z_{0,3}z_{1,1} & z_{0,1}z_{2,3}-z_{0,3}z_{1,2} & z_{0,1}z_{3,3}-z_{0,3}z_{1,3} & z_{1,1}z_{2,3}-z_{1,3}z_{1,2} & z_{1,1}z_{3,3}-z_{1,3}^2 & z_{1,2}z_{3,3}-z_{2,3}z_{1,3}\\ 
z_{0,2}z_{1,3}-z_{0,3}z_{1,2} & z_{0,2}z_{2,3}-z_{0,3}z_{2,2} & z_{0,2}z_{3,3}-z_{0,3}z_{2,3} & z_{1,2}z_{2,3}-z_{1,3}z_{2,2} & z_{1,2}z_{3,3}-z_{1,3}z_{2,3} & z_{2,2}z_{3,3}-z_{2,3}^2
\end{array}
\right)
}
\]
Let us interpret this matrix as a quadric on $\mathbb{P}(\bigwedge^2V)$. Fix homogeneous coordinates $[x_0:\dots:x_5]$ on $\mathbb{P}(\bigwedge^2V)$. Then we may write the quadric corresponding to $\wedge^{2}Z^{+}$ as 
$$Q_{\wedge^{2}Z^{+}}(x_0,\dots,x_5) = \sum_{0\leq i \leq 5}(\wedge^{2}Z^{+})_{(i,i)}x_i^2+2\sum_{0\leq i < j \leq 5}(\wedge^{2}Z^{+})_{(i,j)}x_ix_j$$
Note that the entries of $\wedge^{2}Z^{+}$ satisfy the relation 
\begin{normalsize}
$$(\wedge^{2}Z^{+})_{(0,5)}-(\wedge^{2}Z^{+})_{(1,4)}+(\wedge^{2}Z^{+})_{(2,3)} = z_{0,2}z_{1,3}-z_{1,2}z_{0,3}-z_{0,1}z_{2,3}+z_{1,2}z_{0,3}+z_{0,1}z_{2,3}-z_{1,3}z_{0,2}=0$$
\end{normalsize}
and indeed $x_{0}x_5-x_{1}x_{4}+x_{2}x_{3}=0$ is the Pl\"ucker equation cutting out $\mathcal{G}(1,3)$ in $\mathbb{P}(\bigwedge^2V)$. The vector space $\Sym^2(\bigwedge^2V)$ admits a decomposition 
$$\Sym^2(\bigwedge^2V) = \left\langle x_{0}x_5-x_{1}x_{4}+x_{2}x_{3}\right\rangle\oplus H_{2}$$ 
into two irreducible representations of $SL(4)$, where $H_2\cong H^0(\mathcal{G}(1,3),\mathcal{O}_{\mathcal{G}(1,3)}(2))$ is the linear span of the $SL(4)$-orbit of $z_{2,2}z_{3,3}-z_{2,3}^2$ in $\Sym^2(\bigwedge^2V)$.
\end{Example}

\begin{Remark}
Theorem \ref{gen} yields that $\Cox(\mathcal{X}(n,m))$ has exactly $\sum_{k=1}^{n+1}\binom{n+1}{k}\binom{m+1}{k}+n$ independent generators if $n<m$, and $\sum_{k=1}^{n+1}\binom{n+1}{k}\binom{m+1}{k}+n-1$ if $n=m$.   
\end{Remark}

\subsection*{Relations}
Let us consider $\mathcal{X}(1,m)$. The matrix $Z$ in (\ref{matrix}) is of the form
\begin{equation}\label{mat2}
\left(\begin{array}{cccc}
z_{0,0} & z_{0,1} & \dots & z_{0,m}\\ 
z_{1,0} & z_{1,1} & \dots & z_{1,m}
\end{array}\right)
\end{equation}
For simplicity of notation we denote by $T_{i,j}$ the canonical section associated to the strict transform of the hyperplane $\{z_{i,j}=0\}$, and by $R_{i,j}$ the canonical section of the strict transform of the quadric defined by the determinant of the $2\times 2$ minor constructed with the columns of the matrix (\ref{mat2}) indexed by $i,j$. 

By Theorem \ref{gen} $\Cox(\mathcal{X}(1,m))$ is generated by the $T_{i,j}$, the $R_{i,j}$ and $S_1$. By Theorem \ref{theff} the grading matrix, with respect to the natural grading induced by $\Pic(\mathcal{X}(1,m)) = \mathbb{Z}[H,E]$, of these generators is the following 
$$
\left(\begin{array}{ccc}
\deg(T_{i,j}) & \deg(R_{i,j}) & \deg(S_1)\\
1 & 2 &  0\\ 
0 & -1 & 1
\end{array}\right)
$$
Note that for any $2\times 3$ sub-matrix of (\ref{mat2}) 
$$
\left(\begin{array}{ccc}
z_{0,j_0} & z_{0,j_1} & z_{0,j_2}\\ 
z_{1,j_0} & z_{1,j_1} & z_{1,j_2}
\end{array}\right)
$$
we have two relations among the generators given by 
\stepcounter{thm}
\begin{equation}\label{rel1}
\left\lbrace\begin{array}{l}
T_{0,j_0}R_{j_1,j_2}-T_{0,j_1}R_{j_0,j_2}+T_{0,j_2}R_{j_0,j_1}=0\\
T_{1,j_0}R_{j_1,j_2}-T_{1,j_1}R_{j_0,j_2}+T_{1,j_2}R_{j_0,j_1}=0
\end{array}\right.
\end{equation}
Furthermore, since the quadric $\{z_{0,i}z_{1,j}-z_{0,j}z_{1,j}=0\}$ has multiplicity one along the Segre variety $\mathcal{S}\subseteq\mathbb{P}^N$ with a slight abuse of notation we may write
$$R_{i,j}S_1 = f^{*}(\{z_{0,i}z_{1,j}-z_{0,j}z_{1,i}=0\})$$
where $f:\mathcal{X}(1,m)\rightarrow\mathbb{P}^N$ is the blow-up morphism in Construction \ref{ccc}. In this way we get the relations
\stepcounter{thm}
\begin{equation}\label{rel2}
T_{0,j}T_{1,i}-T_{0,i}T_{1,j}+R_{i,j}S_1=0
\end{equation}
 
\begin{Proposition}\label{Grass}
The Cox ring of $\mathcal{X}(1,m)$ is given by
$$\Cox(\mathcal{X}(1,m))\cong\frac{K[T_{i,j},R_{h,k},S_1]}{I_{1,m}}$$
where $I_{1,m}$ is the ideal generated by the polynomials in (\ref{rel1}) and (\ref{rel2}).
\end{Proposition}
\begin{proof}
Note that the equations (\ref{rel1}) and (\ref{rel2}) are exactly the Pl\"ucker equations generating the homogeneous ideal of the Grassmannian of lines $\mathcal{G}(1,m+2)\subset\mathbb{P}^{\binom{m+3}{2}-1}$, where we are interpreting $T_{i,j},R_{i,j},S_1$ as the Pl\"ucker coordinates on $\mathbb{P}^{\binom{m+3}{2}-1}$. Now, Remark \ref{dimCox} yields 
$$\dim\Cox(\mathcal{X}(1,m)) = \dim(\mathcal{X}(1,m))+\rank\Pic(\mathcal{X}(1,m)) = 2(m+1)-1+2 =2m+3$$ 
Finally, to conclude it is enough to observe that $\mathcal{G}(1,m+2)$ is irreducible and reduced of dimension $2m+2 = \dim\Cox(\mathcal{X}(1,m))-1$.
\end{proof}

Furthermore, for the space $\mathcal{X}(2)$ of complete collineations of $\mathbb{P}^2$, and the space of complete conics $\mathcal{Q}(2)$ we have the following result.

\begin{Proposition}\label{LG}
Let $T_{I,J}^{+}$ be the canonical section associated to the strict transform of the hypersurface $\{\det(Z_{I,J})^{+}=0\}\subset\mathbb{P}^5$. The Cox rings of $\mathcal{X}(2)$ and $\mathcal{Q}(2)$ are given respectively by 
$$\Cox(\mathcal{X}(2))\cong\frac{K[T_{I,J},S_1,S_2]_{ 1\leq |I|,|J|\leq 2}}{I(\mathcal{G}(2,5))},\quad \Cox(\mathcal{Q}(2))\cong\frac{K[T_{I,J}^{+},S_1,S_2]_{ 1\leq |I|,|J|\leq 2}}{I(\mathcal{LG}(2,5))}$$
where $\mathcal{LG}(2,5)\subset\mathbb{P}^{13}$ is the Lagrangian Grassmannian parametrizing $3$-dimensional Lagrangian subspaces of a fixed $6$-dimensional vector space.
\end{Proposition}
\begin{proof}
Note that by Theorem \ref{gen} $\Cox(\mathcal{Q}(2))$ has exactly $14$ generators.
Let $Z = (z_{i,j})$ be a general $3\times 3$ symmetric matrix. The Lagrangian Grassmannian $\mathcal{LG}(2,5)$ is the closure of the image of the map
$$
\begin{array}{ccc}
\Sym^2V& \rightarrow & \mathbb{P}^{13}\\
 Z & \longmapsto & (1,z_{0,0},z_{0,1},z_{0,2},z_{1,1},z_{1,2},z_{2,2},\wedge^2Z,\wedge^3Z)
\end{array}
$$
where $\wedge^2Z$ represents the six $2\times 2$ minors of $Z$. When we take the closure we add the missing variable $S_1$, and the quadrics cutting out $\mathcal{LG}(2,5)\subset\mathbb{P}^{13}$ induce quadratic relations among the generators of $\Cox(\mathcal{Q}(2))$. Now, Remark \ref{dimCox} yields $\dim\Cox(\mathcal{Q}(2)) = 5 + 2 = 7$, and to conclude it is enough to observe that $\mathcal{LG}(2,5)\subset\mathbb{P}^{13}$ is a $6$-dimensional irreducible and reduced variety.

For $\mathcal{X}(2)$ we can argue in a completely analogous way, just note that in this case for the dimension of the Cox ring we have $\dim\Cox(\mathcal{X}(2)) = 8+2 = 10 = \dim(\mathcal{G}(2,5))+1$.
\end{proof}

As a first application of Theorem \ref{gen} we compute the number of extremal rays of the movable cones of the spaces of complete forms. 

\begin{Proposition}\label{PropMov}
The movable cones of $\mathcal{X}(n)$ and $\mathcal{Q}(n)$ are generated by $2^{n-1}$ extremal rays, while if $n< m$ then the movable cone of $\mathcal{X}(n,m)$ is generated by $2^{n-1}+1$ extremal rays. Furthermore, the isomorphism $i^{*}:\Pic(\mathcal{X}(n))\rightarrow\Pic(\mathcal{Q}(n))$, induced by the natural inclusion $i:\mathcal{Q}(n)\rightarrow\mathcal{X}(n)$, yields an isomorphism between $\Mov(\mathcal{X}(n))$ and $\Mov(\mathcal{Q}(n))$. 
\end{Proposition}
\begin{proof}
We will develop in full detail the argument for $\mathcal{X}(n)$. The same ideas with the obvious variations work for the other cases as well.

For $n=1$ the statement is trivially verified. By Proposition \ref{p1} $\mathcal{X}(2)$ has Picard rank two, and Theorem \ref{theff} yields that $\Nef(\mathcal{X}(2)) = \left\langle H, 2H-E\right\rangle$. Note that the divisor $H$ induces the blow-up morphism $f:\mathcal{X}(2)\rightarrow \mathbb{P}^{8}$ contracting the exceptional divisor $E$. Furthermore, the strict transform of a secant line to the Segre variety $\mathcal{S}\subset\mathbb{P}^8$ has zero intersection with $2H-E$, and hence $2H-E$ induces a birational morphism contracting the strict transform of $\sec_2(\mathcal{S})$, which is a divisor in $\mathcal{X}(2)$. This means that $H$ and $2H-E$ are the extremal rays of the movable cone of $\mathcal{X}(2)$, that is $\Mov(\mathcal{X}) = \Nef(\mathcal{X}) = \left\langle H, 2H-E\right\rangle$.

When $n\geq 3$ we use \cite[Proposition 3.3.2.3]{ADHL15} which gives a way of computing the movable cone starting from the generators of the Cox ring. For instance, consider $\mathcal{X}(3)$. By Theorem \ref{gen} the degree of the generators of $\Cox(\mathcal{X}(3))$, with respect to the natural grading on $\Pic(\mathcal{X}(3))$, are the columns of the matrix 
\stepcounter{thm}
\begin{equation}\label{gradmat}
\left(
\begin{array}{cccccc}
0 & 0 & 1 & 2 & 3 & 4\\ 
1 & 0 & 0 & -1 & -2 & -3\\ 
0 & 1 & 0 & 0 & -1 & -2
\end{array} 
\right)
\end{equation}
where the repeated degrees have been omitted. Now, via a straightforward computation \cite[Proposition 3.3.2.3]{ADHL15} yields
\stepcounter{thm}
\begin{equation}\label{mov3}
\Mov(\mathcal{X}(3)) = \left\langle H, 2H-E_1, 3H-2E_1-E_2, 6H-3E_1-2E_2 \right\rangle
\end{equation} 
Therefore the statement is verified for $n = 1,2,3$. We proceed by induction on $n$. Consider a general $\mathbb{P}^{n-1}\times \mathbb{P}^{n-1}\subset\mathbb{P}^n\times\mathbb{P}^n$, and the corresponding inclusion of Segre varieties $\sigma(\mathbb{P}^{n-1}\times \mathbb{P}^{n-1})\subset\sigma(\mathbb{P}^{n-1}\times \mathbb{P}^{n-1})\subset\mathbb{P}^N$. Restricting Construction \ref{ccc} to the smaller Segre variety we get an embedding $i:\mathcal{X}(n-1)\rightarrow\mathcal{X}(n)$ which in turn induces a surjective morphism
$$i^{*}:\Pic(\mathcal{X}(n))\rightarrow\Pic(\mathcal{X}(n-1))$$
mapping $\Mov(\mathcal{X}(n))$ onto $\Mov(\mathcal{X}(n-1))$. Note that $i^{*}$ is the linear projection from $E_n$. 

Let $D$ be an extremal ray of $\Mov(\mathcal{X}(n-1))$. Note that $E_1,H_1,\dots,H_{n-1}$ generate a hyperplane in $\Pic(\mathcal{X}(n))$, and that the plane $i^{*-1}(D)$ must contain at least one and at most two extremal rays of $\Mov(\mathcal{X}(n))$. Note that we may identify the hyperplane $\left\langle E_1,\dots,E_{n-1}\right\rangle\subset\Pic(\mathcal{X}(n))$ with $\Pic(\mathcal{X}(n-1))$.

By \cite[Proposition 3.3.2.3]{ADHL15} and the description of the generators of $\Cox(\mathcal{X}(n))$ in Theorem \ref{gen} we get that the plane $i^{*-1}(D)$ contains exactly two generators of $\Mov(\mathcal{X}(n))$, namely the rays given by the intersections of $i^{*-1}(D)$ with the hyperplanes $\left\langle E_1,H_1,\dots H_{n-1}\right\rangle$ and $\left\langle E_1,H_3,\dots,H_n\right\rangle$. 


Furthermore, note that if $D_1,D_2$ are two distinct extremal rays of $\Mov(\mathcal{X}(n-1))$ then $i^{*-1}(D_1)$ and $i^{*-1}(D_2)$ intersect exactly along the line generated by $E_n$.    

Therefore, for any extremal ray of $\Mov(\mathcal{X}(n-1))$ we have two extremal rays of $\Mov(\mathcal{X}(n))$, and finally we get the statement by induction.  

Finally, the fact that $i^{*}:\Pic(\mathcal{X}(n))\rightarrow\Pic(\mathcal{Q}(n))$ induces an isomorphism between $\Mov(\mathcal{X}(n))$ and $\Mov(\mathcal{Q}(n))$ descends from the computation of the generators of the Cox rings in Theorem \ref{gen} and \cite[Proposition 3.3.2.3]{ADHL15}.
\end{proof}

\begin{Remark}
In Appendix \ref{appendix} we present three Maple scripts, based on Theorem \ref{gen} and \cite[Proposition 3.3.2.3]{ADHL15}, computing the extremal rays of the movable cone of $\mathcal{X}(n,m)$, $\mathcal{X}(n)$ and $\mathcal{Q}(n)$. 
\end{Remark}

In what follows, by a slight abuse of notation, we will keep denoting by $T_{I,J}$ the canonical section associated to the strict transform of the hypersurface $\{\det(Z_{I,J})=0\}\subset\mathbb{P}^N$, and by $S_j$ the canonical section associated to the exceptional divisor $E_j$ in the intermediate spaces $\mathcal{X}(n,m)_i$ in Construction \ref{ccc}.

\begin{Proposition}\label{CoxInt}
Let $\mathcal{X}(n,m)_i$ be the intermediate space appearing at the step $1\leq i\leq n-1$ of Construction \ref{ccc}. The Cox ring $\Cox(\mathcal{X}(n,m)_i)$ is generated by the $T_{I,J}$ for $1\leq |I|,|J|\leq n+1$ and the $S_j$ for $j=1,\dots,i$. 

Furthermore, the analogous statements, with the obvious modifications, hold for the intermediate spaces $\mathcal{Q}(n)_i$ in Construction \ref{ccq}.  
\end{Proposition}
\begin{proof}
It is enough to argue as in the proof of Theorem \ref{gen} using Proposition \ref{int} instead of Proposition \ref{p1}.
\end{proof}

\section{Birational models}\label{birmod}
In this section, using the main results in Sections \ref{sec2} and \ref{sec3}, we will investigate birational models of the spaces of complete collineations. In particular, we will relate these models to varieties parametrizing degeneration loci of matrices of maximal rank and rational normal curves. 

\begin{Proposition}\label{morNEF}
For any $k=1,\dots,n+1$ the divisor $D_k$ in Theorem \ref{theff} induces a morphism $\pi_k:\mathcal{X}(n,m)\rightarrow\mathbb{P}(\Hom(\bigwedge^{k}W,\bigwedge^{k}V))$ mapping a general point $Z\in\mathcal{X}(n,m)$ to $\wedge^kZ$. 

If $ n<m$ then $\pi_k$ is birational onto its image for $k = 1,\dots,n$, while $\pi_{n+1}$ is a fibration onto the Grassmannian $\mathcal{G}(n,m)\subseteq\mathbb{P}(\bigwedge^{n+1}W^{*})$ parametrizing $n$-planes in $\mathbb{P}^m$.

In particular, if $n = m$ then $\pi_k:\mathcal{X}(n)\rightarrow\mathbb{P}(\Hom(\bigwedge^{k}W,\bigwedge^{k}V))$ is birational onto its image for any $k =1,\dots,n$.
\end{Proposition}
\begin{proof}
By (\ref{graph}) the space of complete collineations inherits $n+1$ natural morphisms
$$
 pr_{k|\mathcal{X}(n,m)}:\mathcal{X}(n,m)\subseteq\bigtimes_{k=1}^{n+1} \mathbb{P}(\Hom(\bigwedge^{k}W,\bigwedge^{k}V))\rightarrow \mathbb{P}(\Hom(\bigwedge^{k}W,\bigwedge^{k}V))
$$
given by the restrictions of the projections onto the factors. By Theorem \ref{theff} the sections of the linear system $|D_k|$ are the strict transforms of the hypersurfaces of $\mathbb{P}^N$ defined by the $k\times k$ minors of the matrix $Z$ in (\ref{matrix}). That is, we have a commutative diagram 
  \[
  \begin{tikzpicture}[xscale=3.5,yscale=-1.5]
    \node (A0_0) at (0, 0) {$\mathcal{X}(n,m)$};
    \node (A1_0) at (0, 1) {$\mathbb{P}^N$};
    \node (A1_1) at (1, 1) {$\mathbb{P}(\Hom(\bigwedge^{k}W,\bigwedge^{k}V))$};
    \path (A1_0) edge [->,dashed]node [auto] {$\scriptstyle{}$} (A1_1);
    \path (A0_0) edge [->,swap]node [auto] {$\scriptstyle{f}$} (A1_0);
    \path (A0_0) edge [->]node [auto] {$\scriptstyle{pr_{k|\mathcal{X}(n,m)}}$} (A1_1);
  \end{tikzpicture}
  \]
where $\mathbb{P}^N\rightarrow\mathbb{P}(\Hom(\bigwedge^{k}W,\bigwedge^{k}V))$ is the rational map defined by the $k\times k$ minors of the matrix $Z$ in (\ref{matrix}). Therefore, $\pi_k = pr_{k|\mathcal{X}(n,m)}$ is the morphism induced by $D_k$.

Note that if $k =1,\dots, n$ then by Theorem \ref{theff} $D_k$ is nef and big. Furthermore if $1\leq k\leq n$, two $(n+1)\times (m+1)$ matrices of maximal rank having all the $(k \times k)$-minors equal, differ for some non-zero multiplicative scalar. Therefore, $\pi_k$ is birational onto its image for $k = 1,\dots,n$. 

Since by Theorem \ref{theff} $D_{n+1}$ generates an extremal ray of both $\Nef(\mathcal{X}(n,m))$ and $\Eff(\mathcal{X}(n,m))$ when $n<m$ we get that $\pi_{n+1}$ is a fibration.

Finally, note that since $\pi_{n+1}$ is given by the maximal minors of a general matrix $Z$ as in (\ref{matrix}) we have the following commutative diagram 
\[
  \begin{tikzpicture}[xscale=5.5,yscale=-1.5]
    \node (A0_0) at (0, 0) {$\mathcal{X}(n,m)$};
    \node (A1_0) at (0, 1) {$\mathcal{G}(n,m)$};
    \node (A1_1) at (1, 1) {$\mathbb{P}(\Hom(\bigwedge^{n+1}W,\bigwedge^{n+1}V))\cong\mathbb{P}(\bigwedge^{n+1}W^{*})$};
    \path (A1_0) edge [->,dashed]node [auto] {$\scriptstyle{pl}$} (A1_1);
    \path (A0_0) edge [->,dashed]node [auto] {$\scriptstyle{ }$} (A1_0);
    \path (A0_0) edge [->]node [auto] {$\scriptstyle{\pi_{n+1}}$} (A1_1);
  \end{tikzpicture}
  \]
where the vertical rational map $\mathcal{X}(n,m)\dasharrow \mathcal{G}(n,m)$ associates to a full rank matrix $Z\in\mathcal{X}(n,m)$ the $n$-plane in $\mathbb{P}^m$ generated by its rows, and the morphism $pl:\mathcal{G}(n,m)\rightarrow\mathbb{P}(\bigwedge^{n+1}W^{*})$ is the Pl\"ucker embedding.
\end{proof}

Now, let $Z:W\rightarrow V$ be a linear map. The linear map $\wedge^kZ:\bigwedge^{k}W\rightarrow\bigwedge^{k}V$ defines a bihomogeneous polynomial of bidegree $(1,1)$ on $\mathbb{P}(\bigwedge^{k}W^{*})\times \mathbb{P}(\bigwedge^{k}V)$, which in turn corresponds to an element $D_{Z,k}$ of the linear system $|\mathcal{O}_{\mathcal{G}(k-1,m)\times\mathcal{G}(k-1,n)}(1,1)|$, where $\mathcal{G}(k-1,m)\times\mathcal{G}(k-1,n)$ is embedded in $\mathbb{P}(\bigwedge^{k}W^{*})\times \mathbb{P}(\bigwedge^{k}V)$ via the product of the Pl\"ucker embeddings of $\mathcal{G}(k-1,m)$ and $\mathcal{G}(k-1,n)$. Recalling that $\Hom(\bigwedge^kW,\bigwedge^kV)\cong \bigwedge^{k}W^{*}\otimes\bigwedge^kV$ we may rewrite the morphism $\pi_k$ in Proposition \ref{morNEF} as follows
$$
\begin{array}{ccc}
\pi_k:\mathcal{X}(n,m) & \longrightarrow & \mathbb{P}(\bigwedge^{k}W^{*}\otimes\bigwedge^kV)\\
Z & \longmapsto & D_{Z,k}
\end{array}
$$

\begin{Definition}\label{degpair}
We define the \textit{variety of degeneration pairs} $Deg_{k}(\mathcal{X}(n,m))$ as the image of the morphism $\pi_k$.
\end{Definition}

Now, arguing as in the proof of \cite[Lemma 5]{Ce15} we have the following.

\begin{Proposition}\label{modNEF}
Let $Z:W\rightarrow V$ be a linear map of maximal rank, and $([L_1],[L_2])\in  \mathcal{G}(k-1,m)\times\mathcal{G}(k-1,n)$ be a point corresponding to two $k$-dimensional vector subspaces $L_1\subseteq W^{*}$, $L_2\subseteq V$. Then $D_{Z,k}([L_1],[L_2])=0$ if and only if $Z_{|L_1\times L_2}$ is singular.
\end{Proposition}
\begin{proof}
Let $\overline{Z} :=Z_{|L_1\times L_2}$. We have that $D_{Z,k}([L_1],[L_2])=0$ if and only if $L_1^{t}\wedge^{k}ZL_2 = \wedge^k(L_1^tZL_2) = \wedge^{k}\overline{Z} = \det(\overline{Z})=0$.
\end{proof}

Proposition \ref{modNEF} tells us that $Deg_{k}(\mathcal{X}(n,m))$ carries
a modular interpretation. Indeed, a general point of $Deg_{k}(\mathcal{X}(n,m))$ corresponds to a bihomogeneous polynomial on $\mathcal{G}(k-1,m)\times\mathcal{G}(k-1,n)$ whose zero locus consist of pairs of $k$-dimensional vector subspaces on which a linear map of maximal rank $Z\in W^{*}\otimes V$ degenerates. 

\begin{Lemma}\label{tec1}
Let $D$ be a base-point-free divisor on a normal scheme $X$. Assume that the morphism $f_D:X\rightarrow \mathbb{P}(H^{0}(X,\mathcal{O}_X(D))^{*})$ induced by $D$ is birational onto its image $f_D(X)$. Then the model $X(D)=\Proj(\bigoplus_{r\geq 0}H^0(X,\mathcal{O}_X(rD)))$ is the normalization of $f_D(X)$.
\end{Lemma}
\begin{proof}
The model $X(D)$ is the image of the morphism $f_{sD}:X\rightarrow\mathbb{P}(H^{0}(X,\mathcal{O}_X(sD))^{*})$ for $s\gg 0$. If $C\subseteq X$ is an irreducible curve we have $D\cdot C =0$ if and only if $sD\cdot C =0$. Therefore, $f_D$ and $f_{sD}$ contract the same irreducible curves in $X$, and hence there exists a finite morphism $g:X(D)\rightarrow f_D(X)$ such that $h\circ f_{sD} = f_D$.

Since $X$ is normal we have that $X(D)$ is also normal. Furthermore, since $f_D$ is birational $g$ is birational as well. 

Now, since $X(D)$ is normal $g:X(D)\rightarrow f_D(X)$ factors through the normalization $\nu:f_D(X)^{\nu}\rightarrow f_D(X)$, and we get a morphism $h:X(D)\rightarrow f_D(X)^{\nu}$. Finally, since $g$ is birational and finite $h$ is birational and finite as well, and since $X(D)$ and $f_D(X)^{\nu}$ are normal Zariski's main theorem \cite[Corollary 11.4]{Har77} yields that $h:X(D)\rightarrow f_D(X)^{\nu}$ is an isomorphism. 
\end{proof}

\subsubsection*{Higher order dual varieties}
We will denote by $\Mor_d(\mathbb{P}^1,\mathbb{P}^n)$ the quasi-projective subvariety of $\mathbb{P}^{n(d+1)+d}$ parametrizing morphisms $f:\mathbb{P}^1\rightarrow\mathbb{P}^n$, $f(u,v) = [f_0(u,v):\dots:f_{n}(u,v)]$, where $f_0,\dots,f_n$ are degree $d$ homogeneous polynomials in two variables with no non-constant common factors.

For a smooth point $x\in X\subset\mathbb{P}^N$ of an irreducible variety, the \textit{$r$-osculating space} $T_x^{r}X$ of $X$ at $x$ is roughly speaking the smaller linear subspace where $X$ can be locally approximated up to order $r$ at $x$. 

\begin{Definition}\label{oscdef}
Let $X\subset \P^N$ be a projective variety of dimension $n$, and $x\in X$ a smooth point.
Choose a local parametrization of $X$ at $x$:
$$
\begin{array}{cccc}
\phi: &\mathcal{U}\subseteq\mathbb{C}^n& \longrightarrow & \mathbb{C}^{N}\\
      & (t_1,\dots,t_n) & \longmapsto & \phi(t_1,\dots,t_n) \\
      & 0 & \longmapsto & x 
\end{array}
$$
For a multi-index $I = (i_1,\dots,i_n)$, set $\phi_I = \frac{\partial^{|I|}\phi}{\partial t_1^{i_1}\dots\partial t_n^{i_n}}$. For any $r\geq 0$, let $O^r_xX$ be the affine subspace of $\mathbb{C}^{N}$ centered at $x$ and spanned by the vectors $\phi_I(0)$ with  $|I|\leq r$. The $r$-\textit{osculating space} $T_x^r X$ of $X$ at $x$ is the projective closure of  $O^r_xX$ in $\mathbb{P}^N$.
\end{Definition}
Note that $T_x^0 X=\{x\}$, and $T_x^1 X$ is the usual tangent space of $X$ at $x$. Consider the incidence variety
\[
  \begin{tikzpicture}[xscale=1.5,yscale=-1.4]
    \node (A0_1) at (1, 0) {$\mathcal{I}_{r}= \{(x,H)\: |\: x\in X,\: H\supseteq T_x^{r}X\}\subset\mathbb{P}^N\times\mathbb{P}^{N*}$};
    \node (A1_0) at (0, 1) {$\mathbb{P}^{N}$};
    \node (A1_2) at (2, 1) {$\mathbb{P}^{N*}$};
    \path (A0_1) edge [->]node [auto] {$\scriptstyle{\pi_2^r}$} (A1_2);
    \path (A0_1) edge [->,swap]node [auto] {$\scriptstyle{\pi_1^r}$} (A1_0);
  \end{tikzpicture}
\]
Clearly, $\pi_1^r(\mathcal{I}_r)=X$ while $\pi_2^r(\mathcal{I}_r)=X_r^{*}\subset\mathbb{P}^{N*}$ is the \textit{$r$-dual variety} of $X$. Note that $X_1^{*} = X^{*}$ is the usual dual variety of $X$. Similarly, we may consider the incidence variety
\[
  \begin{tikzpicture}[xscale=1.5,yscale=-1.4]
    \node (A0_1) at (1, 0) {$\mathcal{I}_{r}^{*}= \{(x,H)\: |\: x\in X_r^{*},\: H\supseteq T_x^{r}X_r^{*}\}\subset\mathbb{P}^{N*}\times\mathbb{P}^{N}$};
    \node (A1_0) at (0, 1) {$\mathbb{P}^{N*}$};
    \node (A1_2) at (2, 1) {$\mathbb{P}^{N}$};
    \path (A0_1) edge [->]node [auto] {$\scriptstyle{\pi_2^{r*}}$} (A1_2);
    \path (A0_1) edge [->,swap]node [auto] {$\scriptstyle{\pi_1^{r*}}$} (A1_0);
  \end{tikzpicture}
\]
We have that $\pi_1^{r*}(\mathcal{I}_r^{*})=X_r^{*}$ while $\pi_2^r(\mathcal{I}_r^{*})=(X_r^{*})_{r}^{*}\subset(\mathbb{P}^{N*})^{*}=\mathbb{P}^{N}$ is the $r$-dual variety of $X_r^{*}$. 

\begin{Lemma}\label{dualcurve}
Let $C\subset\mathbb{P}^N$ be an irreducible non-degenerated curve. Then $\dim(\mathcal{I}_r) = N-r = \dim(C_r^{*}) = N-r$. 
\end{Lemma}
\begin{proof}
Let $x\in C$ be a smooth point. Since $C$ is a non-degenerated curve we have $\dim(T^r_xC) = r$ for $r=0,\dots,N$. Note that $\pi_1^{-1}(x)$ parametrizes the hyperplanes containing $T^r_xC$. Hence $\dim(\pi_1^{-1}(x)) = N-r-1$ and $\dim(\mathcal{I}_r) = N-r$. 

Now, fix a general point $H\in C_r^{*}$. Then $\pi_2^{-1}(H)$ consists of the points $x\in C$ such that $H\supset T^r_xC$. Again since $C$ is a non-degenerated curve $H\supset T^r_xC$ just for finitely many points of $x\in C$. Therefore, $\dim(\pi_2^{-1}(H)) = 0$ and $\dim(C_r^{*}) = N-r$.     
\end{proof}

\subsection*{Complete collineations and parametrized rational normal curves}
Consider the standard rational normal curve $C\subset\mathbb{P}^n$ defined by the image of the map $\gamma:\mathbb{P}^1\rightarrow\mathbb{P}^n$, $\gamma(u,v) = [u^n:u^{n-1}v:\dots:uv^{n-1},v^{n}]$. We can associate to an $(n+1)\times (n+1)$ matrix $Z=(z_{i,j})$ the parametrized rational normal curve $C_Z:=Z\cdot C$ given by the image of the morphism $\gamma_Z:\mathbb{P}^1\rightarrow\mathbb{P}^n$, $\gamma_{Z}(u,v) = [\gamma_{Z,0}(u,v):\dots:\gamma_{Z,n}(u,v)]$, where $\gamma_{Z,i}(u,v) = z_{i,0}u^n+z_{i,1}u^{n-1}v+\dots+z_{i,n}v^n$. 

This association gives a one-to-one correspondence between the open subset of $\mathbb{P}(\Hom(V,V))$ parametrizing matrices of maximal rank and the subset of $\Mor_n(\mathbb{P}^1,\mathbb{P}^n)$ of smooth parametrized degree $n$ rational curves. 

\begin{Proposition}\label{proprnc}
Let $Z\in \mathcal{X}(n)$ be a matrix of maximal rank. For any $r=0,\dots,n-1$ the image of the morphism
$$
\begin{array}{ccc}
\gamma_{Z}^r:\mathbb{P}^1 & \longrightarrow & \mathcal{G}(r,n)\subset\mathbb{P}(\bigwedge^{r+1}V)\\
 (u,v) & \longmapsto & T^r_{\gamma_{Z}(u,v)}C_Z
\end{array}
$$
is a rational normal curve of degree $(r+1)(n-r)$ in the Grassmannian $\mathcal{G}(r,n)\subset\mathbb{P}(\bigwedge^{r+1}V)$. 

The rational map 
$$
\begin{array}{ccc}
g^r:\mathcal{X}(n) & \dashrightarrow & \Mor_{(r+1)(n-r)}(\mathbb{P}^1,\mathbb{P}(\bigwedge^{r+1}V))\\
 Z & \longmapsto & \gamma_{Z}^r
\end{array}
$$
is induced by a sub linear system of the complete linear system of the nef divisor $D_{r+1}$ in Theorem \ref{theff}, and it is regular and bijective onto its image on $\mathcal{X}(n)\setminus E_1\cup \dots\cup E_n$.
\end{Proposition}
\begin{proof}
Since $\gamma_Z^r$ can be obtained by translating the image of $\gamma_{Id}^r$ via $\wedge^{r+1}Z$ it is enough to prove that $\gamma_{Id}^r(\mathbb{P}^1)\subset\mathbb{P}(\bigwedge^{r+1}V)$ is a rational normal curve of degree $(r+1)(n-r)$.  

Consider the affine chart $v=1$ and the matrix
$$
D_{\gamma}^r=\left(\begin{array}{ccccc}
u^n & u^{n-1} & \dots & u & 1\\ 
c_{1,0}u^{n-1} & c_{1,1}u^{n-2} & \dots & 1 & 0 \\ 
\vdots & \vdots & \ddots & \vdots & \vdots\\ 
c_{r,0}u^{n-r} & c_{r,1}u^{n-r-1} & \dots & 0 & 0
\end{array}\right)
$$
Then the map $\gamma_{Id}^r$ is given by the $(r+1)\times(r+1)$ minors of $D_{\gamma}^{r}$ which are polynomials of degree
$$n+(n-2)+(n-4)+\dots +n-2r+\sum_{i=0}^r n-2i = (r+1)(n-r)$$
in $u$. For a general matrix $Z\in\mathcal{X}(n)$ consider the corresponding matrix $D_{\gamma_Z}^r$. By the previous description the coefficients of the coordinate functions of the morphism $\gamma_Z^r$ are linear combinations of the $(r+1)\times (r+1)$ minors of $Z$. Now, it is enough to observe that by Proposition \ref{p1} these minors are sections of the complete linear system of the nef divisor $D_{r+1}$. Finally, to see that $g^r$ is well-defined on $\mathcal{X}(n)\setminus E_1\cup \dots\cup E_n$ it is enough to observe that $\gamma_Z^r$ is defined for any matrix $Z$ of maximal rank.
\end{proof}

\begin{Remark}
Proposition \ref{proprnc} generalizes the following classical fact: the tangent lines and the osculating planes of a twisted cubic in $\mathbb{P}^3$ are parametrized respectively by a rational normal quartic in $\mathcal{G}(1,3)\subset\mathbb{P}^5$, and another twisted cubic in $\mathbb{P}^{3*}$. 
\end{Remark}

\begin{Proposition}\label{bidual}
Let $C\subset\mathbb{P}^n$ be a degree $n$ rational normal curve. Then $C_{r}^{*}\subset\mathbb{P}^{n*}$ is non-degenerated for $r = 1,\dots,n-1$, and $T^r_xC_{r}^{*}\subset\mathbb{P}^{n*}$ is a hyperplane for $x\in C_{r}^{*}$ general and $r = 1,\dots,n-1$. Furthermore $(C_{r}^{*})^{*}_{r} = C$ for $r = 1,\dots,n-1$.
\end{Proposition}
\begin{proof}
We may assume that $C$ is the standard rational normal curve given by the image of $\gamma$. Note that $C_{n-1}^{*}\subset\mathbb{P}^{n*}$ is the rational normal curve defined by the image of the morphism $\gamma^{n-1}_{Id}$ in Proposition \ref{proprnc}. Therefore, $C_{n-1}^{*}\subseteq\mathbb{P}^{n*}$ is non-degenerated and since 
$$C_{n-1}^{*}\subset C_{n-2}^{*}\subset \dots \subset C_1^{*}\subset\mathbb{P}^{n*}$$
we conclude that $C_{n-2}^{*}\subset \dots \subset C_1^{*}\subset\mathbb{P}^{n*}$ are non-degenerated as well. 

Now, Lemma \ref{dualcurve} yields that $\dim(C_r^{*}) = n-r$. Therefore, for a general point $x\in C_r^{*}$ we have $\dim(T^1_x C_r^{*}) = n-r$ and $T^1_xC_r^{*}\subsetneqq T^2_xC_r{*}\subsetneqq T^r_xC_r^{*}$. This yields $\dim(T^r_xC_r^{*})\geq n-1$ and hence $\dim(T^r_xC_r^{*})\geq n-1$. In particular, the inverse image of $x\in C_r^{*}$ via $\pi_1^{r*}$ has dimension zero and $\dim(\mathcal{I}_r^{*}) = \dim (C_r^{*}) = n-r$. Finally, by Lemma \ref{dualcurve} we have that $\dim(\mathcal{I}_r^{*}) = \dim(\mathcal{I}_r)$, and then \cite[Proposition 1]{Pi83} yields $(C_{r}^{*})^{*}_{r} = C$. 
\end{proof}

\begin{Proposition}\label{birat}
There exists a rational linear projection 
$$\pi:\mathbb{P}(\Hom(\bigwedge^{r+1}V,\bigwedge^{r+1}V))\dasharrow \Mor_{(r+1)(n-r)}(\mathbb{P}^1,\bigwedge^{r+1}V)\subset\mathbb{P}^{N_r}$$
making commutative the following diagram
  \[
  \begin{tikzpicture}[xscale=4.9,yscale=-1.4]
    \node (A0_0) at (0, 0) {$\mathcal{X}(n)$};
    \node (A0_1) at (1, 0) {$\mathbb{P}(\Hom(\bigwedge^{r+1}V,\bigwedge^{r+1}V))$};
    \node (A1_1) at (1, 1) {$\Mor_{(r+1)(n-r)}(\mathbb{P}^1,\bigwedge^{r+1}V)\subset\mathbb{P}^{N_r}$};
    \path (A0_0) edge [->]node [auto] {$\scriptstyle{pr_{r+1|\mathcal{X}(n)}}$} (A0_1);
    \path (A0_1) edge [->,dashed]node [auto] {$\scriptstyle{\pi}$} (A1_1);
    \path (A0_0) edge [->,dashed]node [auto] {$\scriptstyle{g^r}$} (A1_1);
  \end{tikzpicture}
  \]
where $N_r = \binom{n+1}{r+1}((r+1)(n-r)+1)$. Furthermore, $\pi$ is regular and a bijection onto its image on $pr_{r+1|\mathcal{X}(n)}(\mathcal{X}(n)\setminus E_1\cup\dots\cup E_n)$.
\end{Proposition}
\begin{proof}
Since by Proposition \ref{proprnc} $g^r$ is induced by a sub linear system of the complete linear system $|D_{r+1}|$ inducing $pr_{r+1|\mathcal{X}(n)}$ the existence of the linear projection $\pi$ follows. Furthermore, the claim on its regularity on $pr_{r+1|\mathcal{X}(n)}(\mathcal{X}(n)\setminus E_1\cup\dots\cup E_n)$ is a direct consequence of the last part of Proposition \ref{proprnc}. 

Note that if $C\subset\mathbb{P}^n$ is a degree $n$ rational normal curve then for any $x\in C$ the $r$-osculating space of $C$ at $x$ intersects $C$ just in $x$ with multiplicity $r+1$. In order to show this claim assume there is another point $y\in T_x^{r}C\cap C$, take $x_1,\dots,x_{n-r-1}\in C$ general, and consider the linear span $H = \left\langle T_x^rC,x_1,\dots,x_{n-r-1}\right\rangle$. Since $\dim(T_x^{r}C) = r$ we have $\dim(H) = r+n-r-1 = n-1$, and $H$ intersects $C$ in at least $(r+1)+1+n-r-1 = n+1$ points counted with multiplicity. Now, $\deg(C) = n$ forces $C\subset H$. A contradiction, since $C\subset\mathbb{P}^n$ is non-degenerated. 

Now, let $Z, Z'$ be two matrices of maximal rank such that $\gamma^r_{Z} = \gamma^r_{Z'}$. Then any $r$-osculating plane of $C = \gamma_{Z}(\mathbb{P}^1)$ is an $r$-osculating plane of $\Gamma = \gamma_{Z'}(\mathbb{P}^1)$ as well. Therefore, $C_{r}^{*} = \Gamma_r^{*}$ and Proposition \ref{bidual} yields $C = (C_{r}^{*})_r^{*} = (\Gamma_r^{*})_{r}^{*} = \Gamma$. Then $\gamma_Z,\gamma_{Z'}:\mathbb{P}^1\rightarrow C = \Gamma\subset\mathbb{P}^n$ have the same image. Now, since $\gamma^r_{Z}(p) = T^r_{\gamma_{Z}(p)}C = T^r_{\gamma_{Z'}(p)}C = \gamma_{Z'}^r(p)$ for any $p\in\mathbb{P}^1$, and for any $x\in C$ the $r$-osculating space $T^r_xC$ intersects $C$ just in $x\in C$ with multiplicity $r+1$, we conclude that $\gamma_{Z}(p) = \gamma_{Z'}(p)$ for any $p\in \mathbb{P}^1$. 
\end{proof}

\begin{thm}\label{ModNEF}
For any $k=0,\dots,n-1$ the model associated to the generator $D_{k+1}$ of the nef cone $\Nef(\mathcal{X}(n,m))$ 
$$\mathcal{X}(n,m)(D_{k+1}) = \Proj(\bigoplus_{r\geq 0}H^0(X,\mathcal{O}_X(rD_{k+1})))\cong Deg_{k+1}(\mathcal{X}(n,m))^{\nu}$$ 
is isomorphic to the normalization of the variety of degeneration pairs $Deg_{k+1}(\mathcal{X}(n,m))$.

Furthermore, if $n = m$ the points of the dense open subset $\mathcal{X}(n)(D_{k+1})\setminus f_{D_{k+1}}(E_1\cup\dots\cup E_n)$ are in bijection with the rational normal curves of degree $(k+1)(n-k)$ in $\mathbb{P}(\bigwedge^{k+1}V)$ whose points parametrize the $k$-osculating planes of a rational normal curve of degree $n$ in $\mathbb{P}^n$.
\end{thm}
\begin{proof}
By Definition \ref{degpair} $Deg_{k+1}(\mathcal{X}(n,m))$ is the image of the morphism $\pi_{k+1}$ induced by $D_{k+1}$, and by Proposition \ref{morNEF} $\pi_{k+1}$ is birational onto $Deg_{k+1}(\mathcal{X}(n,m))$ for $k = 0,\dots, n-1$. Now, to conclude it is enough to apply Lemma \ref{tec1} with $X =\mathcal{X}(n,m)$ and $D=D_{k+1}$. For the second statement just recall that $\pi_{k+1} = pr_{k+1|\mathcal{X}(n)}$ and apply Proposition \ref{birat}.
\end{proof}

\section{On the Mori chamber and stable base locus decompositions}\label{sec4}
In this section we will study the Mori chamber and stable base locus decomposition for spaces of complete forms. Indeed, taking advantage of the main results in Sections \ref{sec2} and \ref{sec3}, mostly of Theorem \ref{gen}, we will relate the decompositions of the space of complete collineations $\mathcal{X}(n)$ and of its subvariety $\mathcal{Q}(n)$ parametrizing complete quadrics. In particular, we will provide a complete description of the Mori chamber and stable base locus decompositions for $\mathcal{X}(3)$. 

\subsubsection*{The stable base locus decomposition} 
The stable base locus of an effective $\mathbb{Q}$-divisor on a normal $\mathbb{Q}$-factorial projective variety $X$ has been defined in (\ref{sbl}). Since stable base loci do not behave well with respect to numerical equivalence \cite[Example 10.3.3]{La04II}, we will assume that $h^{1}(X,\mathcal{O}_X)=0$ so that linear and numerical equivalence of $\mathbb{Q}$-divisors coincide. 

Then numerically equivalent $\mathbb{Q}$-divisors on $X$ have the same stable base locus, and the pseudo-effective cone $\overline{\Eff}(X)$ of $X$ can be decomposed into chambers depending on the stable
base locus of the corresponding linear series called \textit{stable base locus decomposition}. 

If $X$ is a Mori dream space, satisfying then the condition $h^1(X,\mathcal{O}_X)=0$, determining the stable base locus decomposition of $\Eff(X)$ is a first step in order to compute its Mori chamber decomposition. 

\begin{Remark}\label{SBLMC}
Recall that two divisors $D_1,D_2$ are said to be \textit{Mori equivalent} if $\textbf{B}(D_1) = \textbf{B}(D_2)$ and the following diagram of rational maps is commutative
   \[
  \begin{tikzpicture}[xscale=1.5,yscale=-1.2]
    \node (A0_1) at (1, 0) {$X$};
    \node (A1_0) at (0, 1) {$X(D_1)$};
    \node (A1_2) at (2, 1) {$X(D_2)$};
    \path (A1_0) edge [->]node [auto] {$\scriptstyle{}$} node [rotate=180,sloped] {$\scriptstyle{\widetilde{\ \ \ }}$} (A1_2);
    \path (A0_1) edge [->,dashed]node [auto] {$\scriptstyle{\phi_{D_2}}$} (A1_2);
    \path (A0_1) edge [->,swap, dashed]node [auto] {$\scriptstyle{\phi_{D_1}}$} (A1_0);
  \end{tikzpicture}
  \]
where the horizontal arrow is an isomorphism. Therefore, the Mori chamber decomposition is a possibly trivial refinement of the stable base locus decomposition.
\end{Remark}

\begin{Lemma}\label{MCrest}
Let $X,Y$ be normal, projective $\mathbb{Q}$-factorial Mori dream spaces. If there exists an embedding $i:X\rightarrow Y$ such that $i^{*}:\Pic(Y)\rightarrow\Pic(X)$ is an isomorphism inducing an isomorphism $\Eff(Y)\rightarrow\Eff(X)$, then
\begin{itemize}
\item[-] the stable base locus decomposition of $\Eff(Y)$ is a refinement of the stable base locus decomposition of $\Eff(X)$, and the two decompositions coincide outside of the movable cones;
\item[-] the Mori chamber decomposition of $\Eff(Y)$ is a refinement of the Mori chamber decomposition of $\Eff(X)$.
\end{itemize}
\end{Lemma}
\begin{proof}
Let $D_1,D_1\in \Pic(Y)$ be two divisors. If $\textbf{B}(D_1) = \textbf{B}(D_2)$ then $\textbf{B}(i^{*}D_1) = \textbf{B}(D_1)_{|X} = \textbf{B}(D_2)_{|X} = \textbf{B}(i^{*}D_2)$. Now, assume that $D_1,D_2$ are not movable with stable base loci $E_1 = \textbf{B}(D_1),E_2 = \textbf{B}(D_2)$ such that $E_1\neq E_2$. Then $\textbf{B}(i^{*}D_1) = i^{*}E_1,\textbf{B}(i^{*}D_2) = i^{*}E_2$, and since $i^{*}:\Pic(Y)\rightarrow\Pic(X)$ is an isomorphism we conclude that $i^{*}E_1\neq i^{*}E_2$.

If $D_1,D_2$ are Mori equivalent then $\textbf{B}(D_1) = \textbf{B}(D_2)$ and we have the following commutative diagram
  \[
  \begin{tikzpicture}[xscale=2.5,yscale=-1.2]
    \node (A0_2) at (2, 0) {$X$};
    \node (A1_2) at (2, 1) {$Y$};
    \node (A2_0) at (0, 2) {$X(i^{*}D_1)$};
    \node (A2_1) at (1, 2) {$Y(D_1)$};
    \node (A2_3) at (3, 2) {$Y(D_2)$};
    \node (A2_4) at (4, 2) {$X(i^{*}D_2)$};
    \path (A0_2) edge [->,swap,dashed]node [auto] {$\scriptstyle{\phi_{i^{*}D_1}}$} (A2_0);
    \path (A0_2) edge [->,dashed]node [auto] {$\scriptstyle{\phi_{i^{*}D_2}}$} (A2_4);
    \path (A1_2) edge [->,dashed]node [auto] {$\scriptstyle{\phi_{D_2}}$} (A2_3);
    \path (A1_2) edge [->,swap,dashed]node [auto] {$\scriptstyle{\phi_{D_1}}$} (A2_1);
    \path (A2_4) edge [left hook->]node [auto] {$\scriptstyle{}$} (A2_3);
    \path (A0_2) edge [right hook->]node [auto] {$\scriptstyle{i}$} (A1_2);
    \path (A2_1) edge [->]node [auto] {$\scriptstyle{}$} node [rotate=180,sloped] {$\scriptstyle{\widetilde{\ \ \ }}$} (A2_3);
    \path (A2_0) edge [right hook->]node [auto] {$\scriptstyle{}$} (A2_1);
  \end{tikzpicture}
  \]
By the first part of the proof $\textbf{B}(D_1) = \textbf{B}(D_2)$ yields $\textbf{B}(i^{*}D_1)= \textbf{B}(i^{*}D_2)$. Furthermore, the isomorphism $Y(D_1)\rightarrow Y(D_2)$ restricts to an isomorphism $X(i^{*}D_1)\cong X(i^{*}D_1)$. 
\end{proof}

The spaces $\mathcal{X}(1)$, $\mathcal{Q}(1)$ have Picard number one, so there is nothing to say on their Mori chamber decomposition. Those of Picard number two are $\mathcal{X}(1,m)$ with $m\geq 2$, $\mathcal{X}(2)$, $\mathcal{Q}(2)$. Their Mori chamber decomposition coincides with their stable base locus decomposition. For instance, the Mori chamber decomposition of $\Eff(\mathcal{X}(2))$ is as follows
$$
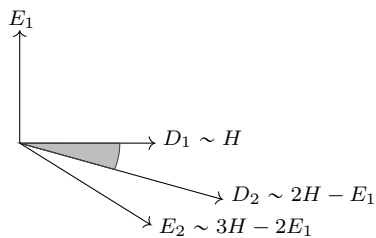
\begin{tikzpicture}[xscale=1.8,yscale=1.5][line cap=round,line join=round,>=triangle 45,x=1cm,y=1cm]\clip(-3.5,-0.82) rectangle(7.8,1.2);\draw [->,line width=0.4pt] (0,0) -- (0,1);\draw [->,line width=0.4pt] (0,0) -- (1,0);\draw [->,line width=0.4pt] (0,0) -- (1.5,-0.5);\draw [shift={(0,0)},line width=0.4pt,color=uququq,fill=uququq,fill opacity=0.34]  (0,0) --  plot[domain=-0.32175055439664213:0,variable=\t]({1*0.7391605118416497*cos(\t r)+0*0.7391605118416497*sin(\t r)},{0*0.7391605118416497*cos(\t r)+1*0.7391605118416497*sin(\t r)}) -- cycle ;\draw [->,line width=0.4pt] (0,0) -- (0.9720430107526882,-0.7277008310249308);\begin{scriptsize}\draw [fill=black] (0,1) circle (0.0pt);\draw[color=black] (0.015053763440860174,1.1) node {$E_1$};\draw [fill=black] (1,0) circle (0.0pt);\draw[color=black] (1.35,0.038227146814404464) node {$D_1\sim H$};\draw [fill=black] (1.5,-0.5) circle (0.0pt);\draw[color=black] (2.1,-0.4631578947368421) node {$D_2\sim 2H-E_1$};\draw [fill=black] (0.9720430107526882,-0.7277008310249308) circle (0.0pt);\draw[color=black] (1.6,-0.75) node {$E_2\sim 3H-2E_1$};\end{scriptsize}
\end{tikzpicture}
$$
while the Mori chamber decomposition of $\Eff(\mathcal{X}(1,m))$ with $m\geq 2$ is obtained from the picture above by removing the ray $E_2\sim 3H-2E_1$, and the Mori chamber decomposition of $\Eff(\mathcal{Q}(2))$ is as the one of $\Eff(\mathcal{X}(2))$. Indeed, if $m\geq 2$ Theorem \ref{theff} yields that $\Eff(\mathcal{X}(1,m)) = \left\langle E_1,D_2\right\rangle$ and $\Nef(\mathcal{X}(1,m)) = \left\langle D_1,D_2\right\rangle$. In this case $D_2$ is a ray of both the effective and the nef cone, hence the relative wall crossing corresponds to the fibration induced by the linear system of quadrics vanishing on the Segre variety. 

Furthermore, again by Theorem \ref{theff} we have that $\Eff(\mathcal{Q}(2)) = \left\langle E_1,E_2\right\rangle$ and $\Nef(\mathcal{X}(2)) = \left\langle D_1,D_2\right\rangle$. In this case crossing the ray $D_1$ corresponds to blowing-down $E_1$, while crossing the ray $D_2$ induces the blow-down of $E_2$. In particular $\Mov(\mathcal{X}(2)) = \Nef(\mathcal{X}(2))$.

Therefore, from this point of view the first interesting variety is $\mathcal{X}(3)$, which has a $3$-dimensional N\'eron–Severi space. Before analyzing the decomposition of $\Eff(\mathcal{X}(3))$ we will prove some general results relating the Mori chamber and stable base locus decompositions of $\Eff(\mathcal{X}(n))$ to those of $\Eff(\mathcal{Q}(n))$.

\begin{Proposition}\label{MCXnQn}
The stable base locus decomposition of $\Eff(\mathcal{X}(n))$ is a refinement of the stable base locus decomposition of $\Eff(\mathcal{Q}(n))$, and the two decompositions coincide outside of the movable cones.

Furthermore, the Mori chamber decomposition of $\Eff(\mathcal{X}(n))$ is a refinement of the Mori chamber decomposition of $\Eff(\mathcal{Q}(n))$.
\end{Proposition}
\begin{proof}
By Constructions \ref{ccc}, \ref{ccq} and Remark \ref{sym_anti} we have an embedding $i:\mathcal{Q}(n)\rightarrow\mathcal{X}(n)$. Furthermore, by Proposition \ref{p1} and Theorem \ref{theff} $i^{*}:\Pic(\mathcal{X}(n))\rightarrow\Pic(\mathcal{Q}(n))$ is an isomorphism inducing an isomorphism between $\Eff(\mathcal{X}(n))$ and $\Eff(\mathcal{Q}(n))$. To conclude it is enough to apply Lemma \ref{MCrest}.
\end{proof}

\begin{Remark}\label{toric}
Recall that by \cite[Proposition 2.11]{HK00} given a Mori Dream Space $X$ there is an embedding $i:X\rightarrow \mathcal{T}_X$ into a simplicial projective toric variety $\mathcal{T}_X$ such that $i^{*}:\Pic(\mathcal{T}_X)\rightarrow \Pic(X)$ is an isomorphism inducing an isomorphism $\Eff(\mathcal{T}_X)\rightarrow \Eff(X)$. Furthermore, the Mori chamber decomposition of $\Eff(\mathcal{T}_X)$ is a refinement of the Mori chamber decomposition of $\Eff(X)$. Indeed, if $\Cox(X) \cong \frac{K[T_1,\dots,T_s]}{I}$ where the $T_i$ are homogeneous generators with non-trivial effective $\Pic(X)$-degrees then $\Cox(\mathcal{T}_X)\cong K[T_1,\dots,T_s]$.
\end{Remark}

\begin{Notation}
We will denote by $\left\langle v_1,\dots,v_s\right\rangle$ the cone in $\mathbb{R}^n$ generated by the vectors $v_1,\dots,v_s\in\mathbb{R}^n$. Given two vectors $v_i,v_j$ we set $(v_i,v_j] := \left\langle v_i,v_j\right\rangle\setminus\{v_i\}$ and $(v_i,v_j) := \left\langle v_i,v_j\right\rangle\setminus\{v_i,v_j\}$.
\end{Notation}

Before moving forward to the case of $\mathcal{X}(3)$ we need to investigate the birational geometry of the first blow-up in Construction \ref{ccc}. 

\begin{Proposition}\label{firstbu}
Let $\mathcal{X}(n,m)_1$ be the blow-up of $\mathbb{P}^N$ along the Segre variety $\mathcal{S}\subset\mathbb{P}^N$ in Construction \ref{ccc}. As usual we write $\Pic(\mathcal{X}(n,m)_1) = \mathbb{Z}[H,E_1]$ where $H$ is the pull-back of the hyperplane section of $\mathbb{P}^N$, and $E_1$ is the exceptional divisor.

Then the Mori chamber and the stable base locus decompositions of $\Eff(\mathcal{X}(n,m)_1)$ coincide and are represented by the following picture
$$
\begin{tikzpicture}[xscale=1.8,yscale=1.5][line cap=round,line join=round,>=triangle 45,x=1cm,y=1cm]\clip(-0.20150176678445222,-1.4) rectangle (2.498498233215547,1.1763586956521748);\draw [->,line width=0.4pt] (0,0) -- (1.3400182225885815,0);\draw [->,line width=0.4pt] (0,0) -- (0,1);\draw [->,line width=0.4pt] (0,0) -- (1.2849546368530922,-0.5152151776936973);\draw [->,line width=0.4pt] (0,0) -- (1.0853491385619454,-0.7905331063711422);\draw [->,line width=0.4pt] (0,0) -- (0.7412017277151399,-1.1071487243502038);\draw [->,line width=0.4pt] (0,0) -- (0.3557566275667178,-1.2516906369058625);\draw [shift={(0,0)},line width=0.4pt,dotted]  plot[domain=5.302329138917821:5.653673312285958,variable=\t]({1*0.9849158007199498*cos(\t r)+0*0.9849158007199498*sin(\t r)},{0*0.9849158007199498*cos(\t r)+1*0.9849158007199498*sin(\t r)});\draw [shift={(0,0)},line width=0.4pt,color=yqyqyq,fill=yqyqyq,fill opacity=0.66]  (0,0) --  plot[domain=-0.38133353271222603:0,variable=\t]({1*1.0480259109671113*cos(\t r)+0*1.0480259109671113*sin(\t r)},{0*1.0480259109671113*cos(\t r)+1*1.0480259109671113*sin(\t r)}) -- cycle ;\begin{scriptsize}\draw [fill=uuuuuu] (0,0) circle (0.0pt);\draw [fill=black] (1.3400182225885815,0) circle (0.0pt);\draw[color=black] (1.7,0.04769021739130457) node {$D_1\sim H$};\draw [fill=black] (0,1) circle (0.0pt);\draw[color=black] (0.011749116607773832,1.1) node {$E_1$};\draw [fill=black] (1.2849546368530922,-0.5152151776936973) circle (0.0pt);\draw[color=black] (1.9,-0.4680706521739131) node {$D_2\sim 2H-E_1$};\draw [fill=black] (1.0853491385619454,-0.7905331063711422) circle (0.0pt);\draw[color=black] (1.75,-0.7436141304347827) node {$D_3\sim 3H-2E_1$};\draw [fill=black] (0.7412017277151399,-1.1071487243502038) circle (0.0pt);\draw[color=black] (1.6,-1.0580163043478263) node {$D_{n}\sim nH-(n-1)E_1$};\draw [fill=black] (0.3557566275667178,-1.2516906369058625) circle (0.0pt);\draw[color=black] (1.3,-1.3) node {$D_{n+1}\sim (n+1)H-nE_1$};\end{scriptsize}\end{tikzpicture}
$$
where $\Eff(\mathcal{X}(n,m)_1) = \left\langle E_1,D_{n+1}\right\rangle$, $\Nef(\mathcal{X}(n,m)_1) = \left\langle D_1,D_2\right\rangle$ and $\Mov(\mathcal{X}(n,m)_1) = \left\langle D_1,D_{n+1}\right\rangle$ if $n<m$, while $\Mov(\mathcal{X}(n)_1) = \left\langle D_1,D_{n}\right\rangle$. As usual, the analogous statements, with the obvious modifications, hold for the space $\mathcal{Q}(n)_1$ in Constructions \ref{ccq}.  
\end{Proposition}
\begin{proof}
The statements on the effective and nef cones follow from Proposition \ref{int}. Furthermore, since by Proposition \ref{CoxInt} we know the generators of $\Cox(\mathcal{X}(n,m)_1)$ the extremal rays of the movable cone can be easily computed via \cite[Proposition 3.3.2.3]{ADHL15}.

Furthermore, Proposition \ref{CoxInt} yields that the generators of $\Cox(\mathcal{X}(n,m)_1)$ are sections of the divisors displayed in the picture above. Therefore, by Remark \ref{toric} the Mori chamber decomposition of $\Eff(\mathcal{X}(n,m)_1)$ must be a possibly trivial coarsening of the decomposition displayed in the statement. We will prove that such coarsening is indeed trivial by proving that the displayed decomposition is the stable base locus decomposition of $\Eff(\mathcal{X}(n,m)_1)$. First of all, note that $\textbf{B}(D) = \emptyset$ if and only if $D\in [D_1,D_2]$.

Recall that the ideal of $\sec_h(\mathcal{S})\subset\mathbb{P}^N$ is cut out by the $(h+1)\times (h+1)$ minors of the matrix $Z$ in (\ref{matrix}). Furthermore, $D_{h+1}$ is the strict transform of the hypersurface defined by such a minor. By Lemma \ref{mult} we have $D_{h+1}\sim (h+1)H-hE_1$. Furthermore, by Theorem \ref{theff} we know that $D_{h+1}$ becomes base-point-free once the strict transform of $\sec_h(\mathcal{S})$ has been blown-up. Therefore, for any $h\geq 2$ we have that 
\stepcounter{thm}
\begin{equation}\label{blsec}
\textbf{B}(D_{h+1}) = \widetilde{\sec_h(\mathcal{S})}
\end{equation}
where $\widetilde{\sec_h(\mathcal{S})}$ denotes the strict transform of $\sec_h(\mathcal{S})\subset\mathbb{P}^N$ in $\mathcal{X}(n,m)_1$. 

Now, note that if $D$ is a $\mathbb{Q}$-divisor in $[E_1,D_1)$ then $\textbf{B}(D)\subset E$. Furthermore, if $e$ is a curve generating the extremal ray of $\NE(\mathcal{X}(n,m)_1)$ corresponding to the blow-down $\mathcal{X}(n,m)_1\rightarrow\mathbb{P}^N$ we have $D\cdot e < 0$, and since the curves of class $e$ cover $E$ we get that $E_1\subset \textbf{B}(D)$. Therefore, $\textbf{B}(D) = E_1$ for any $D\in [E_1,H)$.     

Now, let $D_{b_1} = H + b_1 E_1$, $D_{b_2} = H + b_2E_1$ be effective $\mathbb{Q}$-divisors in $\mathcal{X}(n,m)_1$ such that $b_2\leq b_1\leq 0$. Note that we can write
$$D_{b_2} = D_{b_1}+(b_2-b_1)E_1$$
with $b_2-b_1 \leq 0$. Therefore
\stepcounter{thm}
\begin{equation}\label{sbl1}
\textbf{B}(D_{b_1})\subset \textbf{B}(D_{b_2})
\end{equation}
Consider a divisor $D\in (D_h,D_{h+1}]$. Therefore, $D\sim H+bE_1$ with $-\frac{h}{h+1} \leq b <-\frac{h-1}{h}$. Now, consider a general point $p\in \sec_h(\mathcal{S})$. Then there exist $h$ points $x_1,\dots,x_h\in \mathcal{S}$ such that $p$ lies in the linear span $H_x\cong\mathbb{P}^{h-1}$ of $x_1,\dots,x_h$, and hence there is a rational normal curve $C$ of degree $h-1$ in $H_x$ passing through $x_1,\dots, x_h, p$. This says that if $\widetilde{C}$ is the strict transform of $C$ in $\mathcal{X}(n,m)_1$ then the curves in $\mathcal{X}(n,m)_1$ of class $\widetilde{C}$ cover $\widetilde{\sec_h(\mathcal{S})}$. Furthermore, $\widetilde{C}\sim (h-1)l-he$, where $l$ denotes the pull-back of a general line in $\mathbb{P}^N$, and 
$$D\cdot \widetilde{C} = h-1+bh <0$$
Therefore, $\widetilde{\sec_h(\mathcal{S})}\subseteq\textbf{B}(D)$, and (\ref{blsec}), (\ref{sbl1}) yield that $\textbf{B}(D) = \widetilde{\sec_h(\mathcal{S})}$ for any divisor $D\in (D_h,D_{h+1}]$. Finally, the decomposition displayed in the statement is both the Mori chamber and the stable base locus decomposition of $\Eff(\mathcal{X}(n,m)_1)$.   
\end{proof}

Let $X$ be a $\mathbb Q$-factorial Mori dream space with divisor class group ${\rm Cl}(X)$ of rank two, and $\lambda_0\leq\lambda_1<\dots <\lambda_{s}\leq\lambda_{s+1}$ be the chambers of the Mori chamber decomposition of $\Eff(X)$ where $\lambda_1$ and $\lambda_{s}$ are respectively the first and the last chamber of $\Mov(X)$. Let $X(\lambda_h)$ be the $\mathbb{Q}$-factorial Mori dream space corresponding to the chamber $\lambda_h$, and let us assume that $X = X(\lambda_1)$ corresponds to the first chamber of $\Mov(X)$. The Mori cone of $X(\lambda_h)$ is generated by two extremal rays and we will denote by $\alpha_{h}:X(\lambda_h)\rightarrow Y_i$ and $\beta_h:X(\lambda_h)\rightarrow Y_{h-1}$ the corresponding extremal contractions. Since $X$ is a Mori dream space for any $h = 2,\dots s$ there exists the flip of $\alpha_{h}$ and we have the following diagram   
  \[
  \begin{tikzpicture}[xscale=1.7,yscale=-1.2]
    \node (A0_1) at (1, 0) {$X = X(\lambda_1)$};
    \node (A0_3) at (3, 0) {$X(\lambda_2)$};
    \node (A0_4) at (4, 0) {$\dots$};
    \node (A0_5) at (5, 0) {$X(\lambda_{s-1})$};
    \node (A0_7) at (7, 0) {$X(\lambda_s)$};
    \node (A1_0) at (0, 1) {$X(\lambda_0)$};
    \node (A1_2) at (2, 1) {$Y_1$};
    \node (A1_6) at (6, 1) {$Y_{s-1}$};
    \node (A1_8) at (8, 1) {$X(\lambda_{s+1})$};
    \path (A0_1) edge [->,swap]node [auto] {$\scriptstyle{\alpha_1}$} (A1_2);
    \path (A0_1) edge [->]node [auto] {$\scriptstyle{\beta_1}$} (A1_0);
    \path (A0_3) edge [->,dashed]node [auto] {$\scriptstyle{\chi_2}$} (A0_4);
    \path (A0_3) edge [->]node [auto] {$\scriptstyle{\beta_2}$} (A1_2);
    \path (A0_5) edge [->,swap]node [auto] {$\scriptstyle{\alpha_{s-1}}$} (A1_6);
    \path (A0_7) edge [->,swap]node [auto] {$\scriptstyle{\alpha_s}$} (A1_8);
    \path (A0_4) edge [->,dashed]node [auto] {$\scriptstyle{\chi_{s-2}}$} (A0_5);
    \path (A0_7) edge [->]node [auto] {$\scriptstyle{\beta_s}$} (A1_6);
    \path (A0_5) edge [->,dashed]node [auto] {$\scriptstyle{\chi_{s-1}}$} (A0_7);
    \path (A0_1) edge [->,dashed]node [auto] {$\scriptstyle{\chi_1}$} (A0_3);
  \end{tikzpicture}
  \]
The process leading from $X = X(\lambda_1)$ to $X(\lambda_{s+1})$ is a special type of minimal model program called a $2$-\textit{ray game}. Note that the $2$-ray game may begin with either a divisorial extraction if $\lambda_0\neq \lambda_1$ or a Mori fibration if $\lambda_0 = \lambda_1$, and similarly it may end with either a divisorial contraction if $\lambda_s \neq \lambda_{s+1}$ or a Mori fibration if $\lambda_s = \lambda_{s+1}$. 

Since $X$ is a Mori dream space the minimal model program runs to completion in the Mori category, and the $2$-ray game leads to a link in the sense of Sarkisov \cite{Co95}. Note that the $2$-ray game is entirely determined by the initial step $X = X(\lambda_1)$.

\begin{Proposition}\label{Sarkisov}
In the notations of Proposition \ref{firstbu} let $\mathcal{X}(n,m)_1^{h}$ be the model of $\mathcal{X}(n,m)_1$ corresponding to a divisor $D\in (D_h,D_{h+1})$. If $n<m$ then the rational map $\mathbb{P}^N\dasharrow \mathcal{G}(r,n)$ given by the minors of order $(n+1)$ of a general matrix $Z$ as in Proposition \ref{morNEF} gives rise to a Sarkisov link of type I:
 \[
  \begin{tikzpicture}[xscale=2.3,yscale=-1.4]
    \node (A0_0) at (0, 0) {};
    \node (A0_1) at (1, 0) {$\mathcal{X}(n,m)_1$};
    \node (A0_2) at (2, 0) {$\mathcal{X}(n,m)_1^{n}$};
    \node (A1_0) at (0, 1) {$\mathbb{P}^N$};
    \node (A1_2) at (2, 1) {$\mathcal{G}(n,m)$};
    \path (A1_0) edge [->,dashed]node [auto] {$\scriptstyle{}$} (A1_2);
    \path (A0_1) edge [->]node [auto] {$\scriptstyle{}$} (A1_0);
    \path (A0_1) edge [->,dashed]node [auto] {$\scriptstyle{}$} (A0_2);
    \path (A0_2) edge [->]node [auto] {$\scriptstyle{}$} (A1_2);
    \end{tikzpicture}
  \]
If $n = m$ then the birational involution $i:\mathbb{P}^N\dasharrow \mathbb{P}^{N*}$ given by mapping an invertible matrix $Z$ to its inverse gives rise to a Sarkisov link of type II:
\[
  \begin{tikzpicture}[xscale=2.1,yscale=-1.4]
    \node (A0_1) at (1, 0) {$\mathcal{X}(n)_1$};
    \node (A0_2) at (2, 0) {$\mathcal{X}(n)_1^{n-1}$};
    \node (A1_0) at (0, 1) {$\mathbb{P}^N$};
    \node (A1_3) at (3, 1) {$\mathbb{P}^{N*}$};
    \path (A0_1) edge [->]node [auto] {$\scriptstyle{}$} (A1_0);
    \path (A1_0) edge [->,dashed]node [auto] {$\scriptstyle{}$} (A1_3);
    \path (A0_2) edge [->]node [auto] {$\scriptstyle{}$} (A1_3);
    \path (A0_1) edge [->,dashed]node [auto] {$\scriptstyle{}$} (A0_2);
  \end{tikzpicture}
  \]   
Furthermore, $X(n)_1^{n-1}\cong \mathcal{X}(n)_1$, and the birational involution $i:\mathbb{P}^N\dasharrow \mathbb{P}^{N*}$ lifts to an automorphism $Z^{inv}:\mathcal{X}(n)\rightarrow \mathcal{X}(n)$. Finally, the same statements hold for the variety of complete quadrics $\mathcal{Q}(n)\subset\mathcal{X}(n)$.
\end{Proposition}
\begin{proof}
The description of the maps as Sarkisov links follows from Proposition \ref{morNEF} and the description of the Mori chamber decomposition of $\Eff(\mathcal{X}(n,m)_1)$ in Proposition \ref{firstbu}.

If $n = m$ let us denote by $\mathcal{S}\subset\mathbb{P}^N$ the Segre variety in the source $\mathbb{P}^N$, and by $\mathcal{S}'\subset\mathbb{P}^N$ the Segre variety in the target $\mathbb{P}^{N*}$. Note that $i:\mathbb{P}^N\dasharrow\mathbb{P}^{N*}$ contracts $\sec_n(\mathcal{S})$ onto $\mathcal{S}'$, and this is the only divisor contracted by $i$. Therefore, if $Y$ is the blow-up of $\mathbb{P}^{N*}$ along $\mathcal{S}'$ we get that $i$ lifts to a small $\mathbb{Q}$-factorial transformation $\mathcal{X}(n)\dasharrow Y$. Note that the divisor that gets contracted by crossing the wall $D_n$ in Proposition \ref{firstbu} is exactly the strict transform of $\sec_n(\mathcal{S})$ through the sequence of flips leading from $\mathcal{X}(n,m)_1$ to $\mathcal{X}(n)_1^{n-1}$. Therefore, $Y$ must be the model $\mathcal{X}(n)_1^{n-1}$ corresponding to the second last chamber in Proposition \ref{firstbu}.

Finally, note that $\mathcal{X}(n)$ has a distinguished automorphism given by restricting the automorphism of the product (\ref{ambient}) switching $\mathbb{P}(\Hom(\bigwedge^{k}W,\bigwedge^{k}V))$ and $\mathbb{P}(\Hom(\bigwedge^{n+1-k}W,\bigwedge^{n+1-k}V))$. Clearly, this automorphism is $Z^{inv}:\mathcal{X}(n)\rightarrow \mathcal{X}(n)$. 
\end{proof}

\begin{Remark}
Note that in the sequence of flips $\mathcal{X}(n,m)_1\dasharrow\mathcal{X}(n,m)_1^n$, and $\mathcal{X}(n)_1\dasharrow\mathcal{X}(n)_1^{n-1}$ in the Sarkisov links of Proposition \ref{Sarkisov} the flipped loci are the strict transforms of the secant varieties of $\mathcal{S}$ in order of increasing dimension. Indeed, by the proof of Proposition \ref{firstbu} we have that the stable base locus of a divisor $D\in(D_h,D_{h+1}]$ is exactly $\widetilde{\sec_h(\mathcal{S})}$.  

The models $\mathcal{X}(n,m)_1^{h}$ arising from Proposition \ref{firstbu} are the varieties constructed as Mumford quotient and through a VGIT argument in \cite[Section 3]{Tha99}. Indeed, in \cite[Theorem 2.3]{Tha99} M. Thaddeus proved that the inverse limit of these varieties is exactly the space of complete collineations.    
\end{Remark}

\begin{thm}\label{MCD_main}
The Mori chamber decomposition of $\mathcal{X}(3)$ consists of nine chambers described in the following $2$-dimensional section of $\Eff(\mathcal{X}(3))$
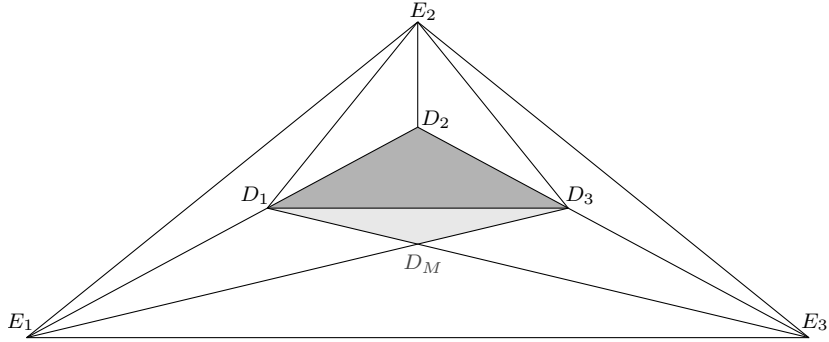
$$
\begin{tikzpicture}[xscale=0.4,yscale=0.7][line cap=round,line join=round,>=triangle 45,x=1cm,y=1cm]\clip(-13.6,-0.23) rectangle (13.9,6.5);\fill[line width=0pt,fill=black,fill opacity=0.3] (-5.000432432432432,2.4614054054054053) -- (5.000432432432432,2.4614054054054053) -- (0,4) -- cycle;\fill[line width=0pt,color=wwwwww,fill=wwwwww,fill opacity=0.15] (-5.000432432432432,2.4614054054054053) -- (5.000432432432432,2.4614054054054053) -- (0,1.7776389756402244) -- cycle;\draw [line width=0.4pt] (-13,0)-- (13,0);\draw [line width=0.4pt] (13,0)-- (0,6);\draw [line width=0.4pt] (0,6)-- (-13,0);\draw [line width=0.4pt] (0,4)-- (-13,0);\draw [line width=0.4pt] (0,4)-- (13,0);\draw [line width=0.4pt] (-5.000432432432432,2.4614054054054053)-- (13,0);\draw [line width=0.4pt] (5.000432432432432,2.4614054054054053)-- (-13,0);\draw [line width=0.4pt] (-5.000432432432432,2.4614054054054053)-- (5.000432432432432,2.4614054054054053);\draw [line width=0.4pt] (-5.000432432432432,2.4614054054054053)-- (0,6);\draw [line width=0.4pt] (0,6)-- (5.000432432432432,2.4614054054054053);\draw [line width=0.4pt] (0,4)-- (0,6);\begin{scriptsize}\draw [fill=black] (-13,0) circle (0pt);\draw[color=black] (-13.2,0.3) node {$E_1$};\draw [fill=black] (13,0) circle (0pt);\draw[color=black] (13.2,0.3) node {$E_3$};\draw [fill=black] (0,6) circle (0pt);\draw[color=black] (0.18536585365853658,6.2) node {$E_2$};\draw [fill=black] (0,4) circle (0pt);\draw[color=black] (0.6,4.143836565096953) node {$D_2$};\draw [fill=black] (-5.000432432432432,2.4614054054054053) circle (0pt);\draw[color=black] (-5.4,2.7) node {$D_1$};\draw [fill=black] (5.000432432432432,2.4614054054054053) circle (0pt);\draw[color=black] (5.4,2.7) node {$D_3$};\draw [fill=uuuuuu] (0,1.7776389756402244) circle (0pt);\draw[color=uuuuuu] (0.18536585365853658,1.4) node {$D_M$};\end{scriptsize}\end{tikzpicture}
$$
where $D_M \sim 6D_1-3E_1-2E_2$, and $\Mov(\mathcal{X}(3)) = \left\langle D_1,D_2,D_3,D_M\right\rangle$. The stable base locus decomposition of $\Eff(\mathcal{X}(3))$ consists of eight chambers and is obtained by removing the wall joining $D_2$ with $E_2$ in the picture above.

Furthermore, the same statements hold, by replacing the relevant divisors with their pull-backs via the embedding $i:\mathcal{Q}(3)\rightarrow \mathcal{X}(3)$, for the space of complete quadrics $\mathcal{Q}(3)$. 
\end{thm}
\begin{proof} 
First of all, note that by Theorem \ref{gen} the sections of $D_1,D_2,D_3,E_1,E_2,E_3$ are homogeneous generators of $\Cox(\mathcal{X}(3))$ with respect to the usual grading on $\Pic(\mathcal{X}(3))$ as displayed in (\ref{gradmat}). Furthermore, by (\ref{mov3}) in the proof of Proposition \ref{PropMov} we have that $\Mov(\mathcal{X}(3)) = \left\langle D_1,D_2,D_3,D_M\right\rangle$, with $D_M \sim 6D_1-3E_1-2E_2$.

Now, let $\mathcal{T}_{\mathcal{X}(3)}$ be a simplicial projective toric variety as in Remark \ref{toric}. Then there is an embedding $i:\mathcal{X}(3)\rightarrow\mathcal{T}_{\mathcal{X}(3)}$ such that $i^{*}:\Pic(\mathcal{T}_{\mathcal{X}(3)})\rightarrow \Pic(\mathcal{X}(3))$ is an isomorphism inducing an isomorphism $\Eff(\mathcal{T}_{\mathcal{X}(3)})\rightarrow \Eff(\mathcal{X}(3))$. Furthermore, if we set $\widetilde{E}_j = i^{*-1}(E_j), \widetilde{D}_j = i^{*-1}(D_j)$ then the sections of $\widetilde{D}_1,\widetilde{D}_2,\widetilde{D}_3,\widetilde{E}_1,\widetilde{E}_2,\widetilde{E}_3$ are homogeneous generators of $\Cox(\mathcal{T}_{\mathcal{X}(3)})$ with respect to the grading on $\Pic(\mathcal{T}_{\mathcal{X}(3)})$ induced by the usual grading on $\Pic(\mathcal{X}(3))$ via the isomorphism $i^{*}$.

Since $\mathcal{T}_{\mathcal{X}(3)}$ is toric, the Mori chamber decomposition of $\Eff(\mathcal{T}_{\mathcal{X}(3)})$ can be computed by means of the Gelfand–Kapranov–Zelevinsky, GKZ for short, decomposition \cite[Section 2.2.2]{ADHL15}. Let us consider the family of vectors in $\Pic(\mathcal{T}_{\mathcal{X}(3)})$ given by $\mathcal{W} = (\widetilde{D}_1,\widetilde{D}_2,\widetilde{D}_3,\widetilde{E}_1,\widetilde{E}_2,\widetilde{E}_3)$, and let $\Omega(\mathcal{W})$ be the set of all convex polyhedral cones generated by some of the vectors in $\mathcal{W}$. By \cite[Construction 2.2.2.1]{ADHL15} the GKZ chambers of $\Eff(\mathcal{T}_{\mathcal{X}(3)})$ are given by the intersection of all the cones in $\Omega(\mathcal{W})$ containing a fixed divisor $D\in\Eff(\mathcal{T}_{\mathcal{X}(3)})$. Since $\Pic(\mathcal{T}_{\mathcal{X}(3)})$ is $3$-dimensional we may picture the vectors of $\mathcal{W}$ in a $2$-dimensional section. It is straightforward to see that taking all the possible intersections of all the convex cones generated by vectors in $\mathcal{W}$ we get a picture completely analogous to the one in the statement. 

Now, Remark \ref{toric} yields that the wall-and-chamber decomposition in the statement is a possibly trivial refinement of the Mori chamber decomposition of $\Eff(\mathcal{X}(3))$. In particular, $\Mov(\mathcal{X}(3))$ is divided in at most two chambers. On the other hand, by Theorem \ref{theff} we have that $\Nef(\mathcal{X}(3)) = \left\langle D_1,D_2,D_3\right\rangle$, and hence on $\Mov(\mathcal{X}(3))$ the Mori chamber decomposition consists of exactly two chambers as displayed in the statement.
  
Let us analyze the stable base locus decomposition of $\Eff(\mathcal{X}(3))$. Now, note that $\Nef(\mathcal{X}(3)) = \left\langle D_1,D_2,D_3\right\rangle$ yields $\textbf{B}(D) = \emptyset$ for any $D\in \left\langle D_1,D_2,D_3\right\rangle$. Furthermore, since $D_M\sim 2D_1+E_3\sim 2D_3+E_1$ for any divisor $D\in \left\langle D_3,E_3,D_M\right\rangle$ along with $(D_3,E_3]\cup (D_M,E_3]$ we have $\textbf{B}(D)\subseteq E_3$. On the other, since such a divisor is not movable its stable base locus must contain at least an irreducible divisor, and hence $\textbf{B}(D) = E_3$.

The same argument shows that $E_1\cup E_3$ contains the stable base locus of any divisor $D$ contained in the interior of $\left\langle D_M,E_1,E_3\right\rangle$ along with $(E_1,E_3)$. On the other hand, considering the curves described in Proposition \ref{moricone} we see that both $E_1$ and $E_3$ are covered by curves intersecting negatively such a divisor, so $\textbf{B}(D) = E_1\cup E_3$. 

Similarly, we can prove that $\textbf{B}(D) = E_1$ if and only if $D$ lies in the interior of $\left\langle D_1,E_1,D_M\right\rangle$ along with $(D_1,E_1]\cup (D_M,E_1]$, $\textbf{B}(D) = E_2\cup E_3$ if and only if $D$ lies in the interior of $\left\langle D_3,E_2,E_3\right\rangle$ along with $(E_2,E_3)$, $\textbf{B}(D) = E_1\cup E_2$ if and only if $D$ lies in the interior of $\left\langle D_1,E_1,E_2\right\rangle$ along with $(E_1,E_2)$.

The chambers $\left\langle D_1,D_2,E_2\right\rangle$ and $\left\langle D_2,D_3,E_2\right\rangle$ require a more careful analysis. In the notation of Proposition \ref{morNEF} the divisor $D_1+D_2$ induces the morphism $\pi_1\times\pi_2:\mathcal{X}(3)\rightarrow\mathbb{P}^{15}\times\mathbb{P}^{35}$. Similarly, the divisor $D_2+D_3$ induces the morphism $\pi_2\times\pi_3:\mathcal{X}(3)\rightarrow\mathbb{P}^{35}\times\mathbb{P}^{15}$. Therefore, for their exceptional loci we have $\Exc(\pi_1\times\pi_2)=\Exc(\pi_2\times\pi_3) = E_2$. Now, by Nakamaye's theorem \cite[Theorem 10.3.5]{La04II} we get that for a divisor $D$ in the interior of $\left\langle D_1,D_2,D_3,E_2\right\rangle$ along with $(D_1,E_2]\cup (D_3,E_2]$ we have $\textbf{B}(D) = E_2$. 

So far we proved that $\Eff(\mathcal{X}(3))$ is subdivided into eight stable base locus chambers and into at most nine Mori chambers. Therefore, to conclude the computation of the Mori chamber decomposition of $\Eff(\mathcal{X}(3))$ it is enough to prove that the wall joining $D_2$ and $E_2$ divides the stable base locus chamber delimited by $D_1,D_2,D_3$ and $E_2$ into two Mori chambers. This could be done simply by arguing that Mori chambers are convex. Anyway, in the following we will give a more constructive and geometrical proof. 

Consider the blow-up $\mathcal{X}(3)_1$ of $\mathbb{P}^{15}$ along the Segre variety $\mathcal{S}$. Note that, in the notation of Propositions \ref{firstbu}, \ref{Sarkisov} we have the following commutative diagram
  \[
  \begin{tikzpicture}[xscale=2.5,yscale=-1.2]
    \node (A0_1) at (1, 0) {$\mathcal{X}(3)$};
    \node (A1_0) at (0, 1) {$\mathcal{X}(3)_1$};
    \node (A1_2) at (2, 1) {$\mathcal{X}(3)_{1}^{2}$};
    \node (A2_0) at (0, 2) {$\mathbb{P}^{15}$};
    \node (A2_1) at (1, 2) {$W$};
    \node (A2_2) at (2, 2) {$\mathbb{P}^{15*}$};
    \path (A1_0) edge [->]node [auto] {$\scriptstyle{}$} (A2_1);
    \path (A0_1) edge [->,swap]node [auto] {$\scriptstyle{\pi_1\times\pi_2}$} (A1_0);
    \path (A1_2) edge [->]node [auto] {$\scriptstyle{}$} (A2_1);
    \path (A0_1) edge [->]node [auto] {$\scriptstyle{\pi_2\times\pi_3}$} (A1_2);
    \path (A1_0) edge [->]node [auto] {$\scriptstyle{}$} (A2_0);
    \path (A1_0) edge [->,dashed]node [auto] {$\scriptstyle{}$} (A1_2);
    \path (A1_2) edge [->]node [auto] {$\scriptstyle{}$} (A2_2);
  \end{tikzpicture}
  \]
where the rational map $\mathcal{X}(3)_1\dasharrow\mathcal{X}(3)_{1}^{2}$ is induced by the automorphism $Z^{inv}:\mathcal{X}(3)\rightarrow\mathcal{X}(3)$ in Proposition \ref{Sarkisov}. Note that by Proposition \ref{firstbu} $\mathcal{X}(3)_{1}^{2}$ is the only small $\mathbb{Q}$-factorial modification of $\mathbb{X}(3)_1$, and $\mathcal{X}(3)_1\dasharrow\mathcal{X}(3)_{1}^{2}$ is the flop associated to the small contraction induced by $D_2$. Indeed, $D_2$ has zero intersection with the strict transform of a line secant to $\mathcal{S}$, and the strict transform of $\sec_2(\mathcal{S})$ is the exceptional locus of the small contraction $\mathcal{X}(3)_1\rightarrow Y$ induced by $D_2$. So, even though $\mathcal{X}(3)_1$ and $\mathcal{X}(3)_{1}^{2}$ are abstractly isomorphic, crossing the wall generated by $D_2$ we get a non-trivial flop among them, and hence in the Mori chamber decomposition the wall $[D_2,E_2]$, that we could not see in the stable base locus decomposition, appears.

Now, consider $\mathcal{Q}(3)\subset\mathcal{X}(3)$. Propositions \ref{MCXnQn} and \ref{Sarkisov} yield that the Mori chamber and the stable base locus chamber decompositions of $\Eff(\mathcal{Q}(3))$ coincide with the corresponding decompositions of $\Eff(\mathcal{X}(3)$ outside of the movable cone. On the other hand, arguing as we did for (\ref{mov3}) in the proof of Proposition \ref{PropMov} we have that $\Mov(\mathcal{Q}(3)) = \left\langle D_1^{+},D_2^{+},D_3^{+},D_M^{+}\right\rangle$, with $D_M^{+} \sim 6D_1^{+}-3E_1^{+}-2E_2^{+}$, and by Theorem \ref{theff} we have $\Nef(\mathcal{Q}(3)) = \left\langle D_1^{+},D_2^{+},D_3^{+}\right\rangle$. Hence $\Mov(\mathcal{Q}(3))$ is subdivided in at least two Mori chambers.  Finally, by Proposition \ref{MCXnQn} and the first part of the proof we conclude that $\Mov(\mathcal{Q}(3))$ is subdivided in exactly two Mori chambers which are also distinct stable base locus chambers.  
\end{proof}
Finally, for $\mathcal{X}(2,m)$ we have the following result.

\begin{thm}\label{MCD_main2}
The Mori chamber decomposition of $\mathcal{X}(2,m)$ with $m\geq 3$ consists of five chambers described in the following $2$-dimensional section of $\Eff(\mathcal{X}(2,m))$
$$
\begin{tikzpicture}[xscale=0.7,yscale=0.5][line cap=round,line join=round,>=triangle 45,x=1cm,y=1cm]\clip(-7.5,-2.1) rectangle (4.1,4.5);\fill[line width=0pt,color=yqyqyq,fill=yqyqyq,fill opacity=0.56] (0,2) -- (-3.5,0) -- (3.52,0) -- cycle;\draw [line width=0.4pt] (-7,-2)-- (0,2);\draw [line width=0.4pt] (0,4)-- (-7,-2);\draw [line width=0.4pt] (-7,-2)-- (3.52,0);\draw [line width=0.4pt] (3.52,0)-- (0,4);\draw [line width=0.4pt] (0,4)-- (0,2);\draw [line width=0.4pt] (-3.5,0)-- (0,4);\draw [line width=0.4pt] (0,2)-- (3.52,0);\draw [line width=0.4pt] (-3.5,0)-- (3.52,0);\begin{scriptsize}\draw [fill=black] (-7,-2) circle (0.0pt);\draw[color=black] (-7.3,-1.7) node {$E_{1}$};\draw [fill=black] (0,2) circle (0.0pt);\draw[color=black] (0.3,2.2) node {$D_{2}$};\draw [fill=black] (-3.5,0) circle (0.0pt);\draw[color=black] (-3.8,0.2) node {$D_{1}$};\draw [fill=black] (3.52,0) circle (0.0pt);\draw[color=black] (3.8,0.2) node {$D_{3}$};\draw [fill=black] (0,4) circle (0.0pt);\draw[color=black] (0.08422939068100468,4.25) node {$E_{2}$};\end{scriptsize}\end{tikzpicture}
$$
where $\Mov(\mathcal{X}(2,m)) = \Nef(\mathcal{X}(2,m)) = \left\langle D_1,D_2,D_3\right\rangle$. The stable base locus decomposition of $\Eff(\mathcal{X}(2,m))$ consists of four chambers and is obtained by removing the wall joining $D_2$ with $E_2$ in the picture above.
\end{thm}
\begin{proof}
It is enough to argue as in the proof of Theorem \ref{MCD_main}. Note that as in the case of $\mathcal{X}(3)$ the stable base locus of any divisor in the interior of the non convex chamber $\left\langle D_1,D_2,D_3,E_2\right\rangle$ along with $(D_1,E_2]$ and $(D_{3},E_{2}]$ is the divisor $E_2$. On the other hand, since Mori chambers are convex the wall joining $D_2$ with $E_2$ must appear in the Mori chamber decomposition.
\end{proof}

\section{Pseudo-automorphisms}\label{psaut}
A birational map $f:X\dasharrow Y$ between two varieties is a pseudo-isomorphism if there are $\mathcal{U}\subseteq X$, $\mathcal{V}\subset Y$ open subsets whose complementary sets have codimension greater than or equal to two such that $f_{\mathcal{U}}:\mathcal{U}\rightarrow\mathcal{V}$ is an isomorphism. When $X=Y$ a pseudo-isomorphism $f:X\dasharrow X$ is called a pseudo-automorphism of $X$. Essentially, a pseudo-automorphism of $X$ is a birational map $f:X\dasharrow X$ such that both $f$ and $f^{-1}$ do not contract any divisor. We will denote by $\PsAut(X)$ the group of pseudo-automorphisms of $X$.

In this section we will prove that for spaces of complete forms the groups of automorphisms and pseudo-automorphisms coincide. We begin by analyzing the relation between pseudo-automorphisms and automorphisms of varieties of Picard number one, and of Fano varieties of any Picard number. For a general account on automorphisms of Mori dream spaces we refer to \cite[Section 4.2.4]{ADHL15}. 

\begin{Lemma}\label{pautl1}
Let $f:X\dasharrow Y$ be a birational map between normal projective varieties with the same Picard number $\rho(X) = \rho(Y)$. If $f$ does not contract any divisor in $X$ then $f^{-1}:Y\dasharrow X$ does not contract any divisor in $Y$, and hence $X$ and $Y$ are isomorphic in codimension one.

Furthermore, assume that $\Eff(X)$ is finitely generated by $r$ extremal rays, and that there exists a birational morphism $f:X\rightarrow Y$ onto a normal projective variety $Y$. Then there is an exact sequence
$$0\rightarrow \PsAut(Y)\rightarrow\PsAut(X)\rightarrow S_r$$
where $S_r$ is the symmetric group on $r$ elements.
\end{Lemma}
\begin{proof}
Consider a resolution
  \[
  \begin{tikzpicture}[xscale=1.5,yscale=-1.2]
    \node (A0_1) at (1, 0) {$Z$};
    \node (A1_0) at (0, 1) {$X$};
    \node (A1_2) at (2, 1) {$Y$};
    \path (A1_0) edge [->,dashed]node [auto] {$\scriptstyle{f}$} (A1_2);
    \path (A0_1) edge [->]node [auto] {$\scriptstyle{q}$} (A1_2);
    \path (A0_1) edge [->,swap]node [auto] {$\scriptstyle{p}$} (A1_0);
  \end{tikzpicture}
  \]
Assume that there exists a divisor $D\subset Y$ which is contracted by $f^{-1}$. Then the strict transform of $D$ in $Z$ is $p$-exceptional but not $q$-exceptional. Now, $\rho(X) = \rho(Y)$ forces the existence of a $q$-exceptional divisor $D'\subset Z$ which is not $p$-exceptional. Therefore, $f(p(D'))\subset Y$ has codimension greater than or equal to two, a contradiction.  

Now, let us prove the second statement. The group $\PsAut(X)$ acts on the extremal rays of $\Eff(X)$ by permutations, and this gives us the morphism of groups $\PsAut(X)\rightarrow S_r$. Now, assume that $\phi\in \PsAut(X)$ induces the trivial permutation on the set of the extremal rays of $\Eff(X)$, and consider the following commutative diagram  
   \[
  \begin{tikzpicture}[xscale=2.1,yscale=-1.2]
    \node (A0_0) at (0, 0) {$X$};
    \node (A0_1) at (1, 0) {$X$};
    \node (A1_0) at (0, 1) {$Y$};
    \node (A1_1) at (1, 1) {$Y$};
    \path (A0_0) edge [->,dashed]node [auto] {$\scriptstyle{\phi}$} (A0_1);
    \path (A1_0) edge [->,dashed]node [auto] {$\scriptstyle{\overline{\phi}}$} (A1_1);
    \path (A0_1) edge [->]node [auto] {$\scriptstyle{\pi}$} (A1_1);
    \path (A0_0) edge [->,swap]node [auto] {$\scriptstyle{\pi}$} (A1_0);
  \end{tikzpicture}
  \]
where $\overline{\phi}$ is the rational map induced by $\phi$. Assume that there exists a divisor $D\subset Y$ contracted by $\overline{\phi}$. Then the strict transform $\widetilde{D}\subset X$ of $D$ via $\pi$ is not $\pi$-exceptional but since $\phi$ does not contract any divisor $\phi(\widetilde{D})$ must be $\pi$-exceptional. A contraction, since $\phi$ acts trivially on $\Eff(X)$. Finally, the first part of the proof yields that $\overline{\phi}$ is a pseudo-automorphism of $Y$. 
\end{proof}

\begin{Proposition}\label{pautp1}
Let $X$ be a normal projective variety with $\Pic(X)\cong\mathbb{Z}$, and let $f:X\dasharrow X$ be a birational self-map of $X$ not contracting any divisor. Then $f$ is an automorphism. In particular 
$$\PsAut(X)\cong \Aut(X)$$
Furthermore, the isomorphism $\PsAut(X)\cong \Aut(X)$ also holds for smooth Fano varieties of arbitrary Picard number. 
\end{Proposition}
\begin{proof}
By Lemma \ref{pautl1} $f^{-1}:X\dasharrow X$ does not contract any divisor, so $f$ is a pseudo-automorphism of $X$. Now, consider the twisting sheaf $\mathcal{O}_X(1)$ on $X$, and let $i:X\rightarrow \mathbb{P}(X,H^0(X,\mathcal{O}_X(1))^{*}) = \mathbb{P}^N$ be the corresponding embedding. Since $\Pic(X)\cong\mathbb{Z}[\mathcal{O}_X(1)]$ for any $\f\in \PsAut(X)$ we have $f^{*}\mathcal{O}_X(1)\cong \mathcal{O}_X(1)$. 

Now, since $X$ is normal any rational section of $\mathcal{O}_X(1)$ extends to a regular section, and hence $f$ yields an automorphism $\overline{f}$ of $\mathbb{P}^N$ stabilizing $i(X)\cong X$. Therefore, $\overline{f}_{|i(X)}:i(X)\rightarrow i(X)$ extends $f$.

Now, assume that $X$ is a smooth Fano variety. Then Remark \ref{sphMDS} yields that $X$ is a Mori dream space. This means that $X$ admits finitely many small $\mathbb{Q}$-factorial transformations $X\dasharrow X_i$. Assume that $X$ has a pseudo-automorphism $X\dasharrow X$ which is not an automorphism. Since $X\dasharrow X$ is a small $\mathbb{Q}$-factorial transformation then we must have $X_i\cong X$ for some $i$. In particular $X_i$ is Fano, and hence $-K_X\cong -K_{X_i}$ is in the interior of both the maximal chambers of the Mori chamber decomposition of $\Mov(X)$ corresponding to $\Nef(X)$ and $\Nef(X_i)$. A contradiction, since these two chambers intersect at most along a wall.
\end{proof}

\begin{Remark}
Consider the Mori dream space $\mathcal{X}(n)_1$ in Proposition \ref{Sarkisov}. In this case $Z^{inv}:\mathcal{X}(n)\rightarrow\mathcal{X}(n)$ induces a pseudo-automorphism of $\mathcal{X}(n)_1$ which is not an automorphism. Indeed, in this case $\mathcal{X}(n)_{1}^{n-1}$ is abstractly isomorphic to $\mathcal{X}(n)_1$, and 
$$-K_{\mathcal{X}(n)_1}\sim (n^2+2n+1)D_1-(n^2-1)E_1$$ 
Hence, Proposition \ref{firstbu} yields that $\mathcal{X}(n)_1$ is not Fano for any $n\geq 3$. Note that $\mathcal{X}(n)_1$ is Fano for $n\in\{1,2\}$ while $\mathcal{X}(3)_1$ is weak Fano but not Fano.  
\end{Remark}

Now, applying the techniques developed in \cite{BM17}, \cite{MM14}, \cite{MaM17}, \cite{Ma14}, \cite{Ma17}, \cite{FM17} and \cite{FM18} to deal with automorphisms of moduli spaces of curves and stable maps and in \cite{AC17} for moduli of parabolic vector bundles, we will finally compute the pseudo-automorphism group of the spaces of complete forms. We will need the following preliminary result. 

\begin{Lemma}\label{autprod}
Let $X$ be a Cartesian product of projective spaces. Then, modulo an automorphism of $X$, we may write $X\cong (\mathbb{P}^{n_1})^{r_1}\times\dots\times (\mathbb{P}^{n_h})^{r_h}$, where $r_i$ is the number of $n_i$-dimensional projective spaces appearing in the product $X$ for $i = 1,\dots,h$. Then
$$\Aut(X)\cong (S_{r_1}\rtimes PGL(n_1+1)^{r_1})\times\dots \times (S_{r_h}\rtimes PGL(n_h+1)^{r_h})$$
where $S_{r_i}$ is the symmetric group on $r_i$ elements.
\end{Lemma}
\begin{proof}
Since $X$ is a toric variety its Mori cone $\NE(X)$ is generated by the classes of toric invariant curves, that is by the classes of a line in each factor of the product. Therefore, if $\phi\in\Aut(X)$ is an automorphism for any $\pi_{i}:X\rightarrow\mathbb{P}^{n_i}$ we have a commutative diagram of the following form 
  \[
  \begin{tikzpicture}[xscale=2.1,yscale=-1.2]
    \node (A0_0) at (0, 0) {$X$};
    \node (A0_1) at (1, 0) {$X$};
    \node (A1_0) at (0, 1) {$\mathbb{P}^{n_{j_i}}$};
    \node (A1_1) at (1, 1) {$\mathbb{P}^{n_i}$};
    \path (A0_0) edge [->]node [auto] {$\scriptstyle{\phi^{-1}}$} (A0_1);
    \path (A1_0) edge [->]node [auto] {$\scriptstyle{\overline{\phi}}$} (A1_1);
    \path (A0_1) edge [->]node [auto] {$\scriptstyle{\pi_{i}}$} (A1_1);
    \path (A0_0) edge [->,swap]node [auto] {$\scriptstyle{\pi_{j_i}}$} (A1_0);
  \end{tikzpicture}
  \]
where $n_{j_i} = n_i$. Therefore, $\phi\in\Aut(X)$ yields a permutation $\sigma_{\phi}:S_{r_1}\rightarrow S_{r_1}$, given by $\sigma_{\phi}(i) = j_i$, and hence a surjective morphism of groups
$$
\begin{array}{cccc}
\chi: &\Aut(X)& \longrightarrow & S_{r_1}\ltimes PGL(n_1+1)^{r_1}\\
      & \phi & \longmapsto & (\sigma_{\phi},\overline{\phi})
\end{array}
$$
Assume $h = 1$. If $\chi(\phi)$ is the identity then $\phi^{-1}$ restricts to the identity on any fiber of $\pi_{i}$ for $i = 1,\dots, r_1$, and hence $\chi$ is an isomorphism. If $h\geq 2$ and $\chi(\phi)$ is the identity then $\phi^{-1}$ restricts to an automorphism of any fiber of the projection $X\rightarrow (\mathbb{P}^{n_1})^{r_1}$. Since such a fiber is isomorphic to $(\mathbb{P}^{n_2})^{r_2}\times\dots\times (\mathbb{P}^{n_h})^{r_h}$ we conclude by induction on $h$.
\end{proof}

\begin{thm}\label{pseudo-aut}
For the pseudo-automorphism group of $\mathcal{X}(n,m)$ we have
$$
\PsAut(\mathcal{X}(n,m)) \cong \Aut(\mathcal{X}(n,m)) \cong
\left\lbrace\begin{array}{ll}
PGL(n+1)\times PGL(m+1) & \text{if n}< \textit{m}\\ 
(S_2 \ltimes (PGL(n+1)\times PGL(n+1)))\rtimes S_2 & \text{if n} = \textit{m}\geq 2
\end{array}\right.
$$
while $\PsAut(\mathcal{X}(1))\cong \Aut(\mathcal{X}(1))\cong PGL(4)$.

Furthermore, the pseudo-automorphism group of $\mathcal{Q}(n)$ is given by $\PsAut(\mathcal{Q}(n))\cong PGL(n+1)\rtimes S_2$ if $n\geq 2$ while $\PsAut(\mathcal{Q}(1))\cong \Aut(\mathcal{Q}(1))\cong PGL(3)$. 
\end{thm}
\begin{proof}
Since by Corollary \ref{Fano} $\mathcal{X}(n,m)$ is Fano, Proposition \ref{pautp1} yields an isomorphism of groups $\PsAut(\mathcal{X}(n,m))\cong\Aut(\mathcal{X}(n,m))$. Now, let $\phi\in\Aut(\mathcal{X}(n,m))$ be an automorphism. Then $\phi$ must act on the extremal rays of $\Nef(\mathcal{X}(n,m))$. If $n<m$ then Theorem \ref{theff} yields that this action must be trivial since for instance $h^0(\mathcal{X}(n,m), D_i)\neq h^0(\mathcal{X}(n,m), D_j)$ for any pair of generators $D_i\neq D_j$ of $\Nef(\mathcal{X}(n,m))$. 

On the other hand, if $n = m$ then either this action is trivial or it switches $D_i$ with $D_{n+1-i}$ for $i = 1,\dots,n$. We know that the latter is indeed realized by the distinguished automorphism $Z^{inv}:\mathcal{X}(n)\rightarrow\mathcal{X}(n)$ in Proposition \ref{Sarkisov}. 

Therefore, if $n < m$ for any automorphism $\phi$ we have in particular that $\phi^{*}D_1 = D_1$, and hence via the blow-up map $f:\mathcal{X}(n)\rightarrow\mathbb{P}^N$ in Construction \ref{ccc} $\phi$ induces an automorphism $\overline{\phi}$ of $\mathbb{P}^N$ stabilizing the Segre variety $\mathcal{S}\cong\mathbb{P}^n\times\mathbb{P}^m$. To conclude it is enough to observe that since $\mathcal{X}(n,m)$ and $\mathbb{P}^N$ are birational we have $\phi = Id_{\mathcal{X}(n,m)}$ if and only if $\overline{\phi} = Id_{\mathbb{P}^N}$, and that since $\mathcal{S}\subset\mathbb{P}^N$ is non-degenerated the subgroup of $PGL(N+1)$ stabilizing $\mathcal{S}$ is isomorphic to $\Aut(\mathcal{S})$. Finally, Lemma \ref{autprod} yields $\Aut(\mathcal{S})\cong PGL(n+1)\times PGL(m+1)$.

If $n = m\geq 2$ we have a surjective morphism $\Aut(\mathcal{X}(n))\rightarrow S_2$ where $S_2 = \{Id_{\mathcal{X}(n)}, Z^{inv}\}$. Assume that the permutation induced by $\phi\in\Aut(\mathcal{X}(n))$ is trivial. Then as before $\phi$ induces an automorphism of $\mathbb{P}^N$ preserving $\mathcal{S}\cong\mathbb{P}^n\times\mathbb{P}^n$, and in this case Lemma \ref{autprod} yields $\Aut(\mathcal{S})\cong S_2 \ltimes (PGL(n+1)\times PGL(n+1))$. We get the exact sequence 
$$0\rightarrow S_2 \ltimes (PGL(n+1)\times PGL(n+1))\rightarrow\Aut(\mathcal{X}(n))\rightarrow S_2\rightarrow 0$$
where the last morphism has a section. So the sequence split, and since the actions of $S_2 \ltimes (PGL(n+1)\times PGL(n+1))$ and $S_2 = \{Id_{\mathcal{X}(n)},Z^{inv}\}$ on $\mathcal{X}(n)$ do not commute the semi-direct product $\Aut(\mathcal{X}(n))\cong (S_2 \ltimes (PGL(n+1)\times PGL(n+1)))\rtimes S_2$ is not direct. If $n=m=1$ recall that Construction \ref{ccc} yields $\mathcal{X}(1)\cong\mathbb{P}^3$. 

For $\mathcal{Q}(n)$, considering the Veronese variety $\mathcal{V}\subseteq\mathbb{P}^{N_{+}}$ instead of the Segre variety $\mathcal{S}\subseteq\mathbb{P}^N$, a completely analogous argument works.
\end{proof}

\begin{Remark}
It is well known that the Kontsevich moduli space $\overline{M}_{0,0}(\mathbb{P}^2,2)$, parametrizing degree two stable maps from a nodal rational curve to $\mathbb{P}^2$, is isomorphic to the space of complete conics $\mathcal{Q}(2)$ \cite[Section 0.4]{FP97}. Then Theorem \ref{pseudo-aut} says in particular that $\Aut(\overline{M}_{0,0}(\mathbb{P}^2,2))\cong PGL(3)\rtimes S_2$, where $PGL(3)$ acts by motions of the target $\mathbb{P}^2$, and the involution of $S_2$ associates to a conic its dual.
\end{Remark}

\begin{Corollary}\label{coraut}
Let $X$ be an irreducible normal $\mathbb{Q}$-factorial complete variety, and let $\Aut^{o}(X)$ be the connected component of the identity of $\Aut(X)$. 

If $X$ is isomorphic in codimension two to $\mathcal{X}(n,m)$ with $(n,m)\neq (1,1)$, then $\Aut^{o}(X)\cong PGL(n+1)\times PGL(m+1)$. Similarly, if $X$ is isomorphic in codimension two to $\mathcal{Q}(n)$ with $n\geq 2$ then $\Aut^{o}(X)\cong PGL(n+1)$.
\end{Corollary}
\begin{proof}
Note that $X$ is a Mori dream space, since it is a small $\mathbb{Q}$-factorial modification of a Mori dream space by Remark \ref{sphMDS}. Therefore, $\Cox(X)$ is finitely generated. Now, to conclude it is enough to apply \cite[Corollary 4.2.4.2]{ADHL15} together with Theorem \ref{pseudo-aut}.
\end{proof}

\begin{Remark}
As observed in the proof of Theorem \ref{pseudo-aut}, in the cases not covered in Corollary \ref{coraut} the spaces of complete forms are projective spaces. Therefore, any irreducible normal $\mathbb{Q}$-factorial complete variety isomorphic in codimension two to them is a projective space as well.
\end{Remark}

\appendix 
\section{Extremal rays of the movable cones}\label{appendix}
We present Maple scripts computing the extremal rays of the movable cone of the varieties $\mathcal{X}(n,m), \mathcal{X}(n), \mathcal{Q}(n)$. The scripts are based on Theorem \ref{gen} and \cite[Proposition 3.3.2.3]{ADHL15}, and require the Maple packages \href{http://www-home.math.uwo.ca/~mfranz/convex/}{\textit{Convex} 1.2.0} by M. Franz and \href{https://www.maplesoft.com/support/help/Maple/view.aspx?path=LinearAlgebra}{\textit{LinearAlgebra}}.

By \cite[Proposition 3.3.2.3]{ADHL15} we have that if $X$ is an irreducible, normal, complete variety with finitely generated Cox ring then $\Mov(X)$ is the intersection of the cones constructed by taking all but one vectors among the vectors giving the grading of a fixed system of generators of $\Cox(X)$. 

Note that, for generators having at least two sections, it is then enough to consider two copies of the grading vectors. This simple observation greatly reduces the computational complexity of the problem. 

\begin{Script} Movable cone of $\mathcal{X}(n,m)$ for $n<m$.
\begin{footnotesize}
\begin{verbatim}
B := Matrix(n+1,n+1): 
for i to n+1 do for j to n+1 do if i = 1 then B[i,j] := j 
else if i > j then B[i,j] := 0 else B[i,j] := i-j-1 
end if; end if; end do; end do;
E := Matrix(n+1,n): 
for i to n+1 do for j to n do if i = j+1 then E[i,j] := 1 else E[i,j] = 0 
end if; end do; end do;
st := time[real]():
A := Matrix([B,B,E]); C := ColumnDimension(A); R := RowDimension(A); 
for i to C do Deg[i] := [seq(A[j, i], j = 1..R)] end do; 
Degaux := [seq(Deg[j], j = 1..C)]; 
for l to C do Coneaux[l] := poshull(op(subsop(l = NULL, Degaux))) end do; 
for l to C do rays(Coneaux[l]) end do; 
MovCone := intersection(seq(Coneaux[i], i = 1..C)); rays(MovCone);
numelems(rays(MovCone)); time[real]()-st
\end{verbatim}
\end{footnotesize}
\end{Script}
\begin{Script}\label{scr2} Movable cone of $\mathcal{X}(n)$ and $\mathcal{Q}(n)$.
\begin{footnotesize}
\begin{verbatim}
B := Matrix(n,n+1): 
for i to n do for j to n+1 do if i = 1 then B[i,j] := j 
else if i > j then B[i,j] := 0 else B[i,j] := i-j-1 
end if; end if; end do; end do;
E := Matrix(n,n-1): 
for i to n do for j to n-1 do if i = j+1 then E[i,j] := 1 else E[i,j] = 0 
end if; end do; end do;
A := Matrix([B,B[1..n,1..n],E]); C := ColumnDimension(A); R := RowDimension(A);
st := time[real](): 
for i to C do Deg[i] := [seq(A[j,i], j = 1..R)] end do; 
Degaux := [seq(Deg[j], j = 1..C)]; 
for l to C do Coneaux[l] := poshull(op(subsop(l = NULL, Degaux))) end do; 
for l to C do rays(Coneaux[l]) end do; 
MovCone := intersection(seq(Coneaux[i], i = 1..C)); rays(MovCone); 
numelems(rays(MovCone)); time[real]()-st
\end{verbatim}
\end{footnotesize}
\end{Script}
These scripts are reasonably fast. For instance, Script \ref{scr2} computes in $0.045$ seconds the rays of $\Mov(\mathcal{X}(3))$ and $\Mov(\mathcal{Q}(3))$ in (\ref{mov3}), and in $0.055$ seconds the rays of $\Mov(\mathcal{X}(4))$ and $\Mov(\mathcal{Q}(4))$, which with respect to the standard basis of $\Pic(\mathcal{X}(4))$ and $\Pic(\mathcal{Q}(4))$ in Proposition \ref{p1} are the vectors $(4,-3,-2,-1)$, $(3,-2,-1,0)$, $(12,-8,-6,-3)$, $(1,0,0,0)$, $(2,-1,0,0)$, $(8,-4,-3,-2)$, $(16,-11,-6,-4)$, $(9,-4,-3,0)$.

On a regular laptop Script \ref{scr2} computes the $512$ rays of $\Mov(\mathcal{Q}(10))$ in $2.527$ seconds, the $4096$ rays of $\Mov(\mathcal{Q}(13))$ in $447.010$ seconds, and the $8192$ rays of $\Mov(\mathcal{Q}(14))$ in $3207.981$ seconds.
\bibliographystyle{amsalpha}
\bibliography{Biblio}

\end{document}